\documentclass{article}
\usepackage{amsmath,amsfonts,amsthm,amssymb,hyperref}
\usepackage[dvips]{graphics}
\usepackage[dvips,usenames,dvipsnames]{color} 
\hypersetup{colorlinks=true, urlcolor=blue, linkcolor=red, citecolor=green}
\usepackage[dvips]{pict2e}

\newcommand{\Z}{\mathbb{Z}}
\newcommand{\Zd}{\mathbb{Z}^d}

\newcommand{\cG}{\mathcal{G}}

\newcommand{\cR}{\mathcal{R}}

\newcommand{\cW}{\mathcal{W}}

\newcommand{\E}{\mathbb{E}}

\newcommand{\locball}{\mathfrak{B}}
\newcommand{\looper}{\mathcal{L}}

\newcommand{\diam}{\mathrm{diam}}
\newcommand{\wave}{\mathcal{W}}
\newcommand{\unitv}{\mathrm{e}}
\newcommand{\prob}{\mathbb{P}}
\newcommand{\av}{\mathrm{Av}}
\newcommand{\tree}{\mathfrak{T}}
\newcommand{\ST}{\mathcal{T}}
\newcommand{\T}{\mathfrak{T}}

\newcommand{\wsf}{\mathbf{WSF}}
\newcommand{\wsfo}{\mathbf{WSF}_o}
\newcommand{\recur}{\mathcal{R}}
\newcommand{\R}{\mathcal{R}}

\newcommand{\sentm}{\mathbf{sent}}
\newcommand{\getm}{\mathbf{get}}
\newcommand{\ang}{\mathrm{ang}}
\newcommand{\es}{\emptyset}
\newcommand{\eps}{\varepsilon}

\newcommand{\rin}{\mathrm{in}}
\newcommand{\rout}{\mathrm{out}}

\newcommand{\newp}{{\,\prime}}

\newcommand{\hS}{\widehat{S}}

\newcommand{\heta}{\widehat{\eta}}
\newcommand{\hxi}{\widehat{\xi}}
\newcommand{\hsigma}{\widehat{\sigma}}
\newcommand{\osigma}{\overline{\sigma}}
\newcommand{\oxi}{\overline{\xi}}
\newcommand{\Burnt}{\mathrm{Bt}}
\newcommand{\Unburnt}{\mathrm{Ut}}
\newcommand{\tBurnt}{\widetilde{\Burnt}}
\newcommand{\tUnburnt}{\widetilde{\Unburnt}}

\newcommand{\cK}{\mathcal{K}}
\newcommand{\cT}{\mathcal{T}}

\newcommand{\frC}{\mathfrak{C}}
\newcommand{\frF}{\mathfrak{F}}
\newcommand{\frT}{\mathfrak{T}}
\newcommand{\tfrT}{\widetilde{\mathfrak{T}}}

\newcommand{\past}{\mathrm{past}}
\newcommand{\future}{\mathrm{future}}

\newcommand{\dist}{\mathrm{dist}}

\newcommand{\SLE}{\mathrm{SLE}}
\newcommand{\Es}{\mathrm{Es}}
\newcommand\tst{\rule{0pt}{4.5ex}}
\newcommand\bst{\rule[-3.4ex]{0pt}{0pt}}
\theoremstyle{plain}
\newtheorem{thm}{Theorem}[section]

\newtheorem{lem}[thm]{Lemma}
\newtheorem{prop}[thm]{Proposition}
\newtheorem{defin}{Definition}
\newtheorem{cor}[thm]{Corollary}

\numberwithin{equation}{section}

\theoremstyle{remark}
\newtheorem{rem}{Remark}

\newenvironment{acknow}{{\bf Acknowledgements:}}{}

\newcommand{\eqn}[2]{\begin{equation}\label{#1}#2\end{equation}}
\newcommand{\eqnst}[1]{\begin{equation*}#1\end{equation*}}
\newcommand{\eqnspl}[2]{\begin{equation}\begin{split}\label{#1}%
   #2\end{split}\end{equation}}
\newcommand{\eqnsplst}[1]{\begin{equation*}\begin{split}%
   #1\end{split}\end{equation*}}

\title{Inequalities for critical exponents in $d$-dimensional sandpiles}

\author{Sandeep Bhupatiraju \thanks{The research of S.B. was supported by NSF grants DMS-1007244 and DMS-1612363.}\\ \small Indiana University \and 
Jack Hanson\thanks{The research of J.H. was supported by an AMS-Simons 
travel grant and NSF grant DMS-1612921.} \\ \small City College of NY \and 
Antal A. J\'arai \\ \small University of Bath}

\begin{document}

\maketitle
\begin{abstract}
Consider the Abelian sandpile measure on $\Zd$, $d \ge 2$, obtained as the
$L \to \infty$ limit of the stationary distribution of the sandpile on 
$[-L,L]^d \cap \Zd$. When adding a grain of sand at the origin, some 
region, called the avalanche cluster, topples during stabilization. We prove bounds on 
the behaviour of various avalanche characteristics: 
the probability that a given vertex topples, the radius of the toppled region, 
and the number of vertices toppled. Our results yield rigorous 
inequalities for the relevant critical exponents. In $d = 2,$ we show 
that for any $1 \le k < \infty$, the last $k$ waves of the avalanche
have an infinite volume limit, satisfying a power law upper bound
on the tail of the radius distribution. 
\end{abstract}

\tableofcontents


\section{Introduction}
\label{sec:intro}

The Abelian sandpile model is a particle system defined in terms of simple local redistribution events, called topplings, which give rise to non-local dynamical events called avalanches. The model has received a lot of attention in the theoretical physics literature (see \cite{D06,BTW88}) due to its remarkable self-organized critical state, conjectured to be characterized by power-law behavior of various quantities related to avalanches. Starting with the seminal work of Dhar, much mathematical progress has been made toward understanding this self-organized critical state. The surveys \cite{Redig} and \cite{J14} collect some of this. However, establishing power law behavior for many fundamental avalanche characteristics on $\mathbb{Z}^d$ appears difficult in general. The purpose of this paper is to establish new rigorous inequalities, which in high dimensions come close to identifying the correct tail behavior, for these quantities.

Given a finite set $V \subset \Z^d$, a \textbf{sandpile} on $V$ is a collection of indistinguishable particles, given by a map $\eta : V \to \{ 0, 1, \dots \}$. We say that 
$\eta$ is \textbf{stable}, if $\eta(x) < 2d$ for all $x \in V$. If $\eta$ is unstable at $x$,
that is $\eta(x) \ge 2d$, we say that $x$ is \textbf{allowed to topple}. On toppling, $x$ sends
one particle along each edge incident with it, resulting in the new sandpile
\eqn{e:toppling-def}
{ \eta'(y) 
  = \begin{cases} 
    \eta(y) - 2d  &\text{if  $y = x$;}\\
    \eta(y) + 1   &\text{if $y \sim x$, $y \in V$;} \\
    \eta(y)       &\text{otherwise.}
    \end{cases} }
Particles sent to vertices in $\Zd \setminus V$ are lost. It is well-known 
\cite{D90} that given any sandpile $\eta$, carrying out all possible topplings in any sequence, 
results in a uniquely defined stable sandpile $\eta^\circ$. The \textbf{sandpile Markov chain} on $V$
is the Markov chain with state space equal to the set of stable sandpiles on $V$, where at each time
step a particle is added at a uniformly chosen vertex of $V$, and the sandpile is stabilized,
if necessary. 
The unique stationary distribution \cite{D90} is denoted $\nu_V$.

We will be interested in sandpiles on $\Zd$, where ``stable'' and ``toppling'' are defined 
the same way as for finite $V$. Athreya and J\'arai \cite{AJ04} proved that 
if $V(L) = [-L,L]^d \cap \Z^d$, $d \ge 2$, then $\nu_L := \nu_{V(L)}$ 
converges weakly, as $L \to \infty$, to a limit measure $\nu$, called 
the \textbf{sandpile measure}. Let $\eta : \Zd \to \{ 0, \dots, 2d-1 \}$ 
be a sample configuration from the measure $\nu$. Let us add a particle to $\eta$ at the origin $o$, and
let $\av = \av(\eta)$ denote the set of vertices that topple, called the \textbf{avalanche cluster}. 
The set of all topplings, with multiplicity, is called the \textbf{avalanche}. 
In this paper we study various characteristics of avalanches.

The concept of waves, introduced by Ivashkevich, Ktitarev and Priezzhev \cite{IKP94} 
in the context of finite graphs, will play an important role. Waves provide a decomposition of an avalanche into smaller sets of topplings: 
$\wave_{L,1}, \dots, \wave_{L,N} \subset V(L)$;
see Section \ref{sec:def} for precise definitions. 
In each wave, every vertex topples at most once, 
and the union of the waves includes all topplings 
of the avalanche with the correct multiplicity. 
The paper \cite{IKP94} analyzed the last wave 
$\wave_{N,L}$ in particular when $d = 2$. 

Our first set of results concern the probability that a given vertex topples.
Based on an analysis of the last wave, we prove the following 
rigorous lower bounds on the \textbf{toppling probability}.

\begin{thm}
\label{thm:topple}
\ \\
(i) Let $d = 2$. Then
\[ \nu(z \in \av) \geq |z|^{-3/4+o(1)}, \quad \text{as $| z | \to \infty$.} \]
(ii) Let $d = 3$. There exist constants $\zeta < 1/2$ and $c > 0$ such that
\[ \nu(z \in \av) \geq c \, |z|^{-2\zeta -1}, \quad \forall z \in \Z^3. \]
(iii) Let $d = 4$. There exists a constant $c>0$ such that
\[ \nu(z \in \av) \geq c \, |z|^{-2} \left(\log |z| \right)^{-1/3}, \quad \forall z \in \Z^4. \]
\end{thm}

In Theorem \ref{thm:topple} (ii), $\zeta$ can be taken to be any value such that a random walk in $\Z^3$ 
of length $n$ does not hit the loop-erasure of an independent random walk of 
length $n$ with probability $\ge c' n^{-\zeta}$ for some $c' > 0$. We prefer to write the bound (ii) in this form to emphasize the dependence on this exponent, whose value is of interest in the theory of loop-erased walks. The exponent $\zeta$ is known to satisfy the bound $\zeta < 1/2$; see \cite[Sections 10.3 and 11.5]{Lawlim}.

The rigorous upper bound $\nu(z \in \av) \le C |z|^{2-d}$, for some $C = C(d)$, 
follows from Dhar's formula (see \cite[Eqn.~(3.5)]{JR08}). In dimensions $d \ge 5$, J\'arai, Redig and Saada \cite[Section 6.2]{JRS15} 
proved that $\nu ( z \in \av ) \ge c |z|^{2-d}$, for some $c = c(d)$, also based on an analysis
of the last wave. \cite{JRS15} introduced the critical exponent $\theta$ to quantify the departure from Dhar's formula, assuming 
that $\nu ( z \in \av ) \approx |z|^{2-d-\theta}$, as $|z| \to \infty$. 
This means that $\theta = 0$ when $d \ge 5$. Our Theorem \ref{thm:topple} 
shows that if $\theta$ exists in the sense that $\log(\nu(z \in \av)) / \log(|z|)$ exists as $|z| \rightarrow \infty$ (``logarithmic equivalence''), then
\eqnst
{ 0 
  \le \theta 
  \ \ \begin{cases}
  \le 3/4                    & \text{when $d = 2$,} \\
  < 1                      & \text{when $d = 3$,} \\
  = 0                      & \text{when $d = 4$.}
  \end{cases} }
In particular, Theorem \ref{thm:topple}(iii) establishes that $\theta = 0$ when
$d = 4$, with at most a logarithmic correction.

The reason behind the fact that $\theta = 0$ for $d \geq 5$ is that in these dimensions loop-erased walk and independent simple random walk do not intersect with positive probability. The difference in behavior when $d \geq 5$ also shows up in our other results, and $d = 4$ is expected to be the upper critical dimension of the model, in the sense that critical exponents are no longer expected to depend on dimension when $d \geq 5$ \cite{P00}.  We expect that $\theta$ is positive in dimensions two and three, in analogy with other statistical physics models below the upper critical dimension. However, it seems difficult to get rigorous upper bounds improving on Dhar's formula, since any such bound would have to control all waves of the avalanche. For the last wave, we have a precise characterization in terms of loop-erased walk; however, we lack a convenient description of the joint distribution of all waves of the avalanche. For similar reasons, we do not expect the bounds coming from last waves in low dimensions to be tight.

Our next set of results concern the radius of the toppled region. Let
$R = R(\eta) =  \sup \{ |z| : z \in \av(\eta) \}$ be the \textbf{radius} of the avalanche.
As we explain below, some of the following inequalities are easy 
consequences of Theorem \ref{thm:topple}, while some others follow from known
results on uniform spanning forests of $\Z^d$. 

\begin{thm}
\label{thm:radius}
\ \\
(i) Let $d = 2$. Then,
\eqnsplst
{ r^{-3/4 + o(1)}
  &\le \nu ( R \ge r ), \quad \text{as $r \to \infty$.} }
(ii) Let $d = 3$. There are constants $c > 0$ and $C$ such that with $\zeta$ as in
Theorem \ref{thm:topple} we have
\eqnst
{ c r^{-(2\zeta+1)} 
  \le \nu ( R \ge r )
  \le C r^{-1/6}, \quad \forall r \ge 1. }
(iii) Let $d = 4$.  Then there exist constants $c > 0$ and $C$ such that
\eqnst
{c r^{-2} \left(\log r \right)^{-1/3}  
  \le \nu ( R \ge r )
  \le C r^{-1/4}, \quad \forall r \ge 1. }
(iv) Let $d \ge 5$. There is a constant $c = c(d) > 0$ such that 
\eqnst
{ c r^{-2}
  \le \nu ( R \ge r )
  \le r^{-2} ( \log r )^{3 + o(1)}, \quad \forall r \ge 1. }
\end{thm}

The lower bounds of Theorem \ref{thm:radius} in dimensions $2,\,3,\,$ and $4$ follow from taking $z = r e_1$ in Theorem \ref{thm:topple}, where $e_1 = (1,0,\dots,0) \in \Z^d$. In dimensions $d \ge 3$, upper bounds can be
derived from results of Lyons, Morris and Schramm \cite{LMS}. They analyzed,
using the ``conductance martingale'' of Morris \cite{Morris},
the wired uniform spanning forest measure $\wsf$ on transient graphs,
including $\Z^d$ for $d \ge 3$, as well as a related measure $\wsfo$, obtained
by ``wiring $o$ to infinity''. See the book \cite{LPbook} for detailed background
on wired spanning forests. 
Let $\tree_o$ denote the component of $o$ under $\wsfo$. The proof of
\cite[Theorem 4.1]{LMS} shows that for $d\ge 3,$
\eqn{e:wsfo-rad-bnd}
{ \wsfo ( \diam(\tree_o) > r ) 
  \le C(d) r^{-\frac{1}{2}+\frac{1}{d}}. }
The measure $\wsfo$ can be related to
waves in sandpiles; in particular, this was used by 
J\'arai and Redig \cite{JR08} to show that 
when $d \ge 3$, avalanches are finite $\nu$-a.s.
We derive the upper bounds in Theorem \ref{thm:radius}(ii)--(iii) 
from \eqref{e:wsfo-rad-bnd}.

Above the critical dimension, $d \ge 5$, Priezzhev \cite{P00} gave heuristic
arguments for the mean-field behaviour $\nu ( R \ge r ) \approx r^{-2}$.
Both the lower bound $\nu(R \geq r) \geq \nu(r e_1 \in \av)$ and  the upper bound of \eqref{e:wsfo-rad-bnd} can be sharpened to 
establish this rigorously, in the sense of logarithmic equivalence. 
On the other hand, Theorem \ref{thm:radius}(ii)--(iii) establishes that, if a critical 
exponent $\alpha$ satisfying $\nu(R > r) \approx r^{-\alpha}$ governs 
the tail of $R$ in low dimensions, then this $\alpha$ is different 
from the mean-field value $2$.

We deduce the lower and upper bounds in Theorem \ref{thm:radius}(iv)
from very general mass transport arguments, stated in Theorem \ref{thm:rad-gen} 
below; see \cite[Chapter 8]{LPbook} for background on mass transport. While the main focus of this paper is sandpile models on $\Z^d$, we believe this result may be useful on other graphs and for other models. The proof is in Section \ref{ssec:rad-past-lb}, and is independent of 
the rest of the paper.
Let $G=(V,E)$ be a graph and let $\Gamma\subset Aut(G)$ be a transitive subgroup of the group of automorphisms of $G$, under the topology of pointwise convergence. It is well known that every closed subgroup of $Aut(G)$ has a Borel measure which is invariant under the left multiplication by $\gamma \in \Gamma.$ The group $\Gamma$ is called unimodular if this measure is also invariant under right multiplication. In addition, we call the graph G unimodular if $Aut(G)$ has some unimodular transitive closed subgroup. In this setting the mass transport principle states that for $o\in V(G)$ and a non-negative function $f:V\times V\rightarrow [0,\infty]$, which is invariant under the diagonal action of $\Gamma$, we have  $\sum_{x \in V} f(o, x) = \sum_{x\in V} f(x,o)$.    Let $d$ be a $\Gamma$-invariant metric 
on $V$, and write $\diam(A;x) = \sup \{ d(v,x) : v \in A \}$, and let 
$\diam(A) = \diam(A;o)$. Write $D_x(r) = \{ y \in V : d(y,x) \le r \}$.
We say that an infinite tree $T$ has \textbf{one end}, if any two infinite self-avoiding paths in $T$ have a finite symmetric difference. Given $x \in T$, we denote by $\past_x$ the set of vertices $y \in T$ such that the unique infinite self-avoiding path in $T$ starting at $y$ contains $x$. 
By a \textbf{percolation} on $(V,E)$, we mean a probability 
measure on subgraphs of $(V,E)$. Given a percolation, we write $\frC_x$ 
for the connected component of $x$. When the percolation is supported
on spanning forests, we write $\frT_x$ for $\frC_x$.

\begin{thm}
\label{thm:rad-gen} 
Let $(V,E)$ be a graph with a transitive unimodular group of automorphisms $\Gamma$, and let $o \in V$ be a fixed vertex. Let $\mu$ be a $\Gamma$-invariant percolation on $(V,E)$. \\
(i) If $\mu$ is supported on spanning forests with one-ended components, then
\eqnst
{ \mu \big( \diam(\past_o) > r \big)
  \ge \sum_{x \in V : r < d(x,o) \le 2r} \frac{\mu(o \in \past_x)^2}{\E_\mu 
     \Big[ \big| \frT_o \cap D_o(4r) \big| \, \mathbf{1}_{o \in \past_x} \Big]}. }
(ii) We have
\eqnst
{ \mu \big( \diam(\frC_o) > 4r \big)
  = \sum_{x \in V : r < d(x,o) \le 4r} \mu ( o \in \frC_x ) \,
     \E_\mu \bigg[ \frac{\mathbf{1}_{\diam(\frC_x;x) > 4r}}{| \frC_x \cap D_x(4r) 
     \setminus D_x(r) |} \,\bigg| o \in \frC_x 
     \bigg]. }
(iii) Suppose that $\wsfo ( |\frT_o| < \infty ) = 1$. Then
\eqnsplst
{ &\wsfo \big( \diam(\frT_o) > 4r \big) \\
  &\qquad = \sum_{x \in V : r < d(x,o) \le 4r} \wsf_x ( o \in \frT_x ) \,
     \E_{\wsf_x} \bigg[ \frac{\mathbf{1}_{\diam(\frT_x;x) > 4r}}{| \frT_x \cap D_x(4r) 
     \setminus D_x(r) |} \,\bigg| o \in \frT_x 
     \bigg]. }
\end{thm}

Regarding the upper bound in Theorem \ref{thm:radius} (iv), Lyons, Morris and Schramm state the result
\eqn{e:wsf-rad-bnd}
{ \wsf ( \diam(\past_o) > r ) 
  \le r^{-2} (\log r)^{O(1)}; } 
see \cite[page 1710]{LMS}. However, since a proof of \eqref{e:wsf-rad-bnd} 
is not included in \cite{LMS}, and we need a sharpening of \eqref{e:wsfo-rad-bnd} 
for our results, we deduce a diameter estimate for $\frT_o$ under
$\wsfo$ from Theorem \ref{thm:rad-gen}(iii). (This implies \eqref{e:wsf-rad-bnd} 
due to a stochastic comparison; see \cite[Lemma 3.2]{LMS}). 
In order to deal with the fact that $\wsfo$ is not translation invariant, 
we restrict attention to $\frT_o$, which is unimodular; see 
Section \ref{ssec:rad-past-lb}.

We do not have an upper bound on $\nu ( R \ge r )$ in $d = 2$, and it is 
an open problem whether $\nu ( R < \infty ) = 1$. It follows from 
Theorem \ref{thm:radius}(i) that $\E_\nu R = \infty$, when $d = 2$. 
It may be of independent interest that a short proof of the weaker
statement, that $\E_{\nu_L} R$ diverges, can be 
given without reference to spanning trees or the burning bijection 
described in Section \ref{ssec:trees}. We state this as a 
separate result.

\begin{prop}
\label{thm:simple_rad}
If $d = 2$, then $\lim_{L \to \infty} \mathbb{E}_{\nu_L} R = \infty$.
\end{prop}

Our last set of results concern the number of topplings in the avalanche. 
Let $S$ denote the total number of topplings in the avalanche (that is, 
elements of $\av$ are counted with multiplicity). Recall that 
$\wave_{N,L}$ denotes the last wave in the finite graph $V_L$.
Based on the fractal dimension of loop-erased walk and scaling assumptions, 
Ivashkevich, Ktitarev and Priezzhev \cite{IKP94} derived the exponent
$\nu_{L} ( | \wave_{N,L} | \ge t ) \approx t^{-3/8}$, in the 
limit $L \to \infty$. We prove a rigorous lower bound with the
same exponent, which we also extend to higher dimensions. 
Above the critical dimension, $d \ge 5$, we also have an upper bound on the 
total number of topplings with an exponent which is independent of $d$. 
The upper bounds of Theorem \ref{thm:radius} on the radius in 
$d = 3, 4$ provide upper bounds on the size of the avalanche cluster.

Since the size of the avalanche cluster could be measured in two different ways -- namely, via $|\av|$ and via $S$ -- there are in principle two different possible critical exponents $\tau_{S'}$ and $\tau_{S}$ given (if they exist) by $\nu(|\av| > t) \approx t^{-\tau_{S'}}$, $\nu(S > t) \approx t^{-\tau_{S}}$. As in the preceding cases, our theorems give corresponding bounds on the possible values of $\tau_{S}, \, \tau_{S'}$. These bounds, as well as the best current bounds on the exponents $\theta$ and $\alpha$ described above, are summarized in Table \ref{tab:exp}.

\begin{thm}
\label{thm:size}
\ \\
(i) Let $d = 2$. Then 
\[ t^{-3/8 + o(1)}
   \le \nu ( |\av| \ge t ), \quad \text{as $t \to \infty$.} \]
(ii) Let $d = 3$. With $\zeta$ as in Theorem \ref{thm:topple}, and for some constants 
$C$ and $c > 0$, we have
\[ c t^{-(2\zeta+1)/3}
   \le \nu ( |\av| \ge t )
   \le C t^{-1/18}, \quad \forall t \ge 1. \]
Moreover, $\nu ( S \ge t ) \le C t^{-1/19}$, $\forall t \ge 1$.\\
(iii) Let $d = 4$. There exists $C$ and $c > 0$ such that 
\[ c t^{-1/2} (\log t)^{-5/6}
   \le \nu ( |\av| \ge t )
   \le C t^{-1/16}, \quad \forall t\geq 1.\]
Moreover, $\nu (S \ge t ) \le C t^{-1/17}$, $\forall t \ge 1$.\\
(iv) Let $d \ge 5$. There exist $c = c(d) > 0$ such that 
\eqnsplst
{ c t^{-1/2}
  \le \nu ( |\av| \ge t )
  \le \nu ( S \ge t )
  \le t^{-2/5 + o(1)}, \quad \forall t \ge 1. }
\end{thm}

We establish the lower bounds in dimensions $d = 2, 3$ by showing that 
once a vertex at distance $t^{1/d}$ from $o$ is in the last wave,
at least $c t$ other vertices in its neighbourhood will also be
in the last wave. In $d = 4$, $c t$ is replaced by $c t / \log t$.
Given this, parts (i)--(iii) of Theorem \ref{thm:size} can be deduced from 
parts (i)--(iii) of Theorem \ref{thm:topple}. For the lower bound in $d \ge 5$, in Theorem \ref{thm:size}(iv), 
we use the following analogue of Theorem \ref{thm:rad-gen}(i). 
Write 
\eqnst
{ \frT_o(r) 
  = \frT_o \cap D_o(r) \quad \! \! \! \text{and}  \quad \! \! \!
  \tfrT_o(r) 
  = \{ x \in \frT_o(r) : \text{the path from $o$ to $x$ in $\frT_o$ stays 
    inside $D_o(r)$} \}. }

\begin{thm}
\label{thm:size-gen}
Let $(V,E)$ be a graph with a transitive unimodular group of automorphisms $\Gamma$.
Let $\mu$ be a $\Gamma$-invariant percolation on $(V,E)$ supported on 
spanning forests with one-ended components. For all $t, r \ge 1$ we have
\eqnst
{ \mu \big( | \past_o | > t \big)
  \ge \sum_{x : r < d(x,o) \le (3/2)r} \frac{\mu \Big( o \in \past_x,\, 
    |\tfrT_o(r/2)| > t \Big)^2}{\mathbb{E} \Big[ |\frT_o(4r)| \, 
    \mathbf{1}_{\{ o \in \past_x \}} \Big]}. }
\end{thm}

The proof of Theorem \ref{thm:size-gen} is in Section \ref{ssec:size-gen-lb}
and does not rely on the rest of the paper.



We believe, as it has been argued by Priezzhev \cite{P00}, that the 
exponent $1/2$ is sharp in Theorem \ref{thm:size}(iv). 
But it seems challenging to establish a matching upper bound for $S$ or $\av$. 
The main difficulty lies in extracting useful information on the 
dependence between the waves from the bijection with $\wsfo$. 
Instead of such an approach, we control the number of waves using 
our upper bound on the radius in Theorem \ref{thm:radius}(iv), which allows us to use a union bound 
instead of estimating the dependence. This leads to the upper bounds 
in Theorem \ref{thm:size}(ii)--(iv).

\bgroup
\def\arraystretch{2.2}
\begin{table}
\label{tab:exp}
\centering
\begin{tabular}{ c| | c |c  |c | c}
& \rule{0pt}{2.7ex}\rule[-1ex]{0pt}{0pt} \parbox{2cm}{\textbf{toppling probability}} & 
\parbox{2cm}{\textbf{radius \\ \ }} &
\parbox{2cm}{\textbf{avalanche cluster size}} & 
\parbox{2cm}{\textbf{avalanche size}} \\
\hline
& \rule{0pt}{2.7ex}\rule[-1ex]{0pt}{0pt} $\boldsymbol{\nu (x \in \mathrm{Av} ) \approx |x|^{2-d-\theta}}$ & 
$\boldsymbol{\nu ( R > r ) \approx r^{-\alpha}}$ &
$\boldsymbol{\nu ( |\mathrm{Av}| > t ) \approx t^{-\tau_{S'}}}$ & 
$\boldsymbol{\nu ( S > t ) \approx t^{-\tau_{S}}}$ \\
\hline
\hline
$\boldsymbol{d=2}$&	$[0,\  3/4]$	& $[0,\  3/4]$ & $[0,\ 3/8]$  &$[0,\ 3/8]$\\
\hline
$\boldsymbol{d=3}$& $[0,\  1)$ & $[1/6,\ 2)$  &$[1/18, \ 2/3)$  & $[1/19,\ 2/3)$\\
\hline
$\boldsymbol{d=4}$& $= 0$& $[1/4, \ 2]$ &$[1/16,\ 1/2]$ & $[1/17,\ 1/2]$\\
\hline
$\boldsymbol{d\geq 5}$& $= 0$ & $= 2$ &$[2/5, \ 1/2]$ & $[2/5,\ 1/2]$ \\
\hline
\end{tabular}
\caption{The best known bounds on the critical exponents introduced in this introduction. Bounds are expressed in interval form: e.g., $1/6 \leq \alpha < 2$ when $d = 3$.}
\end{table}
\egroup

All our results have analogues in large finite $V(L)$, or indeed are
derived therefrom. Passing to the limit of $\Zd$ is not too difficult 
when $d \ge 3$, due to the result of J\'arai and Redig \cite[Theorem 3.11]{JR08}
showing that $\nu (|S| < \infty) = 1$. When $d = 2$, this is not known.
We bypass this problem with a more technical argument, that we believe 
is of interest in its own right. We show that for any $1 \le k < \infty$,
the \textbf{last $k$-waves} (when they exist) have a finite limit
as $V(L) \uparrow \Zd$. Recall that for $\eta \in \recur_L$ the waves 
occurring during the stabilization of $\eta + \mathbf{1}_o$ are denoted
$\wave_{L,1}, \dots, \wave_{L,N}$. Waves can also be defined on $\Zd$, 
denoted $\wave_1, \wave_2, \dots$; see Section \ref{ssec:waves-def}. 
On $\Zd$, the number of waves $N$ may take the value infinity.

Let $\eta_{N-k+1}$ denote a random configuration on $V(L)$ with law $\nu_L ( \cdot \,|\, N \ge k )$. Given a configuration $\eta$ on $V(L)$ such that $\eta(o) =2d$, let $\wave(\eta)$ denote the set of sites that can be toppled with every site toppling at most once. We extend this also to configurations $\xi$ on $\Z^d$ such that $\xi(o) \geq 2d$.
\begin{thm}
\label{thm:last-k-limit}
Assume $d = 2$.\\
(i) For all $k \ge 1$, the law of $\eta_{N-k+1}$ converges weakly to the law of a random configuration $\xi_k$ in $\Z^2$.
Let $\rho_k$ denote the law of $\xi_k$. The law of $\wave_{N-k+1} = \wave(\eta_{N-k+1})$ converges to the law of $\wave^*_k := \wave(\xi_k)$, that is a.s.~finite under $\rho_k$.\\
(ii) For all $k \ge 0$ we have $\lim_{L \to \infty} \nu_L ( N = k ) = \nu (N = k)$.\\
(iii) For every $k \ge 1$, the joint law of $\wave_{L,1}, \dots, \wave_{L,k}$ under 
$\nu_L ( \cdot \,|\, N = k )$, converges weakly, as $L \to \infty$, to the law of $\wave_1, \dots, \wave_k$ under $\nu ( \cdot \,|\, N = k )$, and under this conditioning we have $\nu$-a.s.~$|\wave_\ell| < \infty$, $\ell = 1, \dots, k$.
\end{thm}

See Section \ref{sec:last-k-waves} for more detailed statements. Our argument in fact gives a power law upper bound on the radii of the last $k$ waves, and leads to the following extension of Theorem \ref{thm:radius}(i). Let $R_k = \sup \{ |x| : x \in \wave^*_k \}$.

\begin{thm}
\label{thm:2d-upper}
Assume $d = 2$. \\
(i) There are constants $\alpha_1 > \alpha_2 > \dots > 0$
and $C_1, C_2, \dots$ such that 
\eqnst
{ \rho_k ( R_k > r )
  \le C_k r^{-\alpha_k}, \quad \forall r \ge 1,\, \forall k \ge 1. }
(ii) We have
\eqnst
{ \nu ( R \ge r,\, N \le k )
  \le C_k r^{-\alpha_k}, \quad \forall r \ge 1. }
\end{thm}

We will also use Theorem \ref{thm:last-k-limit} to prove the following 
theorem. 

\begin{thm}
\label{thm:exp-wave-infinite}
Suppose $d = 2$. Then $\E_\nu N = \infty$.
\end{thm}

This is a strengthening of the statement $\E_\nu R = \infty$, due to a 
simple comparison proved in Lemma \ref{lem:ptwise}, and therefore also
of Proposition \ref{thm:simple_rad}.

\medbreak

\textbf{Organization of the paper.} In Section \ref{sec:def}, we give definitions and background on sandpiles, spanning trees, and random walks; we also prove Proposition \ref{thm:simple_rad}. In Section \ref{ssec:waves-def}, we prove Theorem  \ref{thm:topple} modulo a technical argument required for the two-dimensional case, which we defer to Section \ref{sec:last-k-waves}. 

Sections \ref{sec:radius-low}, \ref{sec:last-k-waves}, and \ref{sec:highd_rad} are devoted to the various radius bounds above. In Section \ref{sec:radius-low}, we prove Theorem \ref{thm:radius} (i) -- (iii). Section \ref{sec:last-k-waves} contains additional arguments for the two-dimensional case; here we prove Theorems \ref{thm:last-k-limit}, \ref{thm:2d-upper}, and \ref{thm:exp-wave-infinite}. In Section \ref{sec:highd_rad}, we complete the proof of Theorem \ref{thm:radius} by proving the high-dimensional bounds (Theorem \ref{thm:radius} (iv) and Theorem \ref{thm:rad-gen} ).

Section \ref{sec:size} contains the proofs of the size bounds above: Theorems \ref{thm:size} and \ref{thm:size-gen}.

\medbreak

\textbf{A note on constants.} All our constants will be positive and finite,
and they may depend on the dimension $d$. Other dependence will always be
indicated. Constants denoted $C$ and $c$ may change from line to line; 
those with index (such as $c_1$) stay the same within the same proof.

\section{Definitions and background}
\label{sec:def}

In this section, we provide definitions and collect useful facts 
about the basic objects we use: toppling numbers, waves, spanning trees, bijections and Wilson's algorithm.

\subsection{Graphs and sandpiles}
\label{ssec:graphs}

We will work with finite connected graphs of the form $H = (U \cup \{ s \}, F)$, 
where $s$ is a distinguished vertex, called the \textbf{sink}. We allow
multiple edges, so in general $H$ is a multigraph, but we exclude loop-edges.
If $U = V$ is a finite subset of $\mathbb{Z}^d$, we let 
$G_V = (V \cup \{s\},E)$ denote the \textbf{wired subgraph} 
induced by $V$--i.e., where all vertices in $\Zd \setminus V$ 
are identified to the single vertex $s$, and loops at $s$ are removed.
Of prime importance will be the \textbf{standard exhaustion}  
$V(L) = [-L,L]^d \cap \Z^d$, $L \ge 1$; the wired subgraph induced 
by $V(L)$ will be denoted $G_L$. In general, a subscript $L$ will be shorthand for 
subscript $V(L)$. We write $x \sim y$ to denote that vertices $x$ and $y$ of a graph (understood from context) are neighbours.
Given a graph $H = (U \cup \{ s \}, F)$, we write $\deg_H(x)$ for the degree of 
the vertex $x \in U \cup \{ s \}$ in $H$. When $V \subset \Zd$, we write
$\deg_V(x)$ for the degree of vertex $x \in V$ in the subgraph of
$\Zd$ induced by $V$.

Given $H = (U \cup \{ s \}, F)$, the discrete Dirichlet Laplacian $\Delta_H$
is given by
\eqn{e:Delta_H}
{ \Delta_H(x,y)
  = \begin{cases}
    \deg_H(x) & \text{if $x = y$;} \\
    -a_{xy} & \text{if $x \not= y$;} 
    \end{cases} \qquad x, y \in U, }
where $a_{xy}$ equals the number of edges connecting $x$ and $y$.
In particular, when $V$ is a finite subset of $\Zd$, we have $\Delta_V$ given by
\eqn{e:Delta}
{ \Delta_V(x,y)
  = \begin{cases}
    2d & \text{if $x = y$;} \\
    -1 & \text{if $x \sim y$;} \\
     0 & \text{otherwise,}
    \end{cases} \qquad x, y \in V. }
We denote the inverse matrix by $g_H = (\Delta_H)^{-1}$,
$g_V = (\Delta_V)^{-1}$. 

We write $\Delta$ for the matrix defined as in \eqref{e:Delta}, but with 
$x, y \in \Zd$, and $g = \Delta^{-1}$ when $d \ge 3$. Up to a 
factor $(2d)^{-1}$, these matrices are the Green function of simple random 
walk. Namely, let $(S(n))_{n \ge 0}$ denote a simple random walk in $\Zd$, and let 
\eqnst
{ \sigma_A
  = \inf \{ n \ge 0 : S(n) \not\in A \}. }
Then $2d \, g_V(\cdot, \cdot) = G_V(\cdot, \cdot)$ and $ 2d \, g(\cdot, \cdot) = G (\cdot, \cdot)$, where we define
\eqnsplst
{ G_V(x,y)
  := \E_x \left[ \sum_{0 \le n < \sigma_V} \mathbf{1}_{S(n) = y} \right], 
     \quad x, y \in V; \\
 G(x,y)
  := \E_x \left[ \sum_{0 \le n < \infty} \mathbf{1}_{S(n) = y} \right],
     \quad x, y \in \Zd. }
Whenever $G$ appears with arguments, it refers to a Green function as defined above; when it appears without arguments, it refers to graphs as in the notation introduced previously.

Let us fix a finite connected graph $H = (U \cup \{ s \}, F)$. 
A \textbf{sandpile} on $H$ 
is a function $\eta : U \to \{ 0, 1, 2, \dots \}$. We say that $\eta$ is 
\textbf{unstable} at $x \in U$, if $\eta(x) \ge \deg_H(x)$. In this case 
$x$ is \textbf{allowed to topple}, which means that $x$ sends one particle 
along each edge incident with it. Particles arriving at $s$ are lost.
Toppling $x$ has the effect of subtracting row $\Delta_H(x,\cdot)$ from
$\eta(\cdot)$. It is a basic property of the model that if unstable 
vertices are toppled in any order until there are no such vertices, 
the stable sandpile obtained is independent of the order chosen (called 
the Abelian property) \cite{D90}. Hence for any sandpile $\eta$ there 
is a well-defined \textbf{stabilization of $\eta$}, denoted $\eta^\circ$. 
The sandpile Markov chain is defined as follows. The state space is 
$\Omega_H = \prod_{x \in U} \{ 0, \dots, \deg_H(x) - 1 \}$. Given that 
the current state is $\eta$, a single step is defined by choosing a 
vertex $X \in U$ uniformly at random, and moving to state 
$(\eta + \mathbf{1}_X)^\circ$. The maps $a_x : \Omega_H \to \Omega_H$
defined by $a_x : \eta \mapsto (\eta + \mathbf{1}_x)^\circ$, $x \in U$, 
are called the \textbf{addition operators}. It follows from the 
uniqueness of stabilization that $a_x a_y = a_y a_x$, $x, y \in U$.
When it is necessary to emphasize the graph $H$ on which the operator $a_x$ is applied, we write $a_{x,H}$ for $a_x$; we also use $a_{x,L}$ for $a_{x, G_L}$.
We denote the set of recurrent states of the sandpile Markov chain by 
$\cR_H$. It is known that the unique stationary distribution 
is given by the uniform distribution on $\cR_H$ \cite{D90}, and that 
each (restricted) map $a_x : \recur_H \to \recur_H$ preserves this measure.

In the special case when the graph arises from a finite $V \subset \Zd$, 
we denote the state space by $\Omega_V = \{ 0, \dots, 2d-1 \}^V$, 
and the set of recurrent states by $\cR_V$. 

Now, let $\eta : \Zd \to \{ 0, 1, 2, \dots \}$ be a sandpile on $\Zd$. If $x$ is unstable in $\eta$, we define the toppling of $x$ using the matrix $\Delta$, that is, $\eta \mapsto \eta(\cdot) - \Delta(x,\cdot)$. Let us call a finite or infinite sequence consisting of topplings of unstable vertices \textbf{exhaustive}, if any vertex that is unstable at some point, is toppled at a later time. It can be shown, similarly to the finite graph case, that for all $x \in \Zd$, $x$ topples the same number of times (possibly infinity) in any exhaustive sequence.

We write $\unitv_i$ for the unit vector in the $i$-th positive coordinate direction, $| \cdot |$ for the Euclidean norm, and $\| \cdot \|$ for the $\ell^\infty$ norm on $\Zd$.
We denote by $o$ the origin in $\Zd$, and we let
\[ V_x(n) = \{ y \in \Zd : \, \| y - x \| \le n \}; \qquad 
   B_x(n) = \{ y \in \Zd: \, |y - x| \le n \}\ ,  \]
and write $V(n) = V_o(n)$ and $B(n) = B_o(n)$.

Given $A,\,B \subset \Zd$, we let $\dist(A,B)$ denote their Euclidean distance, 
and write $\dist(x,B)$ when $A = \{ x \}$.
If $A \subset \mathbb{Z}^d$, we let 
$\partial A = \{x \in \Zd \setminus A : \dist(x,A) = 1\}$. 
When considering $A$ as a subset of $G_L$, 
we will often use $\partial A$ to denote the boundary restricted to $G_L$ --- that is, excluding from $\partial A$ 
any $x \notin G_L$ --- the meaning will be clear in context.

If $z_1, \, z_2$ are two elements of $\mathbb{R}^d$, let $\ang(z_1, z_2)$ denote the angle between $z_1$ and $z_2$. We will make use of the ``little $o$'' notation: 
$a_n = n^{a + o(1)}$ if $\lim_n \left[\,\log a_n / \log n\right] = a$.

\subsection{Toppling numbers and waves}
\label{ssec:waves-def}

Given a sandpile $\eta$ on a finite connected graph $H = (U \cup \{ s \}, F)$, we write
$n(x,y) = n(x,y;\eta)$ for the \textbf{toppling numbers}, i.e., number of times $y$ topples when stabilizing $\eta + \mathbf{1}_x$. 
A useful ordering of topplings, introduced by Ivashkevich, Ktitarev and Priezzhev \cite{IKP94}, for stabilizing $\eta + \mathbf{1}_x$ is in terms of \textbf{waves}, 
which we now define. If $\eta + \mathbf{1}_x$ is
stable, there are no waves. Otherwise, topple $x$ and carry out any
further topplings that are possible without toppling $x$ a second time. 
It is easy to verify that in doing so, every vertex in $U$ topples at most 
once. We write $\cW_{1,H} = \cW_{1,H}(x;\eta)$ for the set of sites toppled 
so far; this is the \textbf{first wave}. If $x$ is still unstable after 
the first wave, which happens if and only if all neighbours of $o$ are in 
$\cW_{1,H}$, topple $x$ a second time, and carry out any further topplings 
that are possible without toppling $x$ a third time. The set of vertices 
that topple in doing so is denoted $\cW_{2,H} = \cW_{2,H}(x;\eta)$; this is
the \textbf{second wave}. We define further waves analogously. 

When the graph arises from a finite $V \subset \Zd$, we write $\cW_{1,V}, \cW_{2,V}, \dots$ for the waves. We make similar definitions for sandpiles on $\Zd$. 
On $\Zd$, the toppling numbers $n(x,y;\eta)$ are possibly infinite. Waves $\cW_1, \cW_2, \dots$ are defined analogously to the finite case, and their number may be infinite. The following lemma is straightforward to verify.

\begin{lem}
\label{lem:num-waves} \ \\
(i) Let $H = (U \cup \{ s \}, F)$ be a connected finite graph. 
The number of waves containing $y$ equals $n(x,y;\eta)$. 
In particular, the number of waves is $N = N(x;\eta) := n(x,x;\eta)$.\\
(ii) In the case of $\Zd$, the number of waves containing $y$ equals 
$n(x,y;\eta)$ (possibly infinite). In particular, the number of waves is 
$N = N(x;\eta) := n(x,x;\eta)$.
\end{lem}

The next lemma gives the expected number of topplings in an avalanche.

\begin{lem}[Dhar's formula \cite{D90}]
\label{lem:Dhar}
Let $H = (U \cup \{ s \}, F)$ be finite. We have 
\eqnst
{ \mathbb{E}_{\nu_H} n(x,y)
  = g_H(x,y), \quad x, y \in U. }
\end{lem}

For cubes in $\Zd$, we have the following deterministic comparison between $n(o,o;\eta)$ 
and the radius $R(\eta)$.


\begin{lem}
\label{lem:ptwise} \ \\
(i) Let $\eta$ be any sandpile configuration on $\Z^d$, $d \ge 1$. Then
\[ n(o,o;\eta) \leq R(\eta)\ . \]
(ii) The same statement holds when $\eta$ is a sandpile configuration in $G_L$ for any $L \geq 1$.
\end{lem}

\begin{proof}
(i) In the proof below, it will be convenient to denote 
$n[\eta] := n(o,o;\eta)$. 
We will also use the notion of ``restricted topplings''. For $L > 0$, let $\eta|_L$ denote the restriction of $\eta$ to $V(L)$; the restricted toppling number $n_L[\eta]$ will denote the number of topplings occurring at $o$ in the stabilization of $\eta|_L + \mathbf{1}_o$ on $G_L$. Note that such a stabilization amounts to taking the configuration 
$\eta + \mathbf{1}_o$ on $\Zd$ and toppling only sites $x \in V(L)$ until all such $x$ are stabilized, neglecting any sites in $\Zd \setminus V(L)$ that may become unstable. In particular, for any stable sandpile configuration $\eta$ on $\Zd$, $n[\eta] \geq n_L[\eta]$ for any $L$.

Assume $\eta$ and $R < \infty$ are as in the statement of the lemma. We claim that the equality $n[\eta] = n_R[\eta]$ holds. Indeed, by the above observation, $n_R$ is exactly the number of topplings at $o$ required to stabilize only the sites of $V(R)$, but by assumption these are the only sites which need to be toppled to stabilize $\eta + \mathbf{1}_o$ in all of $\Zd$. Therefore, it suffices to show that $n_R[\eta] \leq R$.

For this, we will use a special ``maximal'' configuration $\phi_R$ from 
\cite[Lemma 4.2]{FR08}:
\[\phi_R(x) = \begin{cases}2d - 1,\, &x \in V(R) \\ 2d-2 \, &\text{otherwise.}  \end{cases} \]
Note that $n_R[\eta] \leq n_R[\phi_R]$, since $\eta|_R \leq \phi_R$ pointwise. Moreover, $n_R[\phi_R] \leq n[\phi_R]$. 
It is proven in \cite[Lemma 4.2]{FR08}, and not difficult to see by computing each wave, that $n[\phi_R] = R$. This completes the proof of (i).

(ii) The above proof applies here as well, with only minor changes.
\end{proof}

\begin{proof}[Proof of Proposition \ref{thm:simple_rad}]
By Dhar's formula \cite{D90}, we have 
\eqnst
{ \lim_{L \to \infty} \mathbb{E}_{\nu_L} n(o,o) 
  = \lim_{L \to \infty} g_{V_L}(o,o) 
  = \infty, \quad \text{when $d = 2$.} } 
By Lemma \ref{lem:ptwise}(ii), this implies $\lim_{L \to \infty} \mathbb{E}_{\nu_L} R = \infty$.
\end{proof}

\begin{rem}
We do not see a simple way to deduce the statement 
$\E_\nu n(o,o) = \infty$ from the simpler statement that $\lim_{L \to \infty} \E_{\nu_L} n(o,o) = \infty$, 
when $d = 2$. This requires ruling out the possibility that 
$\E_{\nu_L} n(o,o)$ is dominated by rare events with many waves.
Our proof of $\E_\nu n(o,o) = \infty$ in Theorem \ref{thm:exp-wave-infinite} 
will build on quite a few other results; in particular, results from
Section \ref{sec:last-k-waves}.
\end{rem}

\subsection{Spanning trees and the burning bijection}
\label{ssec:trees}

In 1990, Dhar \cite{D90} introduced a method for checking whether a particular stable configuration $\eta$ lies in $\recur_H$, called the ``burning algorithm". Application of the burning algorithm provides a bijection $\varphi$ between $\recur_H$ and the set of all spanning trees of $H$, that we denote by $\ST_H$; see \cite{MD92}. We briefly describe this bijection here. In Section \ref{ssec:newspan}, we give a version of it for ``waves'' which will be necessary in our analysis.

Recall that $a_{xy}$ denotes the number of edges between vertices $x$ and $y$.

\begin{lem}[{Burning algorithm \cite{D90}, \cite[Lemma 4.1]{HLMPPW}}]
\label{lem:burning}
Let $\eta$ be a stable sandpile on a connected finite graph
$H = (U \cup \{ s \}, F)$. At each $x \in U$, add $a_{xs}$ grains of sand, and 
stabilize. We have $\eta \in \recur_H$ if and only if each vertex 
in $U$ topples exactly once. 
\end{lem}

Note that instead of adding sand as in the lemma, we may initiate the 
toppling process by placing all $\sum_{x \in U} a_{xs} = \deg_H(s)$ 
grains at $s$, and toppling $s$ first. Suppose we carry out any 
possible topplings in parallel. We say that \textbf{$x$ burns at time $k$}, 
if it is toppled in the $k$-th parallel toppling step, where we regard 
$s$ to have burnt at time $0$.

The bijection is defined as follows. For each $y \in U$, fix an arbitrary ordering $\prec_y$ of the edges adjacent to $y$. Given $\eta \in \recur_H$, 
for each $y \in U$, we adjoin to the tree $T$ an edge connecting $y$ 
to a neighbour burnt one time step before, chosen as follows. 
If $P_y$ is the number of edges joining $y$ to neighbours burnt 
before $y$, and $A_y$ is the subset of such edges leading to sites burnt 
one step before $y$, then the burning rule implies
\[ \eta(y) = \deg_H(y) - P_y + i \quad \text{for some $0 \leq i < |A_y|$}\ .\]
We add to $T$ the $i$-th edge in $A_y$ in the ordering $\prec_y$.

The resulting graph $T$ will be a spanning tree (the fact that it spans --- i.e., that every site topples in this procedure --- is part of the content of Lemma \ref{lem:burning}), and we set $\varphi(\eta) = T$. The map $\varphi : \recur_H \to \ST_H$ is usually referred to as the ``burning bijection'', and the toppling procedure used to construct $\varphi$ will be referred to as the ``burning procedure''.



\subsection{Intermediate configurations}
\label{ssec:interm}

Now, we give a description of waves in terms of recurrent configurations on an 
auxiliary graph, as introduced in \cite{IKP94}. As in the previous section, we
describe this in an arbitrary finite connected graph $H = (U \cup \{ s \}, F)$.
Suppose we are interested in waves started by the addition of a particle at a fixed vertex $w \in U$. Consider the graph $H^\newp$ obtained from $H$ by adding the edge $f' := \{ w, s \}$. 
For readability, we denote $\Delta^\newp_H := \Delta_{H^\newp}$ and $\recur^\newp_H := \recur_{H^\newp}$. Let $a^\newp_x = a^\newp_{x,H}$, $x \in U$ denote the addition operators on $H^\newp$. We reserve the notation $\eta^\circ$ for stabilization on the original graph $H$. When we need to emphasize that addition is applied on the graph derived from $H$, we prefer the $a^{\newp}_{x,H}$ notation.

The burning algorithm (Lemma \ref{lem:burning}) implies that 
$\recur_H \subset \recur^\newp_H$. The following lemma
compares the sizes of these two sets.
 
\begin{lem}
\label{lem:wave-to-recur}
For any finite connected graph $H=(U \cup \{ s \}, F)$, and $w \in U$, we have
$|\recur^\newp_H| = ( 1 + g_H(w,w) ) |\recur_H|$.
\end{lem} 
 
\begin{proof}
Let $\mathbf{1}_{w,w}$ denote the $U \times U$ matrix whose only non-zero
entry is a $1$ at $(w,w)$. We have 
\eqnspl{e:recur-compare}
{ |\recur^\newp_H|
  &= \det (\Delta^\newp_H)
  = \det (\Delta_H + \mathbf{1}_{w,w})
  = (1+ (\Delta_H)^{-1}(w,w)) \, \det(\Delta_H) \\
  &= (1 + g_H(w,w)) \, |\recur_H|. }
\end{proof}

The following immediate corollary will be useful in controlling avalanches. 
A version of part (i) was proved in \cite[Lemma 7.5]{JR08}.

\begin{cor} 
\label{cor:wave-to-recur-Zd}
Consider the sequence $(V_L)_L$. \\
(i) Suppose $d \ge 3$. There exists a constant $C(d)$ such that 
$|\recur^\newp_L| \leq C(d) \, |\recur_L|$ for all $L \ge 1$. \\
(ii) Suppose $d = 2$. There exists a constant $C$ such that 
$|\recur^\newp_L| \leq C \, \log L \, |\recur_L|$ for all $L \ge 2$.
\end{cor}

\begin{proof}
Both statements follow from Lemma \ref{lem:wave-to-recur}, the equality 
$g_L(o,o) = (2d)^{-1} G_L(o,o),$ and known properties of the 
Green function $G_L(o,o)$; see e.g.~\cite{Lawlim}. 
\end{proof}

We now describe the interpretation of waves as elements of 
$\recur^\newp_H \setminus \recur_H$; introduced in \cite{IKP94}. 
Let $\eta \in \recur_H$, and suppose that $\eta + \mathbf{1}_w$
is unstable at $w$ in $H$, i.e.~$\eta(w) = \deg_H(w) - 1$.
Consider the waves occurring in stabilizing $\eta + \mathbf{1}_w$ in $H$.
Recall that $N = N(\eta)$ denotes the number of waves. For $1 \le k \le N,$ let 
$\eta_k = a^\newp_w \eta_{k-1}$, where $\eta_0 = \eta$. It is straightforward 
to check that $\eta_k$ is the configuration seen just before the $k$-th wave is
carried out, and $a_w \eta = (a'_w)^{N+1} \eta$. Note that the latter statement
also holds, trivially, when $\eta + \mathbf{1}_w$ is stable, in which case $N = 0$. 

%
\begin{defin}
Let $\eta \in \recur_H$ be such that $\eta + \mathbf{1}_w$ is unstable at $w$. We call the sequence $\alpha(\eta) := (\eta_1, \ldots, \eta_N)$ 
the \textbf{intermediate configurations} corresponding to $\eta$. 
\end{defin}

We record here the characterization of $\recur'_H \setminus \recur_H$; a similar statement was shown in \cite{JRS15} for a continuous height model. 

\begin{lem}
The collection $\{ \alpha(\eta) : \eta \in \recur_H,\, \eta(w) = \deg_H(w) - 1 \}$ 
forms a partition of $\recur'_H \setminus \recur_H$.
\end{lem}

\begin{proof}
Since $\eta_k(w) = \deg_H(w)$, $k = 1, \dots, N$, we have $\eta_k \in \recur'_H \setminus \recur_H$, $k = 1, \dots, N$. This and the relation $a_w \eta = (a'_w)^{N+1} \eta$ imply that $\alpha(\eta)$ has distinct entries (the order of $a'_w$ is at least $N+1$). By Dhar's formula, the average number of waves per recurrent configuration is $g_H(w,w)$, so $g_H(w,w) \vert \recur_H \vert$ elements of $\recur_H'$ correspond to intermediate configurations.

It is similarly easy to check that $\eta \in \recur_H'$, accounting for another $\vert \recur_H \vert$ elements of $\recur_H'$. Comparing this with Lemma \ref{lem:wave-to-recur}, we see that every element of $\recur_H'$ is either an intermediate configuration or recurrent on $H$, completing the proof.
\end{proof}

Given $\eta_* \in \recur'_H \setminus \recur_H$, we denote by $\wave(\eta_*)$ the set of vertices that topple in the stabilization $a'_{w,H}(\eta_*)$. 
It is immediate from this definition that if $\alpha(\eta) = (\eta_1, \dots, \eta_N)$, 
then $\wave(\eta_k)$ is the $k$-th wave corresponding to $\eta$, i.e.~$\wave_{k,H}(w;\eta)$. Of particular interest will be the last wave $\wave(\eta_N)$. The following corollary follows directly from the definitions.

\begin{cor}
\label{cor:last-waves}
An intermediate configuration $\eta_* \in \recur'_H \setminus \recur_H$ is a last wave if and only if there exists $y \sim w$, $y \in U \cup \{ s \}$, such that $y \not\in \wave(\eta_*)$.
\end{cor}

We will also need the following lemma.

\begin{lem}
\label{lem:num-last-waves}
We have
\eqnst
{ \frac{1}{\deg_H(w)} | \recur_H |
  \le | \{ \eta_* \in \recur'_H : \, \text{$\eta_*$ is a last wave} \} |
  \le | \recur_H |. }
\end{lem}

\begin{proof}
The upper bound is obvious. To see the lower bound, we assign to $\eta \in \recur_H$ the last intermediate configuration in the stabilization of $\eta + (\deg_H(w) - \eta(w)) \mathbf{1}_w$. This map is at most $\deg_H(w)$ to $1$, proving the lower bound.
\end{proof}


\subsection{Bijection for intermediate configurations}
\label{ssec:newspan}

We now specialize to the set-up where $H = G_V$, $V \subset \Z^d$ finite, $o \in V$.
In this section, we describe a version of the burning bijection on $G'_V$, that will allow us to control topplings occurring in a wave. To the best of our knowledge, such a bijection was first introduced by Ivashkevich, Ktitarev and Priezzhev \cite{IKP94}. See also \cite{P00,JR08,JRS15}, where it played a key role. For many of our results, the variant in \cite{IKP94} would suffice. However, a more careful choice of the burning process will be needed in Section \ref{sec:last-k-waves}, so we introduce here the version we need. Our burning process is similar to burning processes introduced in \cite{JW14} and \cite{GJ14}.


Let $\eta_* \in \recur'_V \setminus \recur_V$. We define a pair of vertex-disjoint trees
$(T_o,T_s) = \varphi'(\eta_*)$, such that $T_o \cup T_s$ spans $G_V$.
Send one grain of sand from $s$ to $o$, resulting in $2d$ grains at $o$.
We sequentially topple vertices in the balls $B(0) \cap V, B(1) \cap V, B(2) \cap V, \dots$, and build a tree rooted at $o$, similarly to the usual burning bijection.
The precise definitions of burnt and unburnt sets are as follows. We let
\begin{align*}
  \Burnt^{(0)}_0
  &= \{ o \} & 
  \Unburnt^{(0)}_0
  &= V \cup \{ s \} \setminus \{ o \} \\
  \Burnt^{(0)}_k
  &= \es, \quad k \ge 1, & 
  \Unburnt^{(0)}_k
  &= V \cup \{ s \} \setminus \{ o \}, \quad k \ge 1. 
\end{align*}
For $r \ge 1$, inductively, we set
\begin{align*}
  \Burnt^{(r)}_0
  &= \cup_{\ell \ge 0} \Burnt^{(r-1)}_\ell &
  \Unburnt^{(r)}_0
  &= V \cup \{ s \} \setminus \Burnt^{(r)}_0 \\
  \Burnt^{(r)}_k
  &= \left\{ x \in B(r) \cap \Unburnt^{(r)}_{k-1} : 
    \eta_*(x) \ge \deg_{\Unburnt^{(r)}_{k-1}}(x) \right\} &
  \Unburnt^{(r)}_k 
  &= \Unburnt^{(r)}_{k-1} \setminus \Burnt^{(r)}_k, \quad k \ge 1. 
\end{align*}
For each $r \ge 1$ there exists a smallest index $J = J(r) \ge 1$
such that $\Burnt^{(r)}_J = \es$, and there is a smallest index $R \ge 1$
such that $J(R+1) = 1$. Then $\Burnt^{(R+1)}_0 = \wave(\eta_*)$ is the set of 
vertices toppled in the wave represented by $\eta_*$. 

We complete the burning process by sending $a_{sx}$ grains of sand from $s$ to $x$ for each $x \in V$, and follow the usual burning rule. That is, we set:
\begin{align*}
  \tBurnt_0
  &= \cup_{r \ge 0} \cup_{\ell \ge 0} \Burnt^{(r)}_\ell &
  \tUnburnt_0
  &= V \setminus \tBurnt_0 \\
  \tBurnt_k
  &= \left\{ x \in \tUnburnt_{k-1} : 
    \eta_*(x) \ge \deg_{\tUnburnt_{k-1}}(x) \right\} &
  \tUnburnt_k
  &= \tUnburnt_{k-1} \setminus \tBurnt_k, \quad k \ge 1. 
\end{align*}

We now define the bijection. If $o \not= u \in V \cap \Burnt^{(R+1)}_0$, then there exists a unique pair $(r,k)$ with $r \ge 1$ and $k \ge 1$ such that $u \in \Burnt^{(r)}_k$. Due to the definition of the burning rule, there exists at least one $y \sim x$ such that $y \in \Burnt^{(r)}_{k-1}$. We select an edge that connects $u$ to one of these vertices, using the ordering $\prec_u$, as in Section \ref{ssec:trees}. Namely, if $P_u$ is the number of edges joining $u$ to neighbours in $\cup_{\ell < k} \Burnt^{(r)}_\ell$, and $A_u$ is the subset of such edges leading to vertices in $\Burnt^{(r)}_{k-1}$, then necessarily
\[ \eta_*(u) = 2d - P_u + i \quad \text{for some $0 \leq i < |A_u|$}\ .\]
We add to $T_o$ the $i$-th edge in $A_u$ in the ordering $\prec_u$.

If $u \in \Unburnt^{(R+1)}_0 = \tUnburnt_0$, there is a unique $k \ge 1$ such that $u \in \tBurnt_k$, and there exists at least one $y \sim x$ with $y \in \tBurnt_{k-1}$. We select an edge to one of these vertices, using the ordering $\prec_u$ as before. Namely, if $P_u$ is the number of edges joining $u$ to neighbours in $\cup_{\ell < k} \tBurnt_\ell$, and $A_u$ is the subset of such edges leading to vertices in $\tBurnt_{k-1}$, then necessarily 
\[ \eta_*(u) = 2d - P_u + i \quad \text{for some $0 \leq i < |A_u|$}\ .\]
We add to $T_s$ the $i$-th edge in $A_u$ in the ordering $\prec_u$.
Let $\varphi'(\eta_*) := (T_o,T_s)$ denote the two components spanning forest obtained by the above construction. Let us write $\ST_{V,o}$ for the set of all spanning forests of $G_V$ rooted at $\{ s, o \}$. 

\begin{lem} 
\label{lem:varphi'} \ \\
(i) The map $\varphi'$ is a bijection between $\recur'_V \setminus \recur_V$ and $\ST_{V,o}$.\\
(ii) For any $\eta_* \in \recur'_V \setminus \recur_V$, the vertex set of $T_o(\eta_*)$ equals $\wave(\eta_*)$. \\
(iii) We have the following property:
\eqn{e:path-property-gen}
{ \parbox{14.5cm}{If there is a path from $o$ to a vertex $x \in V$
    in $T_o = \varphi'(\eta_*)$ that stays inside $B(r)$, 
    then starting from $\eta_* + \mathbf{1}_o$ there is a sequence of 
    topplings in $B(r)$ that topples $x$.} }
\end{lem}

\begin{proof}
(i) 
Let $\eta_* \not= \heta_* \in \recur'_V \setminus \recur_V$. Tracing the burning process to the first time when a vertex with $\eta_*(x) \not= \heta_*(x)$ is encountered, we see that $\varphi'$ is injective on $\recur'_V \setminus \recur_V$. It follows from the definitions that $\varphi'(\recur'_V \setminus \recur_V)$ is a subset of the set of spanning forests of $G_V$ rooted at $\{ o, s \}$. 
By the matrix-tree theorem applied to $G'_V$, the number of spanning forests of $G_V$ rooted at $\{ o, s \}$ equals $\det(\Delta'_V) - \det(\Delta_V)$. This also equals $|\recur'_V \setminus \recur_V|$ \cite{D90}, so statement (i) follows.

(ii) The burning process that was used to define $\Burnt^{(R+1)}_0$ can be identified with topplings in the wave corresponding to $a'_{o,V}(\eta_*)$. This implies that $\Burnt^{(R+1)}_0 = \wave(\eta_*)$, and this is the vertex set of $T_o$.

(iii) This again follows directly from the interpretation of the burning process in terms of topplings in the wave.
\end{proof}

\subsection{Random walk notation and basic facts}
\label{ssec:rwdef}

Many of our techniques require a detailed analysis of spanning trees via Wilson's algorithm. For this reason, we will often have to consider collections of simple random walks (``SRW'') and loop-erased random walks (``LERW'') on (subsets of) $\mathbb{Z}^d$.

We will denote by $S_x = (S_x(0),\, S_x(1),\, \ldots)$ an infinite simple random walk on $\mathbb{Z}^d$ started at $x$, so that $S_x(0) = x$.
We will suppress the subscript when the choice of $x$ is clear or when $x = o$. In general, $S_x$ and $S_y$ will be assumed independent when $x \not= y$. When multiple independent walks from one site are necessary, we write $S_x, \, S_x^{\newp}$ (etc).

If $A \subset \mathbb{Z}^d$ we define the standard stopping times
\begin{align}
\sigma_A &= \inf\{n \geq 0: S_x(n) \notin A\} & 
\xi_A &= \inf\{n \geq 0: S_x(n) \in A\}\ , \\
\osigma_A &= \inf\{n \geq 1: S_x(n) \notin A\} & 
\oxi_A &= \inf\{n \geq 1: S_x(n) \in A\}\ . \label{eq:stoptimedefs}
\end{align}
Note that we suppress any dependence of these stopping times on the starting point $x$; when we write (for instance) $S_x(\xi_A)$, we are referring to the location of $S_x$ at its first hitting time on $A$. 
When the starting vertex may be ambiguous, we use subscripts on the symbol $\prob$; for instance,
\[\prob_x(\xi_A < \xi_B) = \prob(S_x[0, \xi_A] \cap B = \emptyset)\ . \]
We will abbreviate $\sigma_{V(n)}$ to $\sigma_n$, and 
write $\xi_o$ for $\xi_{\{o\}}$.

The trace of a walk $S_x$ between two times $a < b$ (where $b$ can be infinite) will be denoted
\[S_x[a,b] := \{ S_x(j) : a \le j \le b \}, \]
and similar notation will be used for other intervals--e.g., $S(a,b)$ and so on. We will sometimes abuse notation and treat $S_x[a,b]$ as a sequence instead of an unordered set.

If $A \subset \Z^d$ is a finite connected set and if $x, y \in \Zd$, recall the Green function
\begin{equation}
  G_A(x,y) := \sum_{0 \le j < \sigma_A} \prob(S_x(j) = y)\ .
\label{eq:green_fun_def}
\end{equation}
As before, $G_{V(n)}$ is abbreviated $G_n$.
We will use the following standard asymptotics for the Green function inside a large ball and the probability of hitting $o$ before exiting a large ball:
\begin{thm}[See \cite{Lawler}, Prop. 1.5.9 and 1.6.7]
\label{thm:green_asymp} \ \\
(i) If $d = 2$, we have uniformly in $x \in B(n)$:
\begin{align*}
G_{B(n)}(o,x) &=  \frac{2}{\pi}\left[\log n - \log |x| \right] + O\left(|x|^{-1} + 1/n\right) \\
\prob_x\left(\xi_o < \sigma_{B(n)} \right) &= \frac{1}{\log n}\left[\log n - \log |x| +  O\left(|x|^{-1} + 1/n\right) \right]\ .
\end{align*}
(ii) For $d \geq 3$, there exist $c_1 = c_1(d), c_2 = c_2(d) > 0$ such that, uniformly in $n$ and $x \in B(n)$:
\begin{align*}
G_{B(n)}(o,x) =  c_1 \left[|x|^{2-d} - n^{2-d}\right]  + O\left(|x|^{1-d}\right)\\
\prob_x\left(\xi_o < \sigma_{B(n)} \right) = c_2 \left[|x|^{2-d} - n^{2-d}\right]  + O(|x|^{1-d})\ .
\end{align*}
\end{thm}

We also note the following simple observation about $G$. If $K_1 \subset K_2$ and $x, y \in K_1,$
\begin{equation}
\label{eq:greenmono}
G_{K_1}(x,y) \leq G_{K_2}(x,y). \end{equation}

We will need the following result, usually called the Beurling estimate, 
which gives an upper bound on the probability that a path in $\mathbb{Z}^2$ 
is not hit by SRW.

\begin{lem}[{Beurling estimate \cite{kestenwalk}, \cite[Section 6.8]{Lawlim}}]
\label{lem:Beurling}
Consider $\mathbb{Z}^2$. There is a constant $C$ such that the following bound holds, uniformly in $x$, $n$, and lattice paths $\alpha$ connecting $o$ to $\partial B(n)$:
\[\prob_x(\sigma_{B(n)} < \xi_\alpha) \leq C \left( \frac{|x|}{n}\right)^{1/2}\ . \]
\end{lem}

Given a transient random walk $S_x$, we denote by $\looper S_x$ the loop-erasure of $S_x$, where loops are erased in forward chronological order. A similar definition is made for $\looper S_x[a,b] := \looper (S_x[a,b])$, etc.; note that for a finite segment of a random walk, the loop erasure can be defined even in the case that the walk is recurrent. See \cite[Chapter 9]{Lawlim} and \cite[Chapter 7]{Lawler} for background on LERW; in particular, for properties not detailed below.

If $A \ni x$ is a finite subset of $\mathbb{Z}^d$, let $\hS_x^A$ denote a finite loop-erased random walk killed at the boundary of $A$:
\[\hS_x^A = \looper S_x[0, \sigma_A]\ . \]
We will also make use of infinite loop-erased random walks $\hS_x$ on $\mathbb{Z}^d$. 
When $d \geq 3$, the definition
\[ \hS_x := \looper S_x \]
is unambiguous with probability $1$, as noted in the last paragraph. 

For $d = 2$, the loop-erasure of the infinite random walk $S_x$ is not well-defined; in this case, $\hS_x$ is defined by taking limits. It is known (see 
\cite[Section 7.4]{Lawler}) that for any finite lattice path $\gamma$ with $|\gamma| = k$,
\[\lim_{n \rightarrow \infty} \prob(\hS_x^{B(n)}[0, k] = \gamma) =: \prob(\hS_x[0, k] = \gamma) \]
exists, and the extension of this to a measure on infinite paths gives a definition of the distribution of $\hS_x$. The rate of convergence of $\hS_x^{B(n)}$ to $\hS_x$ is well-controlled; see Lemma \ref{lem:seg_indep} below. We will refer to all of the processes $\hS_x$, $\hS_x^A$ as loop-erased random walks or LERW. We will also assume as usual (unless stated otherwise) that a LERW is independent of any other walks appearing in a given statement. 

We define LERW stopping times $\hxi$, $\hsigma$ analogously to $\xi$ and $\sigma$; for instance,
\[ \hsigma_K = \inf\{ n \ge 0 : \, \hS_x(n) \notin K \}\ .  \]
As before, we will use $x$ as a subscript on $\mathbb{P}$ when considering stopping times to indicate the starting point. When the LERW is finite (i.e., we are considering $\hS_x^A$), we also will use the superscript $A$ on $\prob$, writing (for instance)
\[ \prob_x^A \big( \hxi_{K_1} < \hxi_{K_2} \big) \]
to avoid confusing the set in which the LERW lives with the set it is hitting. When the LERW is infinite, we omit superscripts altogether.

One result important for analyzing LERW is the following ``Domain Markov Property" (DMP). This roughly says that the terminal segment of a LERW can be built by starting a SRW at the tip of the initial segment, conditioning it not to hit the initial segment, then erasing loops.
\begin{lem}[{Domain Markov Property; see \cite[Chapter 11]{Lawlim}}]
\label{lem:dmp}
Let $\hS_x^{K}$ be a loop-erased walk in $K$, and let $\alpha$ be a finite path of length $m$ such that
\[ \prob \Big( \hS_x^{K}[0,m] = \alpha \Big) > 0 \ .\]
Then for all paths $\beta$,
\[\prob \Big( \hS_{x}^K [m, \hsigma_K] = \beta \, \Big| \, \hS_{x}^K[0, m] = \alpha \Big) 
  = \prob \Big(\looper S_{\alpha(m)}[0, \sigma_K] = \beta \,
    \Big| \, \sigma_K < \oxi_\alpha \Big). \]
\end{lem}


Note that Lemma \ref{lem:dmp} is in fact much more general:
it holds for infinite LERW (on $\Zd$ for $d \geq 3$) and 
finite LERW on graphs which are not necessarily subsets 
of $\mathbb{Z}^d$.



\subsection{Wilson's algorithm}
\label{ssec:wilson}

Let $H = (U \cup \{ s \}, F)$ be a connected finite graph. Since $\nu_H$ is uniform on $\recur_H$, the bijection $\varphi$ maps it to the uniform measure on spanning trees of $H$, which we denote by $\mu_H$. In this section, we collect basic facts about Wilson's algorithm, that we will use to analyze $\mu_H$. For an in-depth introduction to uniform spanning trees (UST) see the book \cite{LPbook}.

We denote a sample from $\mu_H$ by $\frT^H$. We will usually identify a spanning tree of $H$ with the set of edges it contains, so $\frT^H \subset F$.
Pemantle \cite{pemantle} proved that for $x, y \in U$, the path in $\frT^H$ between $x$ and $y$ is distributed as a LERW from $x$ to $y$ (i.e.~as $\looper S_x[0, \xi_y]$). Wilson's algorithm \cite{W96} provides a method for constructing the full UST (a sample from $\mu_H$) from LERWs. 

\emph{Wilson's algorithm.} Let $v_1, \, \ldots, \, v_n$ be an enumeration of the vertices in $U$. We construct a random sequence of tree subgraphs 
$\frF_0 \subset \ldots \subset \frF_n$. 
Let $\frF_0$ have vertex set $\{ s \}$ and empty edge set. If $\frF_{i-1}$ has been defined for some $1 \le i \le n$, we let $S_{v_i}$ be a simple random walk on $H$ started at $v_i$, and let $\xi_{V(\frF_{i-1})}$ be the first hitting time of the vertex set of $\frF_{i-1}$ by $S_{v_i}$. We set $\frF_i = \frF_{i-1} \cup \looper S_{v_i}[0, \xi_{V(\frF_{i-1})}]$--that is, the edges in the loop-erasure of $S_{v_i}[0, \xi_{V(\frF_{i-1})}]$ are added to $\frF_{i-1}$. The output of the algorithm is $\frF_n$. Wilson's theorem \cite{W96} implies that $\frF_n$ is uniform, i.e.~distributed as $\frT^H$. 


The measure $\mu_L := \mu_{G_L}$ is known as the UST in $V(L)$ with the \textbf{wired boundary condition}. For studying the $L \to \infty$ limit of sandpiles on $G_L$, as well as sandpiles on $\mathbb{Z}^d$, it will be useful to consider the weak limit 
$\lim_{L \to \infty} \mu_{L} =: \wsf$, called the \textbf{wired spanning forest measure}.
Existence of the limit is implicit in \cite{pemantle}; see \cite{LPbook} for an in-depth treatment. 
We denote a sample from $\wsf$ by $\frT$.

It is well known that $\wsf$ concentrates on spanning forests of $\Zd$ all whose components are infinite. Pemantle \cite{pemantle} showed that for $d \leq 4$, $\frT$ is a tree $\wsf$-a.s, while for $d \geq 5$, $\frT$ has infinitely many connected components $\wsf$-a.s. This dichotomy is the underlying fact behind mean-field behaviour of the sandpile model for $d \ge 5$; which is reflected in some of our results and proofs. We write $\frT_x$ for the component of $\frT$ containing $x \in \Zd$.

It is possible to construct $\frT$ more directly, using an appropriate extensions of Wilson's algorithm. Let $v_1, v_2, \dots$ be an enumeration of $\Zd$.
When $d \ge 3$, we set $\frF_1 = \looper S_{v_1}[0,\infty)$. Then for $i \ge 2$, we inductively define $\frF_i = \frF_{i-1} \cup \looper S_{v_i}[0,\xi_{V(\frF_{i-1})}]$, where the stopping time may be finite or infinite. See \cite{BLPS,LPbook} for a proof that $\cup_{i \ge 1} \frF_i$ has the distribution of $\frT$ given by $\wsf$. This is called Wilson's method rooted at infinity.

When $d = 2$, a method analogous to that in finite volume can be used. 
We set $\frF_1 = \{ v_1 \}$, and for $i \ge 2$ we inductively define $\frF_i = \frF_{i-1} \cup \looper S_{v_i}[0,\xi_{V(\frF_{i-1})}]$; see \cite{BLPS,LPbook}.


It will be important for our proofs that $\wsf$-a.s.~each component of $\frT$ has \textbf{one end}, for all $d \ge 2$. This means that any two infinite self-avoiding paths lying in the same component of $\frT$ have a finite symmetric difference. For $2 \le d \le 4$ this was proved by Pemantle \cite{pemantle}. For $d \ge 5$ this was first shown by Benjamini, Lyons, Peres and Schramm \cite{BLPS} (who generalized it to a much larger class of infinite graphs).

\subsection{Wiring $o$ to the boundary}
\label{ssec:wsfo}

Let $\mu_{L,o}$ be the uniform measure on $\cT_{L,o}$.  
Under $\mu_{L,o}$, we denote the components containing
$o$ and $s$, respectively, by $\frT_{L,o}$ and $\frT_{L,s}$, respectively.


Due to monotonicity, the weak limit $\lim_{L \to \infty} \mu_{L,o} =: \wsf_o$ exists;
see \cite{BLPS,LPbook}. We also use $\frT$ to denote a sample from $\wsfo$, and write $\frT_x$ for its component containing $x \in \Z^d$. In particular, $\frT_o$ is the component containing $o$ under the measure $\wsfo$. It was shown by Lyons, Morris and Schramm \cite{LMS} that for $d \ge 3$ (and more generally under a suitable isoperimetric condition), we have $\wsfo (|\frT_o| < \infty) = 1$, and also that this is equivalent to the one-end property of $\wsf$. J\'arai and Redig \cite{JR08} used the finiteness of $\frT_o$ to show that $\nu ( S < \infty ) = 1$ when $d \ge 3$.

\section{Toppling probability bounds in low dimensions}
\label{sec:toppling}

In this section, we give a proof of our toppling probability lower bounds
stated in Theorem \ref{thm:topple}. These will follow from the following theorem, 
which gives a lower bound in terms of a non-intersection probability 
between a loop-erased walk and a simple random walk, 
and a random walk hitting probability. Our arguments also apply to
$d \ge 5$ with an identical statement, however, this case is already known 
from \cite[Section 6.2]{JRS15}.

Recall that $S_x$ and $S'_x$ are independent simple random walks from the site $x$.
\begin{thm}
\label{thm:gen_prob_bd}
Assume $2 \le d \le 4$. There is a constant $c = c(d) > 0$ such that, 
for all $z \in \Z^d$ we have
\eqn{e:toppling-lbd-Zd}
{ \nu(z \in \av) 
  \geq c \, \prob \left( S'_o[0, \sigma_{|z|}] \cap 
       \hS_o(0, \hsigma_{|z|}] = \varnothing \right)\, 
       \prob (z \in S_o[0, \infty))\ .}
Moreover, the right hand side of \eqref{e:toppling-lbd-Zd} is a lower bound on
$\nu_L (z \in \av)$ for all $L \ge 4 \| z \|$.
\end{thm}

In Sections \ref{ssec:prelim}--\ref{ssec:steer}, we state and prove preliminary results which are useful for establishing Theorem \ref{thm:gen_prob_bd}, and in Section \ref{ssec:lower-bound-sep} we use these to prove the theorem. 
In Section \ref{ssec:lower-bound-sep}, we also give a corollary which will be useful for proofs of later theorems.

\subsection{Preliminary setup}
\label{ssec:prelim}

Our strategy for proving Theorem \ref{thm:gen_prob_bd} will be to work in large finite volume $V(L)$. That is, given a particular $z \in \Zd,$ we will choose some $L_0$ sufficiently large, so that the probability $\nu_L(z \in \av)$ is close to the claimed value for all $L \ge L_0$.
We will require $L_0$ to be on the order of some large multiple of $\| z \|$. 

The main idea of the proof is to show a lower bound for the probability that $z$ is in the last wave of the avalanche. 
By Corollary \ref{cor:last-waves}, 
\eqnsplst
{ |\{\eta \in \recur_L : \, z \in \av\}| 
  &\ge |\{\eta \in \recur_L : \, \eta(o) = 2d-1,\, z \in \wave_{N(\eta)} \}| \\
  &= |\{\eta_* \in \recur'_L \setminus \recur_L : \, z \in \wave(\eta_*), \, v \notin \wave(\eta_*) \text{ for some } v \sim o \} |\ .}
Dividing by $|\recur_L|$, using Lemma \ref{lem:num-last-waves} and symmetry of $V(L)$, for any fixed $e \sim o$ we get
\eqnspl{eq:treebd}
{ \nu_L (z \in \av) 
  &\geq \mu_{L,o} ( z \in \frT_{L,o} \,|\, 
     v \notin \frT_{L,o} \text{ for some } v \sim o ) \\
  &\geq (2d)^{-1} \, \mu_{L,o} ( z \in \frT_{L,o} \,|\, 
     e \not\in \frT_{L,o} ) }
uniformly in $z, L$.

We analyze the event 
\eqnst
{ A(z,e)
  = \{ z \in \frT_{L,o},\, e \not\in \frT_{L,o} \}. }
Let us apply Wilson's algorithm in the graph $G_{L,o}$, starting with a walk $S_e$ from $e$, followed by a walk $S_z$ from $z$. This gives that for fixed $e \sim o$, the occurrence of $A(z,e)$ is equivalent to a LERW from $e$ to exit $V(L)$ without hitting $o$, and a SRW from $z$ to hit $o$ before hitting the LERW. We will bound the 
right-hand side of \eqref{eq:treebd} from below by analyzing this random walk event. 
By time reversal, we will be able to consider the SRW going from $o$ to $z$.
The lower bound then contains two factors: the LERW and SRW avoiding 
each other near $o$, and the SRW subsequently hitting $z$. This leads 
to the two probabilities in Theorem \ref{thm:gen_prob_bd}.

We begin by expressing the probability in \eqref{eq:treebd} in terms of the random walk construction specified above (Lemma \ref{lem:rwprob}). We then lower bound the probability of the resulting walk event by something amenable to analysis by walk intersection techniques that we give in Section \ref{ssec:steer}. 
Let $\pi = \looper S_e[0, \sigma_L]$. 

\begin{lem}
\label{lem:rwprob} \ \\
(i) We have
\begin{align}
  \mu_{L,o} (A(z,e))
  = \prob \left( \pi \cap S_z[0, \xi_o] = \varnothing,\, 
    \xi_o^{S_z} < \sigma_{L}^{S_z} \right).
\label{eq:to_reverse}
\end{align}
(ii) There are constants $\kappa(d) > 0$, $d \ge 2$, such that as $L \to \infty$, we have 
\begin{align}
  \mu_{L,o} (e \notin \frT_{L,o})
  = \prob \left( o \notin S_e[0, \sigma_{V(L)}] \right)
  \sim \begin{cases}
      \kappa(2) (\log L)^{-1} & \text{when $d = 2$;} \\
      \kappa(d) & \text{when $d \ge 3$.}
      \end{cases}
\label{eq:for_cond_walk}.
\end{align}
\end{lem}

\begin{proof}
Note that $e \notin \frT_{L,o}$ if and only if $S_e$ exits $V(L)$ before hitting $o$.
This implies the equality in (ii). The asymptotics in (ii) for $d = 2$ follow from Theorem \ref{thm:green_asymp}.

Given that the event $e \notin \frT_{L,o}$ has occurred, $z$ will be in
$\frT_{L,o}$ if and only if $S_z$ hits $o$ before exiting $V(L)$, and 
does so avoiding $\pi$. This implies statement (i).
\end{proof}

%
In the sequel, we make use of the event $\Gamma_{z,L}$, defined as
\[ \Gamma_{z,L} 
   = \left\{ \pi \cap V_z(\| z \| / 10) = \varnothing,\, 
     \xi_z^{S_o} < \sigma^{S_o}_{4 \| z \|},\, 
     \pi \cap S_o[0, \xi_z] = \varnothing \right\}\ . \]

\begin{lem}
\label{lem:reverse_walk_z}
For all $L \ge 100 |z|$, we have
\begin{equation}
\label{eq:reverse_walk_z}
 \prob \left( \pi \cap S_z[0, \xi_o] = \varnothing,\, 
   \xi_o^{S_z} < \sigma_{L}^{S_z} \right) 
 \geq \begin{cases}
   c \, \prob(\Gamma_{z,L}) \, \log |z|, \quad &d = 2\ ,\\
   (2d)^{-1} \, \prob(\Gamma_{z,L}), \quad &d > 2\ .
 \end{cases}
 \end{equation}
\end{lem}

\begin{proof}
Using reversibility of the random walk, we can rewrite the probability in the left hand side of \eqref{eq:reverse_walk_z} as follows:
\eqnspl{e:rev-green}
{ &\prob \left( \pi \cap S_z[0, \xi_o] = \varnothing,\, 
    \xi_o^{S_z} < \sigma_{L}^{S_z} \right) \\
  &\qquad = \E \Big( 
    \prob \Big( \pi \cap S_z[0, \xi_o] = \varnothing,\, 
    \xi_o^{S_z} < \sigma_{L}^{S_z} \,\Big|\, \pi \Big) \Big) \\
  &\qquad = \E \left( \frac{G_{V(L) \setminus \pi}(z,z)}{G_{V(L) \setminus \pi}(o,o)} \,
    \prob \Big( \pi \cap S_o[0, \xi_z] = \varnothing,\, 
    \xi_z^{S_o} < \sigma_{L}^{S_o} \,\Big|\, \pi \Big) \right) \\
  &\qquad \ge \E \left( \mathbf{1}_{\pi \cap V_z(\| z \|/10) = \varnothing} \, 
    \frac{G_{V(L) \setminus \pi}(z,z)}{G_{V(L) \setminus \pi}(o,o)} \,
    \prob \Big( \pi \cap S_o[0, \xi_z] = \varnothing,\, 
    \xi_z^{S_o} < \sigma_{L}^{S_o} \,\Big|\, \pi \Big) \right).}
In the presence of the indicator, \eqref{eq:greenmono} implies 
$G_{V(L) \setminus \pi}(z,z) \ge G_{V_z(\| z \| / 10)}(z,z)$. 
Since $e \in \pi$, we also have 
\eqnst
{ G_{V(L) \setminus \pi}(o,o)
  \le G_{\Zd \setminus \{ e \}}(o,o)
  \le 2d, } 
since after each visit to $o$, the random walk next hits $e$ with probability $(2d)^{-1}$.
It follows that the right-hand side of \eqref{e:rev-green} is at least 
\eqnspl{e:bound-green}
{ &(2d)^{-1} \, G_{V_z(\| z \| / 10)}(z,z) \, 
    \E \left( \mathbf{1}_{\pi \cap V_z(\| z \|/10) = \varnothing} \,  
    \prob \Big( \pi \cap S_o[0, \xi_{z}] = \varnothing,\, 
    \xi_{z}^{S_o} < \sigma_{L}^{S_o} \,\Big|\, \pi \Big) \right) \\
  &\qquad \ge (2d)^{-1} \, G_{V_z(\| z \| / 10)}(z,z) \, \prob (\Gamma_{z,L}). }
Using Theorem \ref{thm:green_asymp} and \eqref{eq:greenmono}, we have
\[G_{V_z(\| z \| / 10)}(z,z) 
   \geq \begin{cases} 
   c \, \log |z|\, ,\quad &d = 2\\
   1 \, ,\quad &d > 2\ ,
\end{cases} \]
Inserting this estimates into \eqref{e:bound-green} completes the proof.
\end{proof}

\subsection{SRW and LERW steering} 
\label{ssec:steer}

In order to control the spanning tree event in \eqref{eq:treebd}, we need to give a lower bound on the probability of the event $\Gamma_{z,L}$ of Lemma \ref{lem:reverse_walk_z}. This will be achieved by ``steering'' the two walks so that they become well separated 
as they reach distance of order $\|z\|$, which allows us to arrange that $\pi$ avoids the box $V_z(\| z \|/10)$, and it does not influence very much the probability that $S_o$ hits $z$. In Section \ref{sssec:single-walks} we collect results we need for a single walk.
In Section \ref{sssec:sep-lem} we prove the required separation estimate. In Section \ref{ssec:lower-bound-sep} we prove a lower bound on $\prob ( \Gamma_{z,L} )$ and 
complete the proof of Theorem \ref{thm:gen_prob_bd}.

\subsubsection{Estimates for a single walk}
\label{sssec:single-walks}

The first lemma we need, shown by Masson, holds for general $d$,
and compares an infinite LERW to a finite LERW, allowing us to 
control probabilities by restricting to finite balls. Masson stated
this in the case $K \supset B(4 n)$, however, his proof applies
in the slightly more general case we use here.

\begin{lem}\cite[Corollary 4.5]{Masson}
\label{lem:seg_indep}
Let $d \geq 2$ be arbitrary. For any $\delta > 0$, we have
\[ \prob(\hS_o[0, \hsigma_n] = \alpha) 
   \asymp_\delta \prob(\hS_{o}^K[0, \hsigma_n] = \alpha), \]
for all $\alpha$, all $n \ge 1/\delta$, and all $K \supset V \big( (1 + \delta) n \big)$,
where the constants implied by the $\asymp_\delta$ notation 
only depend on $\delta$ and $d$.
\end{lem}

We will need the following ``Boundary Harnack inequality'', to control
a LERW after it has reached the boundary of a box. Estimates
of this flavour were proved in \cite[Proposition 3.5]{Masson}, 
\cite[Proposition 6.1.1]{S14}, \cite{BK15} and \cite[Section 3]{BJ15}. 
The variant we need here is a simplified version of \cite[Lemma 3.8]{BJ15}.
We define 
\[ H_n = \{ x \in \Zd : \, x \cdot \unitv_1 = n\}, \quad 
   H_n^{+} = \{ x \in \Zd : \, x \cdot \unitv_1 \geq n\}, \quad
   H_n^{-} = \{ x \in \Zd : \, x \cdot \unitv_1 \leq n\}\ . \]

\begin{lem}
\label{lem:bdry-Harnack}
There exists $c(d) > 0$ such that the following holds.
Let $\pi \subset V(n/2)$ and $x \in \partial V(n/2) \cap H_{n/2}$. 
Let $1 \le m \le n/4$, and $L \ge 4n$. We have
\eqnst
{ \prob_x \big( S(\sigma_{V_x(m)}) \in H_{n/2+m} \,\big|\, 
    \sigma_L < \oxi_{\pi} \big)
  \ge c(d) \, \frac{m}{n}. }
\end{lem}

\begin{proof}
Let $\pi' = \pi \cap V_x(m)$. It is shown in \cite[Section 3]{BJ15} that
\eqnst
{ \prob_x \big( S(\sigma_{V_x(m)}) \in H_{n/2+m},\, 
    \sigma_{V_x(m)} < \oxi_{\pi'} \big)
  \ge \frac{1}{2d} \prob_x \big( \sigma_{V_x(m)} < \oxi_{\pi'} \big). }
This yields
\eqnsplst
{ &\prob_x \big( S(\sigma_{V_x(m)}) \in H_{n/2+m},\, \sigma_L < \oxi_{\pi} \big) \\
  &\qquad \ge \prob_x \big( S(\sigma_{V_x(m)}) \in H_{n/2+m},\, 
    \sigma_{V_x(m)} < \oxi_{\pi'} \big) \\
  &\qquad\quad 
    \min_{w \in (\partial V_x(m)) \cap H_{n/2+m}} 
    \prob_w \big( \sigma_{2n} < \xi_{H_{n/2}} \big) \,
    \min_{z \in \partial V(2n)} 
    \prob_z \big( \sigma_L < \xi_{\pi} \big) \\
  &\qquad \ge \frac{1}{2d} \prob_x \big( \sigma_{V_x(m)} < \oxi_{\pi'} \big) \,
    c \, \frac{m}{n} \, \max_{z \in \partial V(2n)} 
    \prob_z \big( \sigma_L < \xi_{\pi} \big) \\
  &\qquad \ge \frac{c (m/n)}{2d} 
    \prob_x \big( \sigma_{V(2n)} < \oxi_{\pi} \big) \,
    \max_{z \in \partial V(2n)} 
    \prob_z \big( \sigma_L < \xi_{\pi} \big) 
  \ge c(d) \, \frac{m}{n} \, \prob_x \big( \sigma_L < \oxi_{\pi} \big). }
Here we used a gambler's ruin estimate and the Harnack principle in the second inequality. 
\end{proof}

\subsubsection{Separation lemma}
\label{sssec:sep-lem}

Separation lemmas for a loop-erased walk and a simple random walk 
appeared in \cite{Masson} for $d = 2$ and in \cite{S14} for $d = 3$.
We give here a unified proof that works for all $d \ge 2$.
Let us write $\hS_{L,o} = \looper S_o[0,\sigma_L]$.
Recall that $S'_x$ generally denotes a simple random walk independent of the walk $S_x$. 
We define $A_n = \{ \hS_{L,o}(0,\hsigma_n] \cap S'_o[0,\sigma_n] = \es \}$,
$n \le L$.
Let 
\eqnst
{ D_n
  = \min \Big\{ \dist \big( \hS_{L,o}(\hsigma_n), S'_o[0,\sigma_n] \big),\, 
    \dist \big( \hS_{L,o}[0,\hsigma_n], S'_o(\sigma_n) \big) \Big\}. }

\begin{lem}[\textbf{Separation Lemma}]
\label{lem:sep-lem}
Let $d \ge 2$. There exists $\delta = \delta(d) > 0$ and $c = c(d) > 0$ such that 
for all $1 \le n \le L/4$ we have
\eqnst
{ \prob ( D_n \ge \delta n \,|\, A_n )
  \ge c. }
\end{lem}

In the proof of the separation lemma, a basic step is to show that given that
non-intersection occurred to distance $n/2$, there is small probability that
separation at distance $n$ is bad.

\begin{lem}
\label{lem:basic-sep}
Let $d \ge 2$. There exists a function $r : (0,1/2] \to (0,\infty)$
with $\lim_{\delta \to 0} r(\delta) = 0$, and for all $\delta \in (0,1/2]$
there exists $n_0(\delta)$ such that we have 
we have
\eqnst
{ \prob \big( A_n \cap \{ D_n < \delta n \} \,\big|\, A_{n/2} \big)
  \le r(\delta), \qquad
  n \ge n_0(\delta), \, 0 < \delta \le 1/2. }
\end{lem}

\begin{proof}
We condition on $\pi_0 := \hS_{L,o}[0,\hsigma_{n/2}]$ and 
$S'_o[0,\sigma_{n/2}]$, and denote $x_1 = \hS_{L,o}(\hsigma_{n/2})$,
$x_2 = \hS_{L,o}(\hsigma_n)$, $y_1 = S'_o(\sigma_{n/2})$,
$y_2 = S'_o(\sigma_n)$. We distinguish two cases:\\
(a) $S'_o[\sigma_{n/2},\sigma_n]$ visits $B_{x_2}(\delta n)$; \\
(b) $\hS_{L,o}[\hsigma_{n/2},\hsigma_n]$ visits $B_{y_2}(\delta n)$. \\
We bound the probabilities of the two cases separately, showing that each
is bounded by a suitable $r(\delta)$.

\emph{Case (a).} Let us further condition on $x_2$. Since 
Brownian motion in the cube $\{ u \in \mathbb{R}^d : \| u \| \le 1 \}$ 
has continuous paths, and the path tends to its exit point, and the 
probability of any given exit point is $0$, there exists
$r_1(\delta)$, tending to $0$, such that given any boundary 
point $w$, the Brownian path intersects 
$\{ u \in \R^d : \| u \| \le 1,\, |u - w| \le 2 \delta \}$
with probability $\le r_1(\delta)$. Hence the required
bound follows from the invariance principle.

\emph{Case (b).} Let us condition on $y_2$. Due to the Domain Markov Property
(Lemma \ref{lem:dmp}), the path $\hS_{L,o}[\hsigma_{n/2},\hsigma_L]$ has the law of 
$\looper (X[0,\sigma_L])$, where $X$ has the law of
$S_{x_1}$ conditioned on the event $\{ \sigma_L < \oxi_{\pi_0} \}$.
On the event in Case (b), the path $X[\sigma_{3n/4},\sigma_L]$
has to visit $B_{y_2}(\delta n)$. Conditioning on the
point $x' = X(\sigma_{3n/4})$, the probability of this
event is 
\eqnspl{e:X-bound}
{ &\prob \big( X[\sigma_{3n/4},\sigma_L] \cap B_{\delta n}(y_2) \not= \es \,\big|\, 
     X(\sigma_{3n/4}) = x' \big) 
  = \frac{\prob_{x'} \big( \sigma_L < \xi_{\pi_0},\, 
     \xi_{B_{y_2}(\delta n)} < \sigma_L \big)}
     {\prob_{x'} \big( \sigma_L < \xi_{\pi_0} \big)}. }
When $d \ge 3$, the right hand side is at most
\eqnst
{ \frac{\prob_{x'} ( \xi_{B_{\delta n}(y_2)} < \infty )}
     {\prob_{x'} \big( \xi_{B_{n/2}(o)} = \infty \big)}
  \le C \delta^{d-2}. }
When $d = 2$, consider
\eqnsplst
{ h(w) 
  := \prob_w \left( \sigma_L < \xi_{\pi_0},\, \xi_{B_{y_2}(\delta n)} < \sigma_L \right)
  \quad \text{ and } \quad
  w^*
  := \underset{w \in V(2n) \setminus V_{y_2}(n/8)}{\operatorname{argmax}} h(w). }
Using the Harnack principle, we have
\eqnsplst
{ h(w^*)
  &\le \prob_{w^*} \left( \xi_{B_{y_2}(\delta n)} < \sigma_{4n} \right) \,
      \max_{v \in \partial V(4n)} 
      \prob_v \left( \sigma_L < \xi_{\pi_0} \right)
      + \prob_{w^*} \left( \sigma_{4n} < \xi_{\pi_0} \wedge 
      \xi_{B_{y_2}(\delta n)} \right) \,
      \max_{w \in \partial V(2n)} h(w) \\
  &\le C \, (\log 1/\delta)^{-1} \, C \, 
      \prob_{x'} \left( \sigma_L < \xi_{\pi_0} \right)
      + c_1 \, h(w^*), }
where $c_1 < 1$. Therefore, we get 
\eqnsplst
{ (1 - c_1) h(x') 
  \le (1 - c_1) h(w^*)
  \le C (\log 1/\delta)^{-1} \, \prob_{x'} \left( \sigma_L < \xi_{\pi_0} \right). }
This completes the proof.
\end{proof}

The next step 
is to show that 
given a ``good'' separation at distance $n/2$, the probability that the
paths can be ``very well'' separated at distance $n$ is at least a
constant. 
We denote $\cG_n := \sigma ( \hS^L_o[0,\hsigma_n], S'_o[0,\sigma_n] )$
the information about the paths up to the exit from $V_n$.

\begin{lem}
\label{lem:extend}
For any $0 < \delta \le 1/2$ there exists $c(\delta) > 0$ and 
$n_1(\delta)$ such that for all $n \ge n_1(\delta)$ and
$L \ge 4n$ the following hold.\\
(i) We have
\eqn{e:extend}
{ \prob \Big( A_n \cap \big\{ \hS^L_o(\hsigma_{3n/4},\hsigma_{n}) \in H^-_{-n/2}, \, 
    S'_o(\sigma_{3n/4},\sigma_{n}) \in H^+_{n/2} \big\} \,\Big|\, \cG_{n/2} \Big)
  \ge c(\delta) }
everywhere on the event $A_{n/2} \cap \{ D_{n/2} \ge \delta n/2) \}$.\\
(ii) Moreover, we may further require the event 
$\hS^L_o[\hsigma_{2n},\hsigma_L] \cap V_n = \es$ in \eqref{e:extend}.
\end{lem}

\begin{proof}
We condition on $\pi_0 = \hS^L_o[0,\hsigma_{n/2}]$ and $S'_o[0,\sigma_{n/2}]$.
Let us write $x_1 = \hS^L_o(\hsigma_{n/2})$ and $y_1 = S'_o(\sigma_{n/2})$.
Due to the conditioning in \eqref{e:extend}, we have $|x_1 - y_1| \ge \delta (n/2)$.
We write $X$ for a random walk starting at $x_1$ conditioned on the event
$\{ \sigma_L < \oxi_{\pi_0} \}$, so that $\hS^L_o[\hsigma_{n/2},\hsigma_L]$
has the law of $\looper X[0,\sigma_L]$.

An application of Lemma \ref{lem:bdry-Harnack} yields that with probability 
$\ge c \delta$, the process $X$ exits $V_{x_1}(\delta n/8)$ on the face
furthest from $V(n/2)$. Also, there is probability $\ge (2d)^{-1}$
that $S'_o[\sigma_{n/2},\infty)$ exits $V_{y_1}(\delta n/8)$ on the face
furthest from $V(n/2)$. Using the Harnack principle for $X$, and 
appropriate disjoint corridors of width of order $\delta n$ for the 
LERW and the SRW, respectively (see Figure \ref{fig:fig_extend} below), there is probability $\ge c(\delta)$ that:\\
(a) $S'_o$ exits $V(n)$ in $H_n$, with the appropriate portion in the 
required halfspace;\\
(b) $X$ exits $V(2n)$ in $H_{-2n}$ with 
$X[\sigma_{3n/4},\sigma_{2n}] \subset [-2n,-n/2] \times [-3n/4,3n/4]^{d-1}$;\\
(c) $S'_o[0,\sigma_n] \cap X[0,\sigma_n] = \varnothing$.\\
In order to further ensure that $\looper X$ first exits $V(n)$ at a point in 
$H_{-n}$, and that $A_n$ occurs, we show that for $w \in \partial V(2n)$ we have
\eqn{e:X-not-return}
{ \prob_w \left( \sigma^X_L < \xi^X_{V(n)} \right)
  \ge c > 0. }
This is indeed sufficient, since the events in point (b) and \eqref{e:X-not-return}
imply that the last visit of $X$ to $\partial V(n)$ must occur at a point
in $H_{-n}$. In order to see \eqref{e:X-not-return}, first note that the
statement is clear for $d \ge 3$, since then
\eqnst
{ \prob_w \big( \sigma_L < \xi_{V(n)} \big)
  \ge \prob_w \big( \xi_{V(n)} = \infty \big)
  \ge c 
  \ge c \, \prob_w ( \sigma_L < \xi_{\pi_0} \big). }
When $d = 2$, let $w_* := \underset{w \in \partial V(2n)}{\operatorname{argmax}} \, 
\prob_w \big( \sigma_L < \xi_{\pi_0} \big)$. Then we have
\eqnsplst
{ \prob_{w_*} \big( \sigma_L < \xi_{\pi_0} \big)
  &\le \prob_{w_*} \big( \sigma_L < \xi_{V(n)} \big)
     + \max_{y \in \partial V(n)} \prob_y \big( \sigma_{2n} < \xi_{\pi_0} \big)
       \, \prob_{w_*} \big( \sigma_L < \xi_{\pi_0} \big). }
The maximum over $y$ is $\le c_2 < 1$, which after rearranging gives 
\eqnst
{ \prob_{w_*} \big( \sigma_L < \xi_{\pi_0} \big)
  \le \frac{1}{1-c_2} \, \prob_{w_*} \big( \sigma_L < \xi_{V(n)} \big). }
For other $w \in \partial V(2n)$ we obtain the statement from the Harnack
principle. 

The construction we gave ensures both the event in (i) and (ii), and thus 
completes the proof of the lemma.
\end{proof}

\begin{figure}
\setlength{\unitlength}{1cm}
\begin{minipage}[c]{0.05\textwidth}
\hskip0.5cm
\end{minipage}
\begin{minipage}[c]{0.4\textwidth}
\begin{picture}(5,6)(-2,-2)
  \linethickness{0.3mm}
  \put(0,0){\line(1,0){1}}
  \put(1,0){\line(0,1){1}}
  \put(1,1){\line(-1,0){1}}
  \put(0,1){\line(0,-1){1}}
  
  \put(-2,-2){\line(1,0){5}}
  \put(3,-2){\line(0,1){5}}
  \put(3,3){\line(-1,0){5}}
  \put(-2,3){\line(0,-1){5}}

  \linethickness{0.1mm}
  \put(2,0.25){\dashbox{0.1}(0,2.75){}} 
  \put(2,-2){\dashbox{0.1}(0,2){}}      
  \put(1,2){\dashbox{0.05}(0,1){}}      
  \put(1,1){\dashbox{0.05}(0,0.75){}}   
  \put(1,-2){\dashbox{0.05}(0,2){}}     

  \setlength{\fboxsep}{0pt}
  \put(1,0){\colorbox[gray]{0.5}{\makebox(2,0.25){}}}
  \put(1,0){\line(1,0){2}}
  \put(3,0){\line(0,1){0.25}}
  \put(3,0.25){\line(-1,0){2}}
  \put(1,0.25){\line(0,-1){0.25}}

  \put(1,0.5){\colorbox[gray]{0.5}{\makebox(0.75,0.25){}}}
  \put(1.5,0.5){\colorbox[gray]{0.5}{\makebox(0.25,1.5){}}}
  \put(-2,1.75){\colorbox[gray]{0.5}{\makebox(3.75,0.25){}}}
  \put(1,0.5){\line(1,0){0.75}}
  \put(1.75,0.5){\line(0,1){1.5}}
  \put(1.75, 2){\line(-1,0){3.75}}
  \put(-2, 2){\line(0,-1){0.25}}
  \put(-2, 1.75){\line(1,0){3.5}}
  \put(1.5, 1.75){\line(0,-1){1}}
  \put(1.5, 0.75){\line(-1,0){0.5}}
  \put(1, 0.75){\line(0,-1){0.25}}

  \put(-2.5,-2.5){\large $V(n)$}
  \put(-1,-0.5){\large $V(n/2)$}
  \put(1.5,3.25){\large $H_{3n/4}$}
\end{picture}
\end{minipage}\hfill
\begin{minipage}[c]{0.55\textwidth}
\caption{Depiction of the extension in disjoint corridors. Each walk is 
assumed to exit $V_{n/2}$ at the center of one of the corridors, which are 
distance $\sim \delta n$ apart. Note the placement of the corridors ensures 
that terminal segments of each walk remain in appropriate half-spaces as 
in \eqref{e:extend}.}\label{fig:fig_extend}
\end{minipage}
\end{figure}

We are ready to prove the Separation Lemma. The argument we give is inspired by
\cite{kestenscaling}.

\begin{proof}[Proof of Lemma \ref{lem:sep-lem}]
Let us write 
\eqnst
{ f(n)
  = \prob \big( A_n \big) \qquad \text{ and } \qquad
  g(n)
  = \prob \big( A_n \cap \{ D_n \ge \delta n \} \big). }
Let $\delta > 0$ that we will choose later. 
Let $n \ge \max \{ n_0(\delta), n_1(\delta) \}$, where $n_0$ and $n_1$ 
are the constants from Lemmas \ref{lem:basic-sep} and \ref{lem:extend}.
Lemma \ref{lem:basic-sep} implies 
\eqnspl{e:f(n)-ineq}
{ f(n)
  &= g(n) + f(n/2) \,
    \prob \big( A_n \cap \{ D_n < \delta n \} \,\big|\, A_{n/2} \big) 
  \le g(n) + f(n/2) \, r(\delta) \\
  &\le \sum_{\ell=0}^{k-1} r(\delta)^\ell \, g(n/2^\ell)
     + r(\delta)^k \, f(n/2^k). }
Lemma \ref{lem:extend} implies that on the event
$A_{n/2^\ell} \cap \{ D_{n/2^\ell} \ge \delta n/2^\ell \}$,
we can extend both the loop-erased walk and the random walk
to opposite faces of $\partial V(n/2^{\ell-1})$, with probability
at least $c(\delta)$. A gambler's ruin estimate then implies that, there is 
probability $\ge (c/2^{\ell})^2$ that the walks reach $\partial V(n)$
without intersecting. This shows that $g(n/2^\ell) \le c(\delta) 2^{2 \ell} g(n)$.
Substituting this into \eqref{e:f(n)-ineq} yields
\eqn{e:f(n)-bound}
{ f(n)
  \le g(n) \, \left[ 1 + c(\delta) \sum_{\ell=1}^{k-1} (4 \, r(\delta))^\ell \right]
     + r(\delta)^k \, f(n/2^k). }
Choose $\delta > 0$ so that $4 \, r(\delta) < 1/2$, and the smallest $k$ 
such that $\max \{ n_0(\delta_0), n_1(\delta) \} \le n/2^{k}$. Then \eqref{e:f(n)-bound} implies $f(n) \le C(\delta) \, g(n)$.
\end{proof}

\subsubsection{Estimate on $\Gamma_{z,L}$}
\label{sssec:Gamma_z,L}

In this section, we prove a lower bound on the probability of the event
$\Gamma_{z,L}$, which yields a finite volume analogue of 
Theorem \ref{thm:gen_prob_bd}. 
In the following, let 
\[\Es^L(n) := \prob \big( \hS^L_o(0,\hsigma_{n}] 
\cap S'_o[0,\sigma_n] = \varnothing \big)\ ,\]
which is an (expectation of an) escape probability for the walk $S'_o$.
\begin{lem}
\label{lem:Gamma_z,L}
Let $d \ge 2$. There exists $c = c(d) > 0$ such that for any 
$z \in \Z^d \setminus \{ o \}$ and $L > 4 \| z \|$ we have
\eqn{e:Gamma_z,L-bound}
{ \prob \big( \Gamma_{z,L} \big)
  \ge \begin{cases}
      c \, \Es^L(\| z \|) \,
      \prob_o \big( \xi_{\{ z \}} < \sigma_{4 \| z \|} \big) & \text{if $d \ge 3$;}\\
      c \, (\log L)^{-1} \, \Es^L(\| z \|) \,
      \prob_o \big( \xi_{\{ z \}} < \sigma_{4 \| z \|} \big) & \text{if $d = 2$.} 
      \end{cases} } 
\end{lem}

\begin{proof}
We may assume that $\| z \|$ is sufficiently large. 
Without loss of generality, we also assume that the
first coordinate of $z$ is positive and has maximal 
absolute value among all coordinates.

We first require the occurrence of the event 
\eqn{e:pi-avoid}
{ B(e)
  := \big\{ \pi \cap \{ o \} = \varnothing \big\}. }
We have 
\eqn{e:pi-avoid-o}
{ q_L 
  := \prob ( B(e) )
  = \prob_e \big( \sigma_L < \xi_{o} \big) 
  \ge \begin{cases}
     c (\log L)^{-1} & \text{when $d = 2$;} \\
     c & \text{when $d \ge 3$.}
     \end{cases} } 
Conditional on $B(e)$, the law of $\pi$ is the same as the law of 
$\hS^L_o[1,\hsigma_L]$ conditional on $\hS^L_o(1) = e$. Therefore, 
we will express properties of $\pi$ conditional on the event $B(e)$
in terms of the properties of $\hS^L_o[1,\hsigma_L]$ conditional on 
$\hS^L_o(1) = e$.

Let us require the occurrence of the event 
\eqn{e:avoid-path}
{ A_{\| z \|/4} \cap \big\{ D_{\| z \|/4} \ge \delta \| z \|/4 \big\}, }
where $\delta = \delta_0$ is the constant chosen in Lemma \ref{lem:sep-lem}. 
According to that lemma, the event in \eqref{e:avoid-path} has unconditional 
probability $\ge c \Es^L(\| z \|/4) \ge c \Es^L( \| z \| )$. 
Due to $\Z^d$-symmetry, the conditional probability of \eqref{e:avoid-path}
given $\hS^L(1) = e$ is the same as the unconditional probability. 

Let us further require the event in Lemma \ref{lem:extend} with $n = \| z \|/2$, 
that is, that the paths extend disjointly to opposite faces of $V(\| z \|/2)$,
with the random walk landing on $H_{\| z \|/2}$, and the LERW landing on 
$H_{-\| z \|/2}$. According to Lemma \ref{lem:extend}, this happens with 
conditional probability $\ge c$. It follows from Lemma \ref{lem:extend}, 
that there is probability bounded away from $0$ that the LERW can be 
further extended to land on $H_{-8 \| z \|}$, in such a way that 
$\pi$ is contained in $H^-_{3 \| z \|/8} \cup V(4 \| z \|)^c$. Since 
$S'_o(\sigma_{\| z \|/2}) \in H_{\| z \|/2}$, the conditional probability,
given $\pi$ and $S'_o[0,\sigma_{\| z \|/2}]$ that $S'_o$ hits $z$ before
$\xi_{\pi} \wedge \sigma_{4 \| z \|}$ is $\ge c \, |z|^{2-d}$ when $d \ge 3$,
and $\ge (\log |z|)^{-1}$ when $d = 2$. Combining the estimates for 
each part of the construction yields:
\eqnst
{ \prob \big( \Gamma_{z,L} \big)
  \ge \begin{cases}
      c \, \Es^L(\| z \|) \, |z|^{2-d} & \text{when $d \ge 3$;} \\
      c \, (\log L)^{-1} \, \Es^L(\| z \|) \, \frac{1}{\log |z|} &\text{when $d = 2$.}
      \end{cases} }
This completes the proof of the lemma when $\| z \|$ is sufficiently large.
\end{proof}

The following proposition summarizes the result of the finite $L$
arguments we made in Sections \ref{ssec:prelim}--\ref{ssec:steer}.

\begin{prop}
Let $d \ge 2$. For any $z \in \Z^d$ and all $L \ge 4 \| z \|$ we have 
\eqnspl{e:toppling-lbd-finite}
{ \nu_L (z \in \av) 
  &\geq c \, \Es^L(\| z \|) \, G_{V_z(\| z \|/10)}(z,z) \, 
       \prob_o \big( \xi_z < \sigma_{4 \| z \|} \big) \\
  &\asymp \Es^{\| z \|} (\| z \|) \, \prob_o ( \xi_z < \infty ) \ .}
\end{prop}

\begin{proof}
Combining \eqref{eq:treebd}, Lemma \ref{lem:rwprob}, Lemma \ref{lem:reverse_walk_z},
and Lemma \ref{lem:Gamma_z,L}, we obtain the statement in both $d \ge 3$ and 
$d = 2$. Lemma \ref{lem:seg_indep} implies that $\Es^L(\| z \|)$ is 
comparable to $\Es^{4 \| z \|} (\| z \|)$. Masson \cite[Lemma 5.1]{Masson} showed
that the latter is comparable to $\Es^{\| z \|} (\| z \|)$.
\end{proof}

\subsection{Proof of Theorem \ref{thm:topple}}
\label{ssec:lower-bound-sep}

In passing to the limit $L \to \infty$, we use the following proposition.
We prove only the $d \geq 3$ case here; the proof of the more technical $d = 2$ case is deferred to the end of
Section \ref{sec:last-k-waves}; see Lemma \ref{lem:approx} there.

\begin{prop}
\label{prop:lim-toppling} \ \\
Assume $d \ge 2$. Then we have 
$\nu ( z \in \av ) = \lim_{L \to \infty} \nu_L ( z \in \av )$.
\end{prop}

\begin{proof}[{Proof of Proposition \ref{prop:lim-toppling}, $d \geq 3$ case}]
Due to \cite[Theorem 3.11]{JR08} we have $\nu ( |\av| < \infty ) = 1$.
Therefore, given $\eps > 0$, we can find $|z| < M < \infty$ such that 
$\nu (\av \subset V(M)) > 1 - \eps$. Due to the weak convergence 
$\nu_L \to \nu$, there exists $M < L_0 < \infty$ such that for all
$L \ge L_0$ we have
\eqnsplst
{ \left| \nu ( z \in \av ) - \nu_L ( z \in \av ) \right| 
  &\le \left| \nu ( z \in \av ) - \nu ( z \in \av,\, \av \subset V(M) ) \right| \\
  &\qquad + \left| \nu( z \in \av,\, \av \subset V(M) ) - 
        \nu_L ( z \in \av,\, \av \subset V(M) ) \right| \\
  &\qquad + \left| \nu_L ( z \in \av,\, \av \subset V(M) ) - 
        \nu_L ( z \in \av ) \right| \\
  &< 3 \eps. }
This completes the proof.
\end{proof}

We can now complete the proof of Theorem \ref{thm:gen_prob_bd}. 

\begin{proof}[Proof of Theorem \ref{thm:gen_prob_bd}]
Lemma \ref{lem:seg_indep} implies that $\Es^L(\| z \|)$ in 
Proposition \ref{prop:lim-toppling} is comparable to the 
avoidance probability in the theorem, so the statement follows 
using Proposition \ref{prop:lim-toppling}.
\end{proof}

For later use we state here a corollary of the construction.

\begin{cor}
\label{cor:sep_bdd}
Let $d \ge 2$. There exists a constant $c = c(d) > 0$ such that the following holds. For each $z \in \Zd$ and $L \ge 4 \|z\|$ and $e \sim o$, with $\pi$ denoting the path in
$\frT_{L,s}$ connecting $e$ to $s$, we have
\[ \mu_{L,o} \big( z \in \frT_{L,o}, \, \pi \cap V_z(\|z\|/10) = \varnothing \,\big|\, 
      e \notin \frT_{L,o} \big) 
   \geq c \, \prob \big( S_o[0, \sigma_{|z|}] \cap 
      \hS'_o(0, \hsigma_{|z|}] = \varnothing \big) \,
      \prob_o \big( \xi_{z} < \infty \big)\ . \] 
\end{cor}

We now turn to the proofs of the explicit bounds for $2 \le d \le 4$.

\subsubsection{Proof of Theorem \ref{thm:topple} when $d = 2$}
\label{sssec:2D}

By Theorem \ref{thm:gen_prob_bd}, Theorem \ref{thm:topple}(i) will hold once we know
\begin{equation}
\label{eq:basicallymasson}
\mathbb{P}\left(S'_o[0, \sigma_n] \cap \hS(0,\hsigma_n] = \varnothing\right) 
\geq n^{-3/4+ o(1)}.
\end{equation}
The exponent $3/4 + o(1)$ was first proved by Kenyon \cite{K00}, who stated it
for simple random walk in the half plane. A proof for more general walks was given by Masson, who derived it from results on $\SLE_2$ \cite[Theorem 5.7]{Masson}. He established the analogue of \eqref{eq:basicallymasson} for a SRW and a finite LERW -- via \cite[Lemma 5.1]{Masson}, where the intersection probabilities for finite and infinite LERW are related. This implies Theorem \ref{thm:topple}(i).
\qed


\subsubsection{Proof of Theorem \ref{thm:topple} when $d = 3$}
\label{sssec:3D}

In this section, we complete the proof of the explicit lower bound in Theorem \ref{thm:topple}(ii) by showing that $\prob(S'_o[0, \sigma(n)] \cap \hS_o(0, \hsigma(n)] = \varnothing) \geq c n^{-2 \zeta}$ (this suffices, by the previously proven Theorem \ref{thm:gen_prob_bd}). Since $\hS_o$ is the loop-erasure of $S$, it is enough to show
\begin{equation}
\label{eq:3dreplaced}
\prob \big( S'_o[0, \sigma(n)] \cap S(0, \infty) = \varnothing \big) 
  \geq c n^{-2 \zeta}\ .
\end{equation}

This is a simple adaptation of the results of \cite{Lawcut}. Indeed, there 
exists a $c_1 > 0$ such that (uniformly in $m$)
\begin{align*} 
\prob( S'_o[0, m] \cap S(0, \infty) = \varnothing) \geq c_1 m^{-\zeta}\\
\end{align*}
by \cite[discussion after (3)]{Lawcut}. On the other hand, by \cite[Lemma 4.7]{Lawcut}, there exists $C_2, c_2 > 0$ such that (uniformly in $a, n$)
\begin{align*}
\prob(S'_o[0, \sigma_n] \cap S(0, \infty) = \varnothing,\, \sigma_n > a n^2) 
\leq C_2 \exp(-a / c_2) n^{-2 \zeta}\ .
\end{align*}
Choosing $a$ sufficiently large (relative to $c_1, \, c_2$) and $m = a n^2$ completes the proof of \eqref{eq:3dreplaced}. From this Theorem \ref{thm:topple}(ii) follows. 
\qed

\subsubsection{Proof of Theorem \ref{thm:topple} when $d = 4$}
\label{sssec:4D}

In four dimensions, the avoidance probability was determined by Lawler 
\cite{Lawler95}, who showed that 
\eqnst
{ \Es(n)
  \asymp (\log n)^{-1/3}. }
This and Theorem \ref{thm:gen_prob_bd} yield Theorem \ref{thm:topple}(iii).

\section{Low-dimensional radius bounds}
\label{sec:radius-low}

In this section we prove the radius bounds stated in Theorem \ref{thm:radius}(i)--(iii) for dimensions $2 \le d \le 4$. We start with the lower bounds.

\begin{proof}[Proof of Theorem \ref{thm:radius}(i)--(iii), lower bounds]
Observe that
\begin{equation}
\label{eq:rad_trivial}
  \nu \left( R \geq r \right) 
  \geq \nu ( r \unitv_1 \in \av )\ .
\end{equation}
Therefore, the claimed lower bounds follow immediately from Theorem \ref{thm:topple}(i)--(iii).
\end{proof}




\begin{proof}[Proof of Theorem \ref{thm:radius}(ii)--(iii), upper bounds]
Let us first consider a finite volume $V(L)$. Recall that for $\eta \in \cR_L$
we write $\alpha(\eta) = (\eta_1, \dots, \eta_N)$ for the waves in the 
stabilization $S_o(\eta + \mathbf{1}_{o})$. Let us extend the notation 
for the radius to waves and two-component spanning trees in the natural way:
\eqnsplst
{ R(\eta_*)
  &= \sup \{ |z| : z \in \cW(\eta_*) \} \\
  R(T_o)
  &= \sup \{ |z| : z \in T_o \} } 
Under the bijection of Section \ref{ssec:newspan} these two notions coincide.
We have
\eqnsplst
{ \big| \big\{ \eta \in \cR_L : R(\eta) > r \big\} \big|
  &= \big| \big\{ \eta \in \cR_L : \text{$R(\eta_i) > r$ for some 
     $1 \le i \le N(\eta)$} \big\} \big| \\
  &\le \big| \big\{ \eta_* \in \cR'_L \setminus \cR_L : R(\eta_*) > r \big\} \big|. } 
Hence we get
\eqnst
{ \nu_L \big( R > r \big)
  \le \frac{| \cR'_L \setminus \cR_L |}{|\cR_L|} \,
    \mu_{L,o} \big( R(\frT_{L,o}) > r \big)
  = g_L(o,o) \, \mu_{L,o} \big( R(\frT_{L,o}) > r \big), }
where the last equality uses Lemma \ref{lem:wave-to-recur}.

Since $\{ R > r \}$ and $\{ R(\frT_{L,o}) > r \}$ are both cylinder events,
we can take the limit $L \to \infty$ on both sides to get
\eqn{e:R-upper-wsfo}
{ \nu \big( R > r \big)
  \le C(d) \, \wsfo \big( R(\frT_{L,o}) > r \big). }
Lyons, Morris and Schramm \cite{LMS} proved that
\eqnst
{ \wsfo \big( R(\frT_{L,o}) > r \big)
  \le C(d) \, r^{-\beta_d}, \quad d \ge 3, }
with $\beta_d = \frac{1}{2} - \frac{1}{d}$. 
Inserting this into \eqref{e:R-upper-wsfo} yields the upper
bounds $C \, r^{-1/6}$ and $C \, r^{-1/4}$ in dimensions 
$d = 3$ and $d = 4$, respectively.
\end{proof}

%

\section{The last $k$ waves in 2D}
\label{sec:last-k-waves}

In this section we prove that for any $1 \le k < \infty$, as $L \to \infty$, the last $k$ waves on $G_L$ (when they exist) have a weak limit. We introduce some notation for the
last $k$ waves. Recall that for $\eta \in \cR_L$, we denote by $N = N(\eta) = n_L(o,o)$ the number of times $o$ topples during the stabilization 
of $\eta + \mathbf{1}_o$. Recall from Section \ref{ssec:interm}, that given $\eta \in \recur_L$, we write $\alpha(\eta) = (\eta_1, \dots, \eta_N)$ for the 
set of intermediate configurations (right before each wave). Let
\eqnst
{ \xi^L_0 
  = \xi^L_0(\eta)
  := a_{o,L} \eta} 
and whenever $N(\eta) \ge k$, define 
\eqnst
{ \xi^L_k 
  = \xi^L_k(\eta)
  := \eta_{N-k+1}. } 
Observe that when $\xi^L_k$ is defined, we have 
$\xi^L_{k-1} = a'_{o,L} \xi^L_k $. 
Analogously to the finite graph case, $a'_{o, \Z^2}$ is the operator which adds a particle at the origin and stabilizes the resulting configuration in the graph $(\Z^2)'$.

\begin{thm}
\label{thm:last-k-waves}
Assume $d = 2$. We have:\\
(i) For every $1 \le k < \infty$ the limit $b_k = \lim_{L \to \infty} \nu_{L}(N \ge k)$ exists.\\
(ii) For every $1 \le k < \infty$, the law of $\xi^L_k$ under the measure
$\nu_{L} ( \cdot \,|\, N \ge k )$ converges weakly to the law $\rho_k$ of 
a configuration $\xi_k$.\\
(iii) The configuration $\xi_k + \mathbf{1}_o$ can be stabilized in $(\Z^2)'$ 
with finitely many topplings $\rho_k$-a.s.\\
(iv) The transformation $\xi_k \mapsto a'_{o,\Z^2}\xi_k$ is 
measure-preserving between $b_k \rho_k$ and the restriction of $b_{k-1} \rho_{k-1}$ 
to the image. (Here $b_0 = 1$, $\rho_0 := \nu$.) \\
(v) With $\nu$-probability $1$ on the event $\{ N = k \}$, all $k$ waves are 
finite, and we have $\nu ( N = k ) = b_k - b_{k+1}$.
\end{thm}

\begin{rem}
It is not difficult to construct, for every $1 \le k < \infty$, an explicit finite 
configuration around $o$ showing that $\liminf_{L \to \infty} \nu_{L} ( N = k ) > 0$.
\end{rem}

Recall the bijection for intermediate configurations from Section \ref{ssec:newspan}.
This bijection was between $\eta_* \in \recur'_L \setminus \recur_L$ and the set of spanning forests of $G_L$ with two components $T_o = T_o(\eta_*)$ and $T_s = T_s(\eta_*)$, where $o \in T_o$ and $s \in T_s$. Recall the following property of the bijection from Section \ref{ssec:newspan}, rephrased for the configurations $\xi^L_k$. 
\eqn{e:path-property}
{ \parbox{14.5cm}{If 
    there is a path in $T_o(\xi^L_k(\eta))$ from $o$ to a vertex $x$
    that stays inside $B(r)$, then starting from 
    $\xi^L_k + \mathbf{1}_o$ there is a sequence of topplings
    in $B(r)$ that topples $x$.} }
We write $\frT_{L,o,k} = T_o(\xi^L_k)$ and $\frT_{L,s,k} = T_s(\xi^L_k)$ for short. 
When $x \in \frT_{L,o,k}$, we denote $\pi_{L,k}(x)$ the unique self-avoiding
path in $\frT_{L,o,k}$ from $o$ to $x$.

Our control on the size of waves will be in terms of the following
random variables.
\eqnspl{e:Rk-in-etc}
{ R^L_{\rin,k} 
  &= \sup \{ r \ge 0 : B(r) \subset \frT_{L,o,k} \}, \quad k \ge 1; \\
  R^L_{\rout, k}
  &= \inf \{ r \ge 0 : \frT_{L,o,k} \subset B(r) \}, \quad k \ge 1; \\
  P^L_{k}
  &= \inf \left\{ r \ge 0 : \text{$\pi_{L,k}(x) \subset B(r)$ for all 
     $x \in B \left( R^L_{\rin,(k-1)} + 1 \right)$} \right\}, \quad k \ge 2. }
All quantities in \eqref{e:Rk-in-etc} are defined to be $0$ when $N < k$.
The following lemma states a basic inequality we will need.

\begin{lem}
\label{lem:inequalities}
We have:\\
(i) $R^L_{\rin, 1} = 0$. \\
(ii) $R^L_{\rin, k} \le P^L_{k}$, $k \ge 2$. 
\end{lem}

\begin{proof}
(i) This follows directly from Corollary \ref{lem:num-last-waves}.

(ii) 
Write $r = P^L_{k}$ for short, and assume for a proof by contradiction, that we had $R^L_{\rin,k} \ge r + 1$. This implies that in the stabilization $a'_{o,L}(\xi^L_k)$ all vertices
in $B(r+1)$ topple, and hence $(\xi^L_{k-1})_{B(r)} = (\xi^L_k)_{B(r)}$.
Starting from the configuration $\xi^L_{k-1} + \mathbf{1}_o$, let us topple all sites
in $B(r)$ we can. The definition of $P^L_{k}$ and property \eqref{e:path-property} 
imply that all vertices in $B(R^L_{\rin, (k-1)} + 1)$ topple. However, this contradicts 
the definition of $R^L_{\rin, (k-1)}$.
\end{proof}

In the following proposition we show that the in-radius of the
$k$-th last wave is tight, with a power law upper bound on the tail.
\begin{prop}
\label{prop:in-rad-tight}
There exist constants $\alpha'_1 > \alpha'_2 > \dots > 0$ and
$C_1, C_2, \dots$ such that 
\eqn{e:power-wave-k}
{ \limsup_{L \to \infty} \nu_L \big( R^L_{\rin, k} > r \big)
  \le C_k r^{-\alpha'_k}, \quad \forall r \ge 1,\, \forall k \ge 1. }
In particular, for all $1 \le k < \infty$, the sequence 
$\{ R^L_{\rin,k} \}_{L \ge 1}$ is tight.
\end{prop}

\begin{proof}
We prove the statement by induction on $k$. The case $k = 1$ holds trivially
due to Lemma \ref{lem:inequalities}(i). Assume $k \ge 2$, and that 
\eqref{e:power-wave-k} holds for $k-1$. Let $1 \le r_0 < \infty$ be fixed, 
and find $L_0 = L_0(r_0) < \infty$ such that for $L \ge L_0$ we have 
\eqn{e:wave-(k-1)}
{ \nu_{L}( R^L_{\rin, (k-1)} > r_0 ) 
  \le 2 C_{k-1} r_0^{-\alpha'_{k-1}}. } 
It is sufficient to bound $R^L_{\rin,k}$ when the event 
$\{ R^L_{\rin, (k-1)} \le r_0 \}$ occurs, and due to Lemma \ref{lem:inequalities}(ii), it is enough to bound $P^L_k$ on this event. 
In what follows, we assume the event $\{ R^L_{\rin, (k-1)} \le r_0 \}$.

For $\ell \ge 1$ we are going to bound the probability
that $r_0 2^\ell < P^L_{k} \le r_0 2^{\ell+1}$. Due to Lemma \ref{lem:inequalities}(ii),
this event implies that $(\frT_{L,o,k}, \frT_{L,s,k})$ belongs to the following event 
$E(x,r_0,\ell)$ for some $x \in \partial B(r_0 + 1)$:
\eqnst
{ E(x,r_0,\ell)
  = \left\{ (T_o,T_s) \in \ST_{L,o} : 
     \text{$x \in T_{o}$, 
     $\pi(x)$ visits $B(r_0 2^{\ell})^c$ and 
     $T_{s} \cap \partial B(r_0 2^{\ell+1}) \not= \varnothing$} \right\}, }
where $\pi(x)$ is the path in $T_o$ from $x$ to $o$. 
Therefore, using Corollary \ref{cor:wave-to-recur-Zd}(ii), we have
\eqnspl{e:Pk-bound}
{ \nu_{L} \left( R^L_{\rin, (k-1)} \le r_0,\, P^L_k > r_0 2^{\ell_0} \right)
  &\le \sum_{\ell \ge \ell_0} \frac{\big| \{ \eta \in \recur_{L} : 
      r_0 2^{\ell} < P^L_{k}(\eta) \le r_0 2^{\ell+1} \} \big|}{|\recur_{L}|} \\
  &\le C \, (\log L) \, \sum_{\ell \ge \ell_0} \, \sum_{x \in \partial B(r_0 + 1)} \, 
      \mu_{L,o} ( E(x,r_0,\ell) ). }
We use Wilson's algorithm to get an upper bound on the probability of $E(x,r_0,\ell)$. 
Let the first random walk start at $x$. Let $\tau$ be the time of the last visit, before $\xi_{\{ o \}}$, to a vertex in 
$\partial B(r_0 2^{\ell})$. Let us condition on the path 
$S_x[0,\tau]$. Let $\gamma = \looper S_x[0,\tau]$,
and let $\gamma_0$ be the initial segment of $\gamma$ from $x$ to the first visit of $\gamma$ to $\partial B(r_0 2^{\ell})$. The walk 
\eqnst
{ S'(m)
  = S_x(\tau+m), \quad m = 0, \dots, \xi_{o} - \tau; }
is a simple random walk on $\Z^2$ conditioned on 
$\xi_{o} < \oxi_{B(r_0 2^{\ell})^c}$.
On the event $E(x,r_0,\ell)$, $S'$ cannot hit 
$\gamma_0$, so we bound the probability that 
$S'$ hits $B(r_0+1)$ before $\gamma_0$. 

The walk $S'$ has to successively cross from $\partial B(r_0 2^q)$
to $B(r_0 2^{q-1})$ for $q = \ell-1, \dots, 1$. During each crossing, 
it has a fixed constant probability of hitting $\gamma_0$, since 
this holds for simple random walk, and the Harnack principle \cite{Lawlim}
then implies it holds for $S'$. 
Hence the probability that $S'$ reaches $B(r_0+1)$
before hitting $\gamma_0$ is less than $(1 - c_1)^{\ell-1}$ for some
$0 < c_1 < 1$. This bounds from above the probability that 
$x \in \frT_{L,o}$ and $\pi(x)$ visits 
$B(r_0 2^{\ell})^c$. Assuming that this event occurs, 
we now bound the conditional probability that $\frT_{L,s}$ contains a 
vertex $y \in \partial B(r_0 2^{\ell+1})$. For this, continue 
Wilson's algorithm with a walk $S_{y_0}$ starting at any 
$y_0 \in \partial B \left( (r_0 2^{\ell+1})^4 \right)$,
followed by walks starting at $y_1, \dots, y_M$, where
the latter is an enumeration of all vertices in 
$\partial B(r_0 2^{\ell+1})$. Let $\frF_j$ denote the
tree generated by the walks $S_x, S_{y_0}, \dots, S_{y_j}$.
Denote 
\eqnst
{ E_j
  = \left\{ \sigma^{S_{y_j}}_L < \xi^{S_{y_j}}_{\frF_{j-1}} \right\}, 
    \quad j = 0, 1, \dots, M, }
where $\frF_{-1} := \pi(x)$. Then on the event 
$\{ x \in \frT_{L,o},\, \pi(x) \cap B(r_0 2^\ell)^c \not= \varnothing \}$ we have
\eqnsplst
{ \mu_{L,o} \big( \frT_{L,s} \cap \partial B(r_0 2^{\ell+1}) \not= \varnothing \,\big|\, 
     \frF_{-1} \big)
  &\le \sum_{j=0}^M \E \left( \prob ( E_j \,|\, \frF_{j-1} ) 
      \mathbf{1}_{E_0^c} \dots \mathbf{1}_{E_{j-1}^c} \right). }
Application of Theorem \ref{thm:green_asymp}(ii) yields that the $j = 0$ term is at most $C \, \log \, (r_0 2^{\ell+1})^4 / \log L$
when $L > (r_0 2^{\ell+1})^4$. 
For $1 \le j \le M$, using Beurling's estimate (Lemma \ref{lem:Beurling}),
on the event $E_0^c \cap \dots E_{j-1}^c$ we have 
\eqnsplst
{ \prob ( E_j \,|\, \frF_{j-1} ) 
  &\le \prob_{y_j} \left( \sigma_{B( (r_0 2^{\ell+1})^4 )} < \xi_{\frF_0} \right)
     \, \max_{w \in \partial B( (r_0 2^{\ell+1})^4 )} 
     \prob_w \left( \sigma_L < \xi_o \right) \\
  &\le C (r_0 2^{\ell+1})^{-1} \, (\log \, r_0 2^{\ell+1}) / \log L. }
Since $M \le C r_0 2^{\ell+1}$, putting the $j = 0$ and $1 \le j \le M$ cases 
together we get that the sum over $0 \le j \le M$ is bounded by 
$C \, (\log \, r_0 2^{\ell+1}) / \log L$. 
Together with the earlier bound on $\pi(x)$ leaving $B_o(r_0 2^{\ell})$ this gives
\eqnst
{ \prob \big( E(x, r_0, \ell) \big)
  \le C \, ( \log \, r_0 2^{\ell+1} ) \, \frac{(1 - c_1)^\ell}{\log L}. }
Substituting into \eqref{e:Pk-bound}, and summing over $\ell \ge \ell_0$ implies, for $L$ sufficiently large that 
\eqnspl{e:Pk-bound2}
{ \nu_{L} \big( R^L_{\rin, (k-1)} \le r_0,\, P^L_{k} > r_0 2^{\ell_0} \big)
  \le C r_0 (\log \, r_0 2^{\ell_0}) \, (1 - c_1)^{\ell_0}. }
We apply \eqref{e:Pk-bound2} with $2^{\ell_0} = r_0^{\beta'}$, for some $\beta' > 0$. 
The expressions in the right hand sides of \eqref{e:Pk-bound2} and \eqref{e:wave-(k-1)} are of equal order (up to logarithms), when
\eqnst
{ \beta' 
  = - (1 + \alpha'_{k-1}) \frac{\log 2}{\log (1 - c_1)}. } 
Since $R^{(k)}_\rin \le P^{(k)}$, the bounds \eqref{e:wave-(k-1)} and \eqref{e:Pk-bound2}
imply, for all large enough $L$, that 
\eqnst
{ \nu_L ( R^L_{\rin,k} > r_0^{1+\beta'} )
   \le \nu_L (R^L_{\rin,k-1} > r_0 ) 
       + \nu_L (R^L_{\rin,k-1} \le r_0,\, P^{(k)} > r_0^{1+\beta'} )
   \le C_k \, (\log r_0 ) \, r_0^{- \alpha'_{k-1}}. }
Hence we get \eqref{e:power-wave-k} for $k$ with a choice of 
$0 < \alpha'_k < \alpha'_{k-1}/(1+\beta')$.
\end{proof}

We next prove that the out-radius of the $k$-th last wave is also
tight, and satisfies a power law upper bound. We are going to need 
the following lemma. 

\begin{lem}
\label{lem:circuit}
There exist constants $C$ such that the following holds. 
Let $1 \le r < r' < L$, and let $K \subset V(L) \cup \{ s \}$ 
be a connected set of edges that contains a path connecting 
$B(r)$ to $s$. We have
\eqnst
{ \mu_{L,o} ( \frT_{L,o} \not\subset B(r') \,|\, K \subset \frT_{L,s} ) 
  \le C \, r^{3/2} \, (r')^{-1/2}. }
\end{lem}

\begin{proof}
Let us use Wilson's algorithm in the contracted graph
$G_L / K$, that is, the edges in $K$ are already 
present at the start of the algorithm. We let walks start 
at $\{ x_1, \dots, x_M \} = \partial B(r)$. 
If $\frT_{L,o} \not\subset B(r')$, then at least one of these 
walks has to reach $\partial B(r')$ before hitting $K$.
Beurling's estimate (Lemma \ref{lem:Beurling}) implies that for each $x_j$, this 
has probability at most $C \, (r'/r)^{-1/2}$. Since
$M = O(r)$, the statement follows.
\end{proof}

\begin{prop}
\label{prop:out-rad-tight}
There exist constants $\alpha_1 > \alpha_2 > \dots > 0$ and
$C_1, C_2, \dots$ such that 
\eqn{e:power-wave-k2}
{ \limsup_{L \to \infty} \nu_L \big( R^L_{\rout, k} > r \big)
  \le C_k r^{-\alpha_k}, \quad \forall r \ge 1,\, \forall k \ge 1. }
In particular, for all $1 \le k < \infty$ the sequence
$\{ R^L_{\rout,k} \}_{L \ge 1}$ is tight.
\end{prop}

\begin{proof}
Fix $1 \le k < \infty$, and $1 \le r_0 < \infty$. From Proposition \ref{prop:in-rad-tight} we have that there exists $L_0 = L_0(r_0) < \infty$ such that for all $L \ge L_0$ we have
\eqn{e:Rin-bnd}
{ \nu_L ( R^L_{\rin,k} > r_0 )
  \le 2 \, C_k \, r_0^{-\alpha'_k} }
Assume the event $\{ R^L_{\rin,k} \le r_0 \}$,
which implies that $\frT_{L,s,k} \cap \partial B(r_0) \not= \varnothing$. 
We bound the probability that $R^L_{\rout,k} > r_0^{1 + \beta}$, where
the parameter $\beta > 0$ will be chosen at the end.

Similarly to \eqref{e:Pk-bound}, we have:
\eqnspl{e:Rk-bound}
{ \nu_{L} \left( R^L_{\rin,k} \le r_0,\, R^L_{\rout,k} > r_0^{1+\beta} \right)   
  \le C \, (\log L) \, \sum_{x \in \partial B(r_0)} 
      \mu_{L,o} \left( x \in \frT_{L,s},\, 
      \frT_{L,o} \not\subset B ( r_0^{1+\beta} ) \right). }
Let $\cK$ be the set of edges in $\frT_{L,s}$ on the path from $x$ to $s$. Conditioning
on the value $\cK = K$, the right hand side of \eqref{e:Rk-bound} equals
\eqnspl{e:sum-x}
{  C \, (\log L) \, \sum_{x \in \partial B(r_0)} 
      \mu_{L,o} \left( x \in \frT_{L,s} \right) \, 
      \sum_{K} \mu_{L,o} \left( \cK = K \,|\, x \in \frT_{L,s} \right) \,
      \mu_{L,o} \Big( T_o \not\subset B ( r_0^{1+\beta} ) \,\Big|\, \cK = K \Big). }
We have $\mu_{L,o} ( x \in \frT_{L,s} ) \le C (\log r_0) / (\log L)$.
Applying Lemma \ref{lem:circuit} to the conditional probability in 
\eqref{e:sum-x} gives that the expression in \eqref{e:sum-x} is at most
\eqnspl{e:r0-power}
{ C \, (\log r_0) \, r_0^{3/2} \, r_0^{-(1+\beta)/2} \, \sum_{x \in \partial B(r_0)}  
      \sum_{K} \mu_{L,o} \left( \cK = K \,|\, x \in \frT_{L,s} \right) 
  = C \, (\log r_0) \, r_0^{2 - \beta/2}. }
Choose $\beta$ so that $2 - \beta/2 < -\alpha'_k$, so that  
\eqref{e:r0-power} together with \eqref{e:Rin-bnd} gives
\eqnst
{ \nu_L ( R^L_{\rout,k} > r_0^{1+\beta} )
  \le C_k r_0^{-\alpha'_k}. }
This implies the statement of the proposition with $\alpha_k = \alpha'_k/(1+\beta)$.
\end{proof}

The next proposition shows that 
$\mu_{L,o} ( A \subset \frT_{L,s},\, B \subset \frT_{L,o} \}$ 
for fixed finite fixed sets of vertices $A$ and $B$, satisfies a certain
normalization as $L \to \infty$.
Tightness of the in-radius established in Proposition \ref{prop:in-rad-tight} 
will allow us to apply this proposition, and subsequently prove 
Theorem \ref{thm:last-k-waves}. We introduce the notation 
$q_L := \mu_{L,o} ( e \not\in \frT_{L,o} )$, where $e \sim o$. Due to symmetry, 
$q_L$ does not depend on $e$. In fact, since $q_L$ is the escape
probability of random walk from $o$, we have $q_L = G_L(o,o)^{-1}$.
We remark that for the square grid, the quantity $G_L(o,o)$ has an explicit 
formula:
\eqn{e:GLoo}
{ G_L(o,o)
  = \sum_{\ell=0}^L
    \frac{T'_{L+1}(2 - \cos \theta_\ell)}{T_{L+1}(2 - \cos \theta_\ell)}, }  
where $\theta_\ell = 2 \pi (2 \ell + 1) / (4 L + 4)$, and 
$T_{L+1}$ is the degree $L+1$ Tchebyshev polynomial. The formula \eqref{e:GLoo}
can be derived via contour integration. However, we will not need it, and we 
omit the proof.

\begin{prop}
\label{prop:cond-lim-exists}
Assume $d = 2$. Let $A, B \subset \Z^2$ be disjoint, non-empty finite sets, 
with $o \in B$. Then the limit $p_{A,B} := \lim_{L \to \infty} q_L^{-1} 
\mu_{L,o} (A \subset \frT_{L,s},\, B \subset \frT_{L,o})$ exists.
\end{prop}

We first need the following lemma. In its statement, $a(x)$ is the 
potential kernel for simple random walk on $\Z^2$; see \cite{Lawlim}.

\begin{lem}
\label{lem:path-converge}
Fix $x \in A$.\\
(i) We have
$\lim_{L \to \infty} q_L^{-1} \mu_{L,o} ( x \not\in \frT_{L,o} ) = \frac{a(x)}{a(e)}$,
where $e \sim o$.\\
(ii) Conditional on $x \not\in \frT_{L,o}$, the law of the  
path from $x$ to $s$ has a weak limit, as $L \to \infty$. 
\end{lem}

\begin{proof}
(i) We use Wilson's algorithm in the graph $G_{L,o}$ with a walk starting at $x$.
Considering the limit of the bounded martingale $\{ a(S_x(t \wedge \sigma_{V(L)} \wedge \xi_{\{ o \}})) \}_{t \ge 0}$, and using Lemma \ref{lem:rwprob} we have
\eqnsplst
{ q_L^{-1} \, \mu_{L,o} ( x \not\in \frT_{L,o} )
  &= \frac{\prob \big( o \not\in S_x[0,\sigma_{V(L)}] 
     \big)}{\prob \big( o \not\in S_e[0,\sigma_{V(L)}] \big)} 
  = \frac{a(x) (\log L)^{-1} + o \big( (\log L)^{-1} 
     \big)}{a(e) (\log L)^{-1} + o \big( (\log L)^{-1} \big)}
  = \frac{a(x)}{a(e)} + o(1). }

(ii) Let $S^{h,L}_x$ denote a random walk conditioned on the event
$\{ \sigma_L < \xi_{o} \}$, started from $x$. The path in $\frT_{L,s}$ 
from $x$ to $s$, conditional on $x \not\in \frT_{L,o}$ has the law of 
$\looper S^{h,L}_x[0,\sigma_L]$. As $L \to \infty$, $S^{h,L}_x$ converges
weakly to a transient process $S^h_x$ (the $h$-transform of random walk
by the potential kernel $a(\cdot)$). This implies that 
$\looper S^{h,L}_x[0,\sigma_L]$ converges weakly to 
$\looper S^h_x[0,\infty)$; see \cite[Exercise 11.2]{Lawlim}.
\end{proof}

\begin{proof}[Proof of Proposition \ref{prop:cond-lim-exists}.]
Let $A = \{ x_1, \dots, x_p \}$, ($p \ge 1$) and 
$B = \{ o, w_1, \dots, w_q \}$, ($q \ge 0$). 
We use Wilson's algorithm in the graph $G_{L,o}$. 
We start with the vertex $x_1$, followed by the vertices 
$x_2, \dots, x_p$, followed by the vertices in $B$. For the rest of the 
vertices we use an ordering such that their Euclidean norms form a non-decreasing 
sequence. Due to Lemma \ref{lem:path-converge}(i), the probability that 
the first walk hits $s$ before $o$ is asymptotic to $q_L a(x_1)/a(e)$, 
as $L \to \infty$. Assuming this event happens, let $\pi^{L}_{x_1}$ denote 
the loop-erasure of the walk starting at $x_1$, 
and write $\pi_{x_1}$ for its weak limit under the conditioning,
whose existence is guaranteed by Lemma \ref{lem:path-converge}(ii).
The probability that the walks starting in $(A \setminus \{ x_1 \}) \cup B$
terminate before exiting a ball $B(r)$ of large radius $r$ goes to $1$ as
$r \to \infty$, and $L > r$, uniformly in the path $\pi^{L}_{x_1}$. 
Since these walks determine the event 
$\{ A \subset \frT_{L,s},\, B \subset \frT_{L,o} \}$, 
statement (i) follows.
\end{proof}

In the proof of Theorem \ref{thm:last-k-waves} we are going to need
the following quantitative bound from \cite{GJ14} on the rate of 
convergence of $\nu_{L}$ to $\nu$.

\begin{thm}[Theorem 4.1 of \cite{GJ14}]
\label{thm:quantitative}
Let $d = 2$. There exist constants $0 < \alpha < \beta$ and $C$ such that 
if $E$ is any cylinder event depending only on the configuration
in $B(\ell)$, then 
\eqnst
{ \left| \nu_{L}(E) - \nu(E) \right|
  \le C \ell^\beta L^{-\alpha}. }
\end{thm}

\begin{proof}[Proof of Theorem \ref{thm:last-k-waves}.]
(i)--(ii) We showed in Proposition \ref{prop:out-rad-tight} that 
for any fixed $1 \le k < \infty$, the sequence $R^{(k)}_\rout = R^{(k)}_{L,\rout}$, 
$L \ge 1$, is tight. Therefore, we have
\begin{equation}
\label{e:waves-tight} 
  \lim_{\ell \to \infty} \limsup_{L \to \infty} 
     \nu_{L} ( N \ge k,\, \text{$\cW_{N-i+1} \not\subset B(\ell)$ 
       for some $1 \le i \le k$} ) 
   = 0. 
\end{equation}

We establish weak convergence of $\xi^L_k$. Fix $\eps > 0$, and let
$\ell$ and $L_0$ be such that for all $L \ge L_0$ the probability appearing in 
\eqref{e:waves-tight} is $\le \eps$. Let $\zeta \in \cR'_{B(\ell)}$ be a 
configuration with the following properties:\\
\ \\
(a) $(a'_{o,{B(\ell)}})^{j}(\zeta) \in \cR'_{B(\ell)} \setminus \cR_{B(\ell)}$ for
$j = 1, \dots, k-1$ and $(a'_{o,{B(\ell)}})^k(\zeta) \in \cR_{B(\ell)}$. \\
(b) In the stabilization $(a'_{o,{B(\ell)}})^k(\zeta)$ none of the 
boundary vertices of $B(\ell)$ topple; \\
\ \\
In other words, $\zeta$ is an intermediate configuration in $B(\ell)$,
corresponding to a $k$-th last wave such that all of the last $k$ waves 
stay inside $B(\ell)$. We first show that 
\[ \lim_{L \to \infty} \nu_{L} \big( (\xi^L_k)_{B(\ell)} = \zeta \,\big|\, N \ge k \big) \]
exists for any $\zeta$ satisfying (a)--(c).

First observe that the properties of $\zeta$ imply that for any
$\eta_* \in \cR'_{L} \setminus \cR_L$, if $(\eta_*)_{B(\ell)} = \zeta$ then 
$\eta_*$ is an intermediate configuration, in the graph $G_L$, corresponding to a 
$k$-th last wave, and the last $k$ waves all stay in $B(\ell)$. 
In particular, using that the $k$-th last wave stays inside
$B(\ell)$, property \eqref{e:path-property} of the bijection 
implies that $\zeta$ determines a unique pair $(A_0,B_0)$, 
with $A_0 \cup B_0 = B(\ell)$, such that whenever 
$\eta_* \in \cR'_{L} \setminus \cR_L$ and $\eta_*\vert_{B(\ell)} = \zeta$,
then we have $V(T_o(\eta_*)) = B_0$. 
Therefore, using Lemma \ref{lem:wave-to-recur} we can write
\eqnspl{e:dist-zeta}
{ \nu_{L} \left( N \ge k,\, \xi^L_k\vert_{B(\ell)} = \zeta \right) 
  &= \frac{ \big| \{ \eta_* \in \cR'_{L} \setminus \cR_L : 
     \eta_*\vert_{B(\ell)} = \zeta \} \big| }{ \big| \cR_{L} \big| } \\
  &= g_L(o,o) \, \nu_{L,o} \left( \eta_* : 
     \eta_*\vert_{B(\ell)} = \zeta \right). }
For any $\eta_*$ appearing in the right hand side of \eqref{e:dist-zeta}, let
$\eta = \eta_*\vert_{V_L \setminus B_0}$.
Due to the burning process, the conditional distribution of $\eta$, given the event 
$\{ \eta_*\vert_{B_0} = \zeta\vert_{B_0},\, V(T_o(\eta_*)) = B_0 \}$, 
equals that of a recurrent sandpile in the subgraph $G^{B_0}_L$ of $G_L$ induced by 
the set of vertices $V(L) \cup \{ s \} \setminus B_0$ (i.e. with closed boundary condition 
at $B_0$). Hence the last expression in \eqref{e:dist-zeta} equals
\eqnspl{e:zeta-cond}
{ &q_L^{-1} \, \nu_{L,o} \big( \eta_* : \eta_*\vert_{B_0} = \zeta_{B_0},\, 
    V(T_o(\eta_*)) = B_0 \big)
    \, \nu_{G^{B_0}_L} \big( \eta : \eta\vert_{A_0} = \zeta\vert_{A_0} \big). }
Since the wired spanning forest in the subgraph of $\Z^2$ induced by 
$\Z^2 \setminus B_0$ is one-ended, we can apply \cite[Theorem 3]{JW14}  
to deduce that the last factor in \eqref{e:zeta-cond} has a limit as 
$L \to \infty$. The first factor equals
\eqn{e:num-trees}
{ \frac{1}{| \cT_{B_0} |} \, q_L^{-1} \, \mu_{L,o} \left( V(\frT_{L,o}) = B_0 \right), }
where $| \cT_{B_0} |$ is the number of spanning trees in the graph induced by $B_0$.
Due to Proposition \ref{prop:cond-lim-exists}, the product of the second and third factors in \eqref{e:num-trees} approaches $p_{A_0,B_0}$, as $L \to \infty$. 
This implies the existence of the limit 
\eqnst
{ \lim_{L \to \infty} \nu_L \left( N \ge k,\, (\xi^L_k)\vert_{B(\ell)} = \zeta \right)
  =: c_k(\zeta). }
Summing over all $\zeta$ satisfying (a)--(c), and using the choice of $\ell$ made after \eqref{e:waves-tight}, we get
\eqnst
{ \left| \limsup_{L \to \infty} \nu_L ( N \ge k ) - 
    \liminf_{L \to \infty} \nu_L ( N \ge k ) \right|
  \le \eps. }
But since the left hand side does not depend on $\ell$, we have that 
the limit $\lim_{L \to \infty} \nu_L (N \ge k ) =: b_k$ exists, proving
statement (i). It follows that 
\eqnst
{ \lim_{L \to \infty} \nu_L \big( (\xi^L_k)_{B(\ell)} = \zeta \,\big| N \ge k \big)
  = \frac{c_k(\zeta)}{b_k}
  =: \rho_k \left( \xi_k : (\xi_k)\vert_{B(\ell)} = \zeta \right). }
Statement (ii) follows immediately from this.

(iii) Observe that the proof of parts (i)--(ii) shows that 
up to a set of measure $0$, the support of $\rho_k$ can be partitioned into
a countable disjoint union of cylinder sets, such that on each element of
the partition, the stabilization $(a'_{o,{\mathbb{Z}^2}})^k (\xi_k)$ 
takes place within a finite set $B(\ell)$.

(iv) The countable partition into cylinder sets has the further property that 
the map $\xi_k \mapsto a'_{o,{\mathbb{Z}^2}}(\xi_k)$ is measure
preserving on each cylinder set of the partition. Hence the claim
follows.

(v) Let $\eps > 0$ be fixed. Let $N_{B(\ell)}$ denote the number of times 
$o$ topples if all topplings in $B(\ell)$ are carried out, but no 
site in $B(\ell)^c$ is allowed to topple. Due to the $\nu$-a.s.\ convergence 
$N_{B(\ell)} \uparrow N$, there exists $1 \le \ell < \infty$
such that $\nu ( \{ N = k \} \Delta \{ N_{B(\ell)} = k \} ) < \eps$,
where $\Delta$ denotes symmetric difference. Let 
\eqnst
{ F_{\ell,k} 
  = \{ \text{$N_{B(\ell)} = k$ and some boundary vertex of $B(\ell)$ topples} \}. }
Since $\{ N = k \} \cap F_{\ell,k}^c$ is a cylinder event, we have
\eqnsplst
{ \nu ( N = k )
  &\ge \nu ( N = k,\, F_{\ell,k}^c )  
  = \lim_{L \to \infty} \nu_L ( N = k,\, F_{\ell,k}^c ) 
  \ge \lim_{L \to \infty} \nu_L ( N = k ) - \eps(\ell,k), }
where $\eps(\ell,k) \to 0$ as $\ell \to \infty$, due to \eqref{e:waves-tight}.
It follows that $b_k - b_{k+1} = \lim_{L \to \infty} \nu_{L} ( N = k ) \le \nu ( N = k )$. 

For an inequality in the other direction, we write:
\eqnsplst
{ \nu ( N = k )
  &\le \nu ( N = k,\, F_{\ell,k}^c ) + \nu ( F_{\ell,k} ) 
  = \nu ( N_{B(\ell)} = k,\, F_{\ell,k}^c ) + \nu ( F_{\ell,k} ) \\
  &= \lim_{L \to \infty} \nu_L ( N_{B(\ell)} = k,\, F_{\ell,k}^c )
    + \lim_{L \to \infty} \nu_L ( F_{\ell,k} ) \\
  &\le \lim_{L \to \infty} \nu_L ( N = k )
    + \limsup_{L \to \infty} \nu_L ( N > k,\, F_{\ell,k} )
    + \lim_{L \to \infty} \nu_L ( N = k,\, F_{\ell,k} ). }
Due to \eqref{e:waves-tight}, the third term on the right hand side is at most $\eps(\ell,k) \to 0$
as $\ell \to \infty$. Therefore, it is enough to show that 
\eqn{e:on-F-ell-k}
{ \lim_{\ell \to \infty} \limsup_{L \to \infty} 
    \nu_{L} \left( N > k,\, F_{\ell,k} \right)
  = 0. }

We fix $0 < \delta < \alpha / \beta$ (where $\alpha, \beta$ are the 
constants from Theorem \ref{thm:quantitative}). 
Let $r(i) = L^{\delta/\rho^i}$, $i = 0, \dots, k$, where 
the constant $1 < \rho < \infty$ will be chosen later.
Recall we denote by $\eta_1, \dots, \eta_k \in \cR'_L \setminus \cR_L$ the 
first $k$ waves corresponding to $\eta \in \cR_L$. We define the events
\eqnst
{ H(i)
  = \left\{ \parbox{8.5cm}{the $i$-th wave $\cW(\eta_i)$ does not topple any vertices in
    $B(r(i))$ after it has reached $\partial B(r(i-1))$} \right\}, \quad
    i = 1, \dots, k. }
Recalling property \eqref{e:path-property} of the bijection,
an argument similar to the one made in Proposition \ref{prop:in-rad-tight}
yields
\eqn{e:H(i)^c-bound}
{ \nu_{L} (H(i)^c)
  \le C \, (\log L) \, r(i-1)^{-1/4} r(i)^{9/4}
  \le C \, (\log L) \, L^{2 \delta \rho^{-i}} \, L^{-\delta \rho^{-i} (\rho - 1) / 4}. }
We choose $\rho > 9$, so that 
the right hand side of \eqref{e:H(i)^c-bound} 
goes to $0$ as $L \to \infty$.
Therefore, in order to prove \eqref{e:on-F-ell-k}, it is enough to bound 
\eqn{e:H's}
{ \nu_{L} \left( N > k,\, F_{\ell,k},\, H(1) \cap \dots \cap H(k) \right)
  \le \nu_{L} \left( N > k,\, N_{B(\ell)} = k,\, H(1) \cap \dots \cap H(k) \right). }

Suppose now that we are given a configuration $\eta \in \cR_{L}$. Let us 
carry out the first wave up to $\partial B(r(0))$, that is, stop the first
wave when a vertex of $B(r(0))^c$ would need to be toppled, if any.
Then carry out the second wave up to $\partial B(r(1))$, the third
wave up to $\partial B(r(2))$, etc. Let $F'$ denote the event
that during the $k$-th ``partial wave'' defined this way, all neighbours of the
origin topple. The event in the right hand side of \eqref{e:H's}
implies the event $F'$. This is because the event $\{ N > k \}$, in the
presence of $H(1) \cap \dots \cap H(k)$, implies that the origin can be toppled
a $k+1$-st time after the first $k$ partial waves. 

Observe that $F'$ is measurable with respect to the pile inside $B(r(0))$,
and $r(0) = L^\delta$. Hence, using Theorem \ref{thm:quantitative}, and
$\nu ( \{ N = k \} \Delta \{ N_{B(\ell)} = k \} ) \le \eps$, 
the right hand side of \eqref{e:H's} is at most
\eqnsplst
{ \nu_{L} \left( F',\, N_{B(\ell)} = k \right)
  &\le \nu \left( F',\, N_{B(\ell)} = k \right) 
      + C L^{\delta \beta} L^{-\alpha}
  \le \nu ( F',\, N = k ) + \eps + C L^{-\alpha + \beta \delta} \\
  &= \eps + C L^{-\alpha + \beta \delta}, }
where the last equality follows from the fact that $F' \subset \{ N > k \}$.
Due to the choice of $\delta$, the second term goes to $0$, as $L \to \infty$.
Since $\eps$ is arbitrary, we obtain statement (v) of the Theorem. 
\end{proof}

We can now complete the proof of Theorem \ref{thm:last-k-limit}.

\begin{proof}[Proof of Theorem \ref{thm:last-k-limit}]
(i) This follows immediately from Theorem \ref{thm:last-k-waves}(ii)--(iii).

(ii) This follows from Theorem \ref{thm:last-k-waves}(i),(v).

(iii) This follows from Theorem \ref{thm:last-k-waves}(ii),(v), since on the
event $\{ N = k \}$ we can approximate by cylinder events on which 
no vertex topples outside a fixed ball.
\end{proof}

The following lemma completes the proof of the $d = 2$ case of Proposition \ref{prop:lim-toppling}.

\begin{lem}
\label{lem:approx}
We have
\[ \nu ( z \in \av )
   = \lim_{L \to \infty} \nu_L ( z \in \av ). \]
\end{lem}

\begin{proof}
For a sufficiently large number $k = k(z)$, we have the deterministic 
implication 
\eqnst
{ N > k \quad \Rightarrow \quad z \in \av, }
both in $\Z^2$, and in $V_L$ for $L$ sufficiently large. 
With such $k$, we have
\eqnsplst
{ \nu ( z \in \av )
  &= \nu ( N \ge k+1 ) + \sum_{\ell = 1}^k \nu ( N = \ell,\, z \in \av ) \\
  &= \lim_{L \to \infty} \nu_L ( N \ge k+1 ) + 
     \sum_{\ell = 1}^k \lim_{L \to \infty} \nu_L ( N = \ell,\, z \in \av ) 
  = \lim_{L \to \infty} \nu_L ( z \in \av ). }
In the second equality, we applied Theorem \ref{thm:last-k-waves}(v) to the
first term. In the second term, a.s.~finiteness of the last $\ell$ waves 
allows us to approximate $\{ N = \ell,\, z \in \av \}$ by a cylinder event,
and the equality follows.
\end{proof}

\begin{proof}[Proof of Theorem \ref{thm:2d-upper}] 
(i) The bound follows immediately from the Proposition \ref{prop:out-rad-tight}.

(ii) Since on the event $\{ N \le k \}$ we have $R \le \max \{ R_1, \dots, R_k \}$,
this also follows Proposition \ref{prop:out-rad-tight}.
\end{proof}

We now prove Theorem \ref{thm:exp-wave-infinite}. The idea of the argument is that, if $f(x) := \E_\nu n(o,x)$ were finite, then by invariance of $\nu$ under $a_o$, $f$ would have to be a bounded harmonic function, hence constant. This is in contradiction with the structure of the avalanche. We first give two short lemmas on which this argument will be based.

  \begin{lem}
    Assuming $\nu(N < \infty) = 1$, we have $\nu(R < \infty) = 1$.
    \label{lem:finiteav}
  \end{lem}
  \begin{proof}
    This follows easily from Theorem \ref{thm:2d-upper}.
  \end{proof}

Let $a_o$ denote the operation on stable sandpiles on $\Z^2$ which maps $\eta$ to $(\eta + 1_o)^\circ$ if a finite stabilization is possible (i.e., if $ S  < \infty$). Then the preceding lemma implies if $\nu(N < \infty) = 1$, there exists a set $\Omega$ with $\nu(\Omega) = 1$ such that $a_o$ is defined on $\Omega$. Given such an $\Omega$, the next lemma shows that, similarly to $\nu_L$,  the infinite-volume measure $\nu$ is invariant under $a_o$.

\begin{lem}
 Assuming $\nu(N < \infty) = 1$,  $\nu$ is invariant under $a_o$. That is, for any $\nu$-integrable function $f$,
\[\int f(a_o \eta)\, \nu(\mathrm{d} \eta) = \int f(\eta)\,\nu(\mathrm{d} \eta)\ . \]
\end{lem}
\begin{proof}
The argument of \cite[Prop. 3.14]{JR08} carries over exactly. There a similar statement is proved in the case $d \geq 3$, but the argument requires only almost sure finiteness of avalanches.
\end{proof}

\begin{proof}[Proof of Theorem \ref{thm:exp-wave-infinite}]
  If $\nu(N = \infty) > 0$, there is nothing to prove. We thus 
assume that $\nu(N < \infty) = 1$. We first note that, under this assumption, the infinite-volume addition operators are well-defined, since $n(o,x) \le n(o,o)$.

The invariance statement above makes possible a version of the argument underlying Dhar's formula (Lemma \ref{lem:Dhar}). Assume for the sake of contradiction that $\E_\nu N < \infty$. In particular, $0 \leq \mathbb{E}n(o,x)  \leq \E n(o,o) < \infty$ for all $x \in \Z^2$.  Since $N < \infty$, Lemma \ref{lem:finiteav} above shows we can (almost surely) write
\[(\eta + \mathbf{1}_o)^\circ = \eta + \mathbf{1}_o - \sum_{z \in \Z^2} n(o,z) \Delta(z, \cdot)\ , \]
where the sum above has finitely many nonzero terms and $\Delta$ is the graph Laplacian on $\Z^2$. In particular, taking expectations and using the invariance above (which implies $\mathbb{E}_\nu \eta(x) = \mathbb{E}_\nu (\eta + \mathbf{1}_o)^\circ(x)$ for all $x$) gives
\begin{equation}
\label{eq:harmonic}
  \sum_{z \in \Z^2} \E_\nu n(o,z) \Delta(z, x) 
  = \mathbf{1}_o(x)\quad \text{for all } x \in \Z^2\ . 
\end{equation}
In other words, \eqref{eq:harmonic} says that $f(x) = \E_\nu n(o,x)$ is harmonic away from $o$ and has Laplacian $1$ at $o$. Since $f$ is bounded, recurrence of random walk
implies that $f$ is constant. Since $n(o,x) \le n(o,o)$ for all $x$, we have 
$\nu ( \text{$n(o,x) = n(o,o)$ for all $x \in \Z^2$} ) = 1$. However, if all vertices topple, the avalanche is infinite, a contradiction.
\end{proof}

%
%

\section{High-dimensional radius bounds}
\label{sec:highd_rad}
In this section we prove the bounds for the radius of the avalanche for $d\geq 5.$ 
We prove Theorem \ref{thm:rad-gen} in Section \ref{ssec:rad-past-lb}.
We use the results of 
Section \ref{ssec:rad-past-lb} in Sections \ref{sec:lowerbd} and \ref{sec:upperbd} to prove the lower and upper bounds on the radius.

\subsection{Radius bounds on transitive unimodular graphs}
\label{ssec:rad-past-lb}

%
%
%



See \cite[Chapter 8]{LPbook} for background on unimodularity and mass transport.

\begin{proof}[Proof of Theorem \ref{thm:rad-gen}]
(i) We denote $A_x(a,b) = D_x(b) \setminus D_x(a)$.
Consider the following mass transport. When $\diam(\past_x) > r$, 
let $x$ send unit mass distributed equally among all vertices 
$y \in \past_x \cap A_x(r,2r)$. 
Let us write $\sentm(x)$ and $\getm(x)$, respectively, for the amount
sent and received by $x$, respectively. Then 
$\E [ \sentm(o) ] = \mu \big( \diam(\past_o) > r \big)$.
On the other hand, using Jensen's inequality, we have
\eqnsplst
{ \E [ \getm(o) ]
  &= \sum_{x \in A_o(r,2r)} \mu \big( o \in \past_x \big)
    \E \bigg[ \frac{1}{| \past_x \cap A_x(r,2r) |} \,\bigg|\,
       o \in \past_x \bigg] \\
  &\ge \sum_{x \in A_o(r,2r)} \frac{\mu \big( o \in \past_x \big)^2}
    {\E \bigg[ \big| \past_x \cap A_x(r,2r) \big| \, 
    \mathbf{1}_{o \in \past_x} \bigg]}. }
Since $\past_x \cap A_x(r,2r) \subset \frT_o \cap D_o(4r)$, the 
statement follows.

(ii) When $\diam (\frC_x) > 4r$, let $x$ send unit mass distributed
equally among vertices $y \in \frC_x \cap A_x(r,4r)$.
Then $\E [ \sentm(o) ] = \mu \big( \diam(\frC_o) > 4r \big)$.
On the other hand,
\eqnsplst
{ \E [ \getm(o) ]
  &= \sum_{x \in A_o(r,4r)} \mu \big( o \in \frC_x \big)
    \E \bigg[ \frac{\mathbf{1}_{\diam(\frC_x;x) > 4r}}{| \frC_x \cap A_x(r,4r) |} 
    \,\bigg|\, o \in \frC_x \bigg]. }
and the statement follows.


(iii) Although $\wsfo$ is not automorphism invariant, essentially the same 
mass transport can be used as in (ii), due to the fact that the rooted
random network $(\frT_o,o)$ is unimodular, in the sense of \cite{BS01,AL07}.
Let $T \subset V$ be a finite tree, and $x, y \in T$. Define the function
\eqnst
{ f(T,x,y)
  = \begin{cases}
    \mathbf{1}_{y \in T} \, | T \cap A_x(r,4r) |^{-1} & 
      \text{when $\diam(T;x) > 4r$;} \\
    0  & \text{otherwise.}
    \end{cases} }
Let $x$ send the following mass to $y$:
\eqnst
{ F(x,y) 
  = \sum_{T \ni x, y} \wsf_x ( \frT_x = T ) \, f(T,x,y). }
Let $\gamma \in \Gamma$. It is clear that $f(\gamma T, \gamma x, \gamma y) = f(T,x,y)$,
and shifting Wilson's algorithm by $\gamma$ shows that 
$\wsf_x ( \frT_x = T ) = \wsf_{\gamma x} ( \frT_{\gamma x} = \gamma T)$. 
Therefore, $F$ is invariant under the diagonal action of $\Gamma$, and hence
\eqnspl{e:wsfo-m-tr}
{ &\wsfo ( \diam(\frT_o) > 4r )
  = \sum_{x \in V} F(o,x)
  = \sum_{x \in V} F(x,o) \\
  &\qquad = \sum_{x \in A_o(r,4r)} \, 
    \sum_{\substack{T \ni x, o \\ \diam(T;x) > 4r}}
    \frac{\wsf_x ( \frT_x = T )}{| T \cap A_x(r,4r) |}. }
This yields the statement.
\end{proof}

\begin{rem}
Part (iii) of Theorem \ref{thm:rad-gen} can also be deduced from part (ii)
via the following limiting procedure, the details of which we omit.
(This construction was our initial approach to the radius upper bound.) 
Let $\omega \subset V$ be independent site percolation on $V$ with density 
$0 < p \ll 1$. Conditional on $\omega$, let $\wsf_\omega$ denote the measure
on spanning forests of $(V,E)$ where each vertex in $\omega$ is ``wired to 
infinity", in analogy with $\wsfo$. The measure obtained by averaging 
$\wsf_\omega$ over $\omega$ is automorphism invariant. When $x \in \omega$,
let $x$ send or not send mass according to the same rule as in part (ii), 
otherwise let $x$ send no mass. Now let $p \downarrow 0$. Conditional on 
$o \in \omega$ we have $\wsf_\omega \Rightarrow \wsfo$, and the mass 
transport in part (iii) can be recovered in this limit.
\end{rem}

\subsection{Radius lower bound when $d\geq 5$}
\label{sec:lowerbd}

We prove the lower bound using the result of the previous section.

\begin{proof}[Proof of Theorem \ref{thm:radius}(iv), lower bound]
We begin with some terminology. As in \cite{LMS}, let the `past of a vertex $x$' in a spanning tree $T$ of $G_L$, be the union of the connected components 
of $T \setminus \{x\}$, which do not contain $s$. 
We denote this object by $\past_x(T)$. 
Using the characterization of last waves from Section \ref{ssec:interm} and the bijection for intermediate configurations from Lemma \ref{lem:varphi'} we have the following observation:
\begin{align}
\label{eq:1}
 \big| \big\{ \eta \in \R_L : R(\eta) > r \big\} \big| 
 \geq \big| \big\{ (T_o,T_s) \in \ST_{L,o} : R(T_o) > r,\, 
    \text{$e \notin T_o$ for some $e \sim o$} \big\} \big|.
\end{align}
Now, consider the map $\Psi : \ST_L \to \ST_{L,o}$, 
which modifies a spanning tree $T$ by deleting the unique 
edge $\{ o, e \}$ such that $e$ is in the future of $o$ 
and $\{ o, e \}$ belongs to $T$. 
The following properties of the map $\Psi$ are immediate.

\begin{lem}
Let $U \subset \ST_L$. 
Then $|\Psi(U)| \leq |U| \leq 2d \, |\Psi(U)|.$
\end{lem}


We use the lemma with $U = \{ T \in \ST_L : R(\past_o(T)) > r \}$, for which we have 
\begin{equation*}
 \Psi(U) \subset \big\{ (T_o,T_s) \in \ST_{L,o} : R(T_o) > r,\, 
    \text{$e \notin T_o$ for some $e \sim o$} \big\}.
\end{equation*}
Therefore equation \eqref{eq:1} gives that 
\begin{equation*}
 \dfrac{\big| \big\{ \eta \in \R_L : R(\eta) > r \big\} \big|}{|\R_L|}
 \geq \frac{|\Psi(U)|}{|\R_L|}
 \geq \frac{1}{2d} \, \frac{|U|}{|\R_L|}
 \geq \frac{1}{2d} \, \dfrac{\big| \big\{ T \in \ST_L : R(\past_o(T)) > r \big\} \big|}{ | \ST_L | }.
\end{equation*}
This equation holds for any fixed $r$, and all $L$. Using the observation that for fixed $r$ the events at both ends are cylinder events, and taking the limit as $L\nearrow \infty$, we have,
\begin{equation}\label{past}
  \nu \big( R > r \big)
  \geq \frac{1}{2d} \, \wsf \big( R(\past_o(\frT)) > r \big).
\end{equation}
The proof is completed by applying Theorem \ref{thm:rad-gen}(i).
Using the fact that in dimensions $d \ge 5$ there is probability at least 
$c$ that two independent simple random walks starting at $x$ do not 
intersect, we deduce that $\wsf ( x \in \past_o) \ge c |x|^{2-d}$. 
On the other hand, Wilson's algorithm gives
\eqnsplst
{ \E \big[ | \frT_o \cap B_o(4r) | \, \mathbf{1}_{o \in \past_x} \big]
  &\le \sum_{y \in B_o(4r)} \sum_{v \in \Z^d} 
    \big[ G(o,v) \, G(y,v) \, G(v,x) + G(o,x) \, G(x,v) \, G(y,v) \big] \\
  &\le C \, r^{6-d}. }
Therefore, Theorem \ref{thm:rad-gen}(i) implies $\wsf ( R(\past_o) > r ) \ge c r^{-2}$. 
\end{proof}

\subsection{Radius upper bound when $d\geq 5$}
\label{sec:upperbd}

In this section we prove the upper bound in Theorem \ref{thm:radius}(iv).
Taking $d$ to be $\ell_\infty$ distance, Theorem \ref{thm:rad-gen}(iii) yields
\eqnspl{e:limQ_o}
{ &\wsf_o( \diam (\frT_o) > 4r) \\
  &\qquad = \sum_{x \in \Z^d : r < \| x \|_\infty \le 4r}
    \wsf_x ( o \in \frT_x ) \, \E_{\wsf_x} \bigg[ \frac{\mathbf{1}_{\diam(\frT_x;x) > 4r}}{| \frT_x \cap V_x(4r) 
    \setminus V_x(r) |} \,\bigg|\, x \in \frT_o \bigg] \\
  &\qquad \le C \, r^{2-d} \, \sum_{x \in \Z^d : r < \| x \|_\infty \le 4r}
    \E_{\wsf_x} \bigg[ \frac{\mathbf{1}_{\diam(\frT_x;x) > 4r}}{| \locball_x |} \,\bigg|\, x \in \frT_o \bigg], }
where we wrote $\locball_x = \tree_x \cap  V_x(4r) \setminus V_x(r)$, and 
used $\wsf_x ( o \in \frT_x ) \le G(o,x) \le C \, r^{2-d}$.
We show that for $\delta > 0$ there exists $C_1 = C_1(\delta)$ such that 
the expectation in the right hand side 
is bounded above by $C_1 \, (\log r)^{3 + \delta} \, r^{-4}$. 
This implies the required upper bound on the tail of the diameter.
We are going to use the following theorem of \cite{BJ15} on the lower tail of the volume of $\wsf$ components. Given $D \subset \Z^d$, write $\wsf_{D^c}$ for the wired spanning 
forest measure on the contracted graph $\Z^d / D^c$.

\begin{thm}\label{thm:BJ12}\cite{BJ15}
Let $x, y \in \Z^d$ be such that $\| y - x \|_\infty = 4r$. Let $D \subset \Z^d$ be such $y \in \partial D$, and $V_x(4r) \subset D$. Let $x \leftrightarrow y$ denote the event that in $\mathbf{WSF}_{D^c}$ the path from $x$ to $D^c$ reaches $D^c$ via an edge incident to $y$. There exist constants $C, c > 0$ independent of $D$ and $r$, such that for all $\lambda > 0$ we have
\begin{equation*}
  \mathbf{WSF}_{D^c} \Big( \vert \tree_x \cap V_x(2r) \setminus V_x(r) \vert 
     < \lambda r^4 \,\Big|\, x \leftrightarrow y \Big)
  \le C \exp ( - c \lambda^{-1/3} ). 
\end{equation*}
\end{thm}

We use Theorem \ref{thm:BJ12} to establish the following regularity estimate.
Fix a positive constant $\delta > 0$. Let us call $\tree_x$ \emph{bad}, if $\diam(\tree_x) > 4r$, but $|\tree_x \cap V_x(2r) \setminus V_x(r) | < (\log r)^{-3-\delta} \, r^4$. 
We show the following lemma.

\begin{lem}
\label{lem:bad}
Let $x \in V_o(4r) \setminus V_o(r)$. We have
\begin{equation*}
 \wsf_x ( \text{$\tree_x$ is bad} )
 \le C \exp ( - c (\log r)^{1 + \delta/3} ). 
\end{equation*}
\end{lem}

\begin{proof}
If $\diam(\tree_x) > 4r$, then there exists $y$ with $\| y - x \|_\infty = 4r$ such that 
$y \in \tree_x$. Therefore,
\begin{equation*}
 \wsf_x ( \text{$\tree_x$ is bad} )
 \le \sum_{y : \| y - x \|_\infty = 4r} \wsf_x ( y \in \tree_x ) \, 
     \wsf_x ( \text{$\tree_x$ is bad} \,|\, y \in \tree_x ). 
\end{equation*}
The term $\wsf_x ( y \in \tree_x ) \le C \, r^{2-d}$, and there are $O(r^{d-1})$ terms. 
On the other hand, we have $\wsf_x ( \frT_x = T \,|\, y \in \frT_x ) = 
\wsf_y ( \frT_x = T \,|\, x \in \frT_y )$. This follows from the fact that 
LERW from $y$ to $x$ has the same distribution as LERW from $x$ to $y$, and 
hence we can use these walks in the first step of Wilson's algorithm on the two
sides. See \cite[Corollary 11.2.2]{Lawlim} for "reversibility" of LERW.
Hence Theorem \ref{thm:BJ12} with $D = \Z^d \setminus \{ y \}$ implies that 
\begin{equation*}
 \wsf_x ( \text{$\tree_x$ is bad} \,|\, y \in \tree_x )
 = \wsf_y ( \text{$\tree_x$ is bad} \,|\, x \in \tree_y )
 \le C \, \exp ( - c \, (\log r)^{1 + \delta/3} ).
\end{equation*}
The statement of the lemma follows.
\end{proof}
%

\begin{lem}
\label{lem:inv-moment}
We have
\[ \mathbb{E}_{\wsf_x} \Bigg[\frac{\mathbf{1}_{\diam(\tree_x;x) > 4r}}{\vert\locball_x\vert} \,\Bigg|\, o \in \tree_x \Bigg]
   \le C \frac{(\log r)^{3 + \delta}}{r^4}. \]
\end{lem}

\begin{proof}
When $\tree_x$ is bad, we use that $\vert \locball_x \vert \geq 1$, and hence 
due to Lemma \ref{lem:bad} we have
\begin{equation*}
\begin{split}
  \mathbb{E}_{\wsf_x} \Bigg[\frac{\mathbf{1}_{\diam(\tree_x;x) > 4r}\,   
    \mathbf{1}_{\text{$\tree_x$ is bad}}}{\vert\locball_x\vert} \,\Bigg|\, 
    o \in \tree_x \Bigg]
  &\le \wsf_x ( \text{$\tree_x$ is bad},\, o \in \tree_x ) \, C \, r^{d-2} \\
  &\le \wsf_x ( \text{$\tree_x$ is bad}) \, C \, r^{d-2} \\
  &\le C \exp ( -c (\log r)^{1 + \delta/3} ) r^{d-2} 
  = o(r^{-4}).  
\end{split}
\end{equation*}
Therefore it is enough to consider the contribution when $\tree_x$ is not bad. When this
occurs, we have $\vert \locball_x \vert \ge (\log r)^{-3 - \delta} r^4$, and hence
\begin{equation*}
\begin{split}
  \mathbb{E}_{\wsf_x} \Bigg[\frac{\mathbf{1}_{\diam(\tree_x;x) > 4r}\, 
     \mathbf{1}_{\text{$\tree_x$ is not bad}}}{\vert\locball_x\vert} \,\Bigg|\, 
     o \in \tree_x \Bigg]
  \le \frac{(\log r)^{3 + \delta}}{r^4}.
\end{split}
\end{equation*}
This proves the claim.
\end{proof}

We can now complete the proof of the upper bound in Theorem \ref{thm:radius}(iv).
Lemma \ref{lem:inv-moment} implies
that the right hand side of \eqref{e:limQ_o} is at most 
\eqnst
{ C \, r^d \, r^{2-d} \, (\log r)^{3 + \delta} \, r^{-4} 
  = C \, (\log r)^{3 + \delta} \, r^{-2}. }
This completes the proof.

\section{Bounds on the size}
\label{sec:size}


We now prove Theorem \ref{thm:size}.
Sections \ref{ssec:low-d-upper}, \ref{ssec:size-lbd-low} and \ref{ssec:4D-moments} consider low dimensions,
and Sections \ref{ssec:size-gen-lb}, \ref{ssec:size-lbd-highD} and \ref{ssec:size-ubd-highD}
the case $d \ge 5$.

\subsection{Upper bounds on the size when $d = 3, 4$}
\label{ssec:low-d-upper}

\begin{proof}[Proof of Theorem \ref{thm:size}(ii)--(iii), upper bounds.]
The claimed upper bounds on $\nu(|\av| \geq t)$ follow immediately from Theorem \ref{thm:radius}, 
via the trivial estimate $\nu(|\av| \geq t) \leq \nu(R > c(d)\,t^{1/d})$.
In order to obtain an upper bound on the tail of $S$, 
fix some $k\geq 1$. Recalling that $N$ denotes the number of waves, we have
\begin{align}
  \nu( S > t )
  = \nu( S > t,\, N > k ) + \nu( S > t,\, N \leq k ).
\end{align}
Recalling that $d \ge 3$, we can upper bound the first term in the expression above by 
\eqnst
{ \nu( N > k )
  \le \E_{\nu} N \, k^{-1}
  \le C \, k^{-1}. } 
Next, writing $S^j$ for the size of the $j$-th wave, if $S > t$ and $N \le k$, then we
have $S^j > t/k$ for some $1 \le j \le N$. 
Hence, 
\begin{align}\notag \label{size-low-d-ubd}
 \nu( S > t)
 &\leq C \, k^{-1} + \nu( S^j > t/k, \text{ for some } j \leq N) \\ \notag
 &\leq C \, k^{-1} + C \, \nu ( R > c(d) (t/k)^{1/d} ). 
\end{align}
Due to Theorem \ref{thm:radius}(ii)--(iii), we get the bounds:
\eqnst
{ \nu ( S > t )
  \le \begin{cases}
      C \, k^{-1} + C \, t^{-1/18} \, k^{1/18} & \text{when $d = 3$;} \\
      C \, k^{-1} + C \, t^{-1/16} \, k^{1/16} & \text{when $d = 4$.} 
      \end{cases} }
Optimizing the choice of $k$ yields:
\eqnst
{ \nu ( S > t )
  \le \begin{cases}
      C \, t^{-1/19} & \text{when $d = 3$;} \\
      C \, t^{-1/17} & \text{when $d = 4$.} 
      \end{cases} }
\end{proof}


\subsection{Lower bounds on the size when $2 \le d \leq 4$}
\label{ssec:size-lbd-low}
 
We fix $z = r \, e_1$, where $r = t^{1/d}$, and write $u = \| z \|/10$. 
Recall that in the course of the proof of Lemma \ref{lem:reverse_walk_z} in Section \ref{sec:toppling}, we showed that for $L \ge 100 \, \| z \|$, and $e$ a neighbour of the origin, and with $\pi = \looper S_e[0,\sigma_L]$, we have
\eqn{e:FzL}
{ \prob \left( \pi \cap V_z(u) = \varnothing,\, \xi^{S_z}_o < \xi^{S^z}_\pi \right)
  \ge \begin{cases}
      c \, \prob(\Gamma_{z,L}) \, \log |z| & \text{when $d = 2$;} \\
      (2d)^{-1} \prob(\Gamma_{z,L})     & \text{when $d = 3, 4$.} 
      \end{cases} }
Let us write $F_{z,L}$ for the event in the left hand side of \eqref{e:FzL}.
Our goal will be to show that conditional on the event $F_{z,L}$ (when $z$ is in the last wave), a large number of vertices in $V_z(u)$ are also in the last wave, with probability bounded away from $0$. We will use a second moment argument to prove this, that is based on Proposition \ref{prop:moments-Y} below. 

Let $u' = (1-\eps) \, u$, where we are going to choose $0 < \eps < 1/4$ later.
Consider the following partial cycle popping in $G_L$, defined in three stages; see \cite{W96} and \cite{LPbook} for background on cycle popping. In the first stage, reveal the LERW $\pi$ started at $e$ and ending at $s$. In the second stage, reveal a LERW started at $z$, ending on hitting $\pi \cup \{ o \}$. In the third stage, pop all cycles that are entirely contained in $V_z(u)$ that can be popped. Condition on the event $F_{z,L}$, that is measurable with respect to the result of the first two stages. Let $\pi' = \looper S_z[0, \xi_o]$, and let $\pi'(u')$ be the portion of $\pi'$ from $z$ to the first exit from $V_z(u')$. Let 
\eqnsplst
{ I(x)
  = \{ \text{stage three reveals a path from $x$ 
    to $\pi' \cap V_z(u)$} \}, \quad x \in V_z(u/2). }
For each $x \in V_z(u/2)$ for which $I(x)$ occurs, let $p(x) \in \pi' \cap V_z(u)$ be the point where the revealed path first meets $\pi' \cap V_z(u)$. For technical reasons (that are only required for our argument when $d = 4$), we also define:
\eqnsplst
{ J(x)
  = \left\{ \xi^{S_{p(x)}}_{\pi'(u')} < \sigma^{S_{p(x)}}_{V_z(u)} \right\}. }
Let 
\eqnst
{ Y
  = Y_{u,\eps}
  = \sum_{x \in V_z(u/2)} \mathbf{1}_{I(x)} \, \mathbf{1}_{J(x)}. }
The following proposition states bounds on the first and second moments of $Y$.

\begin{prop}
\label{prop:moments-Y} \ \\
There exist $0 < \eps < 1/4$, $c_1 > 0$ and $C$ such that the following hold. \\
(i) When $d = 2$, we have
\eqnst
{ \E ( Y_{u, \eps} \,|\, F_{z,L} )
  \ge c_1 \, u^2 \qquad \text{ and } \qquad
  \E ( Y_{u, \eps}^2 \,|,\ F_{z,L} ) 
  \le C \, u^4. }
(ii) When $d = 3$, we have 
\eqnst
{ \E ( Y_{u, \eps} \,|\, F_{z,L} ) 
  \ge c_1 \, u^3 \qquad \text{ and } \qquad
  \E ( Y_{u, \eps}^2 \,|\, F_{z,L} ) 
  \le C \, u^6. }
(iii) When $d = 4$, we have
\eqnst
{ \E ( Y_{u,\eps} \,|\, F_{z,L} ) 
  \ge c_1 \, u^4 \, (\log u)^{-1} \qquad \text{ and } \qquad
  \E ( Y_{u,\eps}^2 \,|,\, F_{z,L} )
  \le C \, u^8 \, (\log u)^{-2}. }
\end{prop}

\begin{proof}[Proof of Theorem \ref{thm:size}(i)--(iii); lower bounds; assuming Proposition \ref{prop:moments-Y}] \ \\
Since on the event $F_{z,L}$, we have $\pi' \subset \frT_{L,o}$, we have
$|\frT_{L,o}| \ge Y$. Hence for any $t' > 0$ we have
\eqnspl{e:t'-bound}
{ \nu_L ( |\av| \ge t' )
  &\ge (2d)^{-1} \frac{\mu_{L,o} ( z \in \frT_{L,o},\, e \not\in \frT_{L,o},\, 
     |\frT_{L,o}| \ge t' )}{\mu_{L,o} ( e \not\in \frT_{L,o} )} \\
  &\ge (2d)^{-1} \, \frac{\prob ( F_{z,L} )}{\prob_e ( \sigma_L < \xi_o )} \, 
     \prob ( Y \ge t' \,|\, F_{z,L} ). }
Let us choose $t' = (1/2) \, c_1 \, u^2$ in $d = 2$; $t' = (1/2) \, c_1 \, u^3$ in $d = 3$, and $t' = (1/2) \, c_1 \, u^4 \, (\log u)^{-1}$ in $d = 4$. The Paley-Zygmund inequality implies that 
\eqn{e:PZ}
{ \prob (Y \ge t' \,|\, F_{z,L} )
  \ge \frac{1}{4} \frac{c_1^2}{C}. }
Letting $L \to \infty$, we obtain the required lower bounds from \eqref{e:t'-bound}, \eqref{e:PZ}, \eqref{e:FzL}, Lemma \ref{lem:Gamma_z,L}, and the dimension-dependent estimates in Sections \ref{sssec:2D}--\ref{sssec:4D}. 
\end{proof}

\begin{proof}[Proof of Proposition \ref{prop:moments-Y}(i),(ii)] \ \\
The upper bounds are immediate from $Y_{u,\eps} \le |V_z(u/2)|$.

For the lower bounds, using the strong Markov property of $S_x$ at time $\xi_{\pi'}$, 
when $S_x$ is at the point $p(x)$, we have
\eqnsplst
{ \prob ( I(x) \cap J(x) \,|\, F_{z,L} )
  &= \prob_x \left( \xi_{\pi'} \le \xi_{\pi'(u')} < \sigma_{V_z(u)} \right) 
  = \prob_x \left( \xi_{\pi'(u')} < \sigma_{V_z(u)} \right). }
Let $\pi'' := \looper S_z[0,\sigma_{V_z(u)}]$, and let $\pi''(u')$ be the portion of
$\pi''$ up to its first exit from $V_z(u')$. Due to Lemma \ref{lem:seg_indep}, there exists $c_2 = c_2(\eps)$, such that the distribution of $\pi'(u')$ is bounded below by $c_2$ times the distribution of $\pi''(u')$. This implies 
\eqnspl{e:hit-pi''}
{ \prob_x \left( \xi_{\pi'(u')} < \sigma_{V_z(u)} \right)
  &\ge c_2 \prob_x \left( \xi_{\pi''(u')} < \sigma_{V_z(u)} \right). }
We lower bound the probability in the right hand side of \eqref{e:hit-pi''} separately in $d = 2, 3$.  

When $d = 2$, we have
\eqnst
{ \prob_x \left( \xi_{\pi''(u')} < \sigma_{V_z(u)} \right)
  \ge \prob_x \left( \text{$S_x$ completes a loop around $z$ before exiting 
     $V_z(u')$} \right)
  \ge c }
with some constant $c > 0$. This follows from the invariance principle;
see for example \cite[Exercise 3.4]{Lawlim}. Summing over $x \in V_z(u/2)$
yields the required lower bound.
  
When $d = 3$, we have
\eqnspl{e:intersect-diff}
{ \prob_x \left( \xi_{\pi''(u')} < \sigma_{V_z(u)} \right)
  &\ge \prob_x \left( \xi_{\pi''} < \sigma_{V_z(u)} \right)
     - \prob_x \left( \sigma_{V_z(u')} \le \xi_{\pi'' \setminus \pi''(u')} 
     < \sigma_{V_z(u)} \right). }
A result of Lyons, Peres and Schramm \cite[Lemma 1.2]{LPS03} states that for two independent copies of the same transient Markov chain, the probability for one path to
intersect the loop-erasure of the other is at least a universal constant $c_3 > 0$
times the probability that the Markov chain paths themselves intersect. Applying this to the random walks $S_z[0,\sigma_{V_z(u)}]$ and $S_x[0,\sigma_{V_z(u)}]$, we have
\eqn{e:intersect-lps-3d}
{ \prob_x \left( \xi_{\pi''} < \sigma_{V_z(u)} \right)
  \ge c_3 \, \prob \left( S_x[0,\sigma_{V_z(u)}] \cap S_z[0,\sigma_{V_z(u)}] 
    \not= \varnothing \right). } 
The right hand side in \eqref{e:intersect-lps-3d} is bounded below by a constant 
$c_3 > 0$, independent of $u$. This can be seen by arguments due to Lawler; by adapting the proof of \cite[Theorem 3.3.2]{Lawler}.

It remains to bound the negative term in \eqref{e:intersect-diff}. For this 
we write
\eqn{e:near-bdry}
{ \prob_x \left( \sigma_{V_z(u')} \le \xi_{\pi'' \setminus \pi''(u')} 
    < \sigma_{V_z(u)} \right)
  \le \prob \left( S_x[\sigma_{V_z(u')},\sigma_{V_z(u)}] \cap 
    S_z[\sigma_{V_z(u')},\sigma_{V_z(u)}] \not= \varnothing \right). }
Consider independent Brownian motions, starting at $x/u$ and $z/u$. Since the paths are continuous, and with probability $1$ they exit the cube $(z/u) + [-1,1]^3$ at different points, the invariance principle implies that the probability in the right hand side of \eqref{e:near-bdry} goes to $0$ uniformly in $u \ge (1/\eps)$, as $\eps \to 0$. 
Therefore, we can fix $\eps > 0$ such that the right hand side of \eqref{e:near-bdry}
is at most $c_3/2$, uniformly in $u$. With such a choice of $\eps$, the first 
moment is bounded below by $c \, u^3$.
\end{proof}

\subsection{Proof of the moment bounds in $d = 4$}
\label{ssec:4D-moments}

\subsubsection{Proof of the first moment bounds}
We begin with the first moment lower bound -- that is, the lower bound in Proposition \ref{prop:moments-Y}(iii). The line of reasoning leading to
\eqref{e:intersect-diff} and \eqref{e:intersect-lps-3d} above holds also for 
$d=4$, and we take these as our starting point. We next establish an 
appropriate analogue of the constant lower bound given above for 
\eqref{e:intersect-diff}. We restrict to $x \notin V_z(u/4)$ for simplicity.
\begin{lem}
  \label{lem:usual_int_boxes}
  There is a $c > 0$ such that, uniformly in $z$ and $x \in V_z(u/ 2) \setminus V_z(u/4)$,
\[\prob\left(S_x[0, \sigma_{V_z(u)}] \cap S_z[0, \sigma_{V_z(u)}] \neq \varnothing \right) \geq c / \log u\ . \]
\end{lem}
\begin{proof}
  Fix such an $x$, and consider the number of intersections
\begin{equation}
\label{eq:total_int_no}
J_x:= \sum_{k=0}^{\sigma_{V_z(u)}} \sum_{\ell=0}^{\sigma_{V_z(u)}}\mathbf{1}_{S_x(k) = S_z(\ell)}\ .  \end{equation}
Taking expectations and using Theorem \ref{thm:green_asymp} gives a  $c>0$ such that 
$\E J_x  \geq c$.

On the other hand, $\E J_x^2$ is of order at most $\log u$. By a computation similar to
\cite[Theorem 3.3.2; lower bounds]{Lawler}, we have
\begin{align*}
  \E J_x^2 
&\leq 2 \sum_{y_1, \, y_2 \in V_z(u)}\left[ G(x,y_1)G(z,y_1)G(y_1,y_2)^2 + G(x, y_1) G(z, y_2) G(y_1, y_2)^2 \right]\ .
\end{align*}
Each term above gives a contribution of order $\log u$; we discuss in detail only the first term. By summing first over $y_2$ and using Theorem \ref{thm:green_asymp}(ii) (taking the $n \rightarrow \infty$ limit in this theorem), we get an upper bound of order
\[\log u \sum_{y_1 \in V_z(u)} G(x,y_1) G(z, y_1)\ . \]
For each $y_1$, either $|x-y_1|$ or $|z-y_1|$ is at least $u/8$ since $|x - z|$ is at least~$u/4$. Thus either $G(x,y_1)$ or $G(z,y_1)$ is at most~$Cu^{-2}$.
Summing the other factor over $y_1$ gives a factor of order $u^2$, giving the required upper bound.

Using the second moment method and noting that $S_x$ and $S_z$ intersect if $J_x > 0$ completes the proof.
\end{proof}
It remains to control the negative term of \eqref{e:intersect-diff}. This will be accomplished using the following lemma:
\begin{lem}
\label{lem:no_int_bdy_4d}
  There exists a constant $C_3 > 0$ such that, uniformly in $0 < \varepsilon < 1/2$ and $u$ large (how large depends on $\varepsilon$), and uniformly in $y \in V_z(u) \setminus V_z(u')$,
  \[\prob\left(S_y[0, \sigma_{V_z(u)}] \cap S_z[0, \sigma_{V_z(u)}] \neq \varnothing\right) \leq C_3 \varepsilon / \log u\ . \]
\end{lem}

\begin{proof}[Proof of Proposition \ref{prop:moments-Y}(iii), lower bound;
assuming Lemma \ref{lem:no_int_bdy_4d}]
Note that
\begin{align}
\label{e:sup-over-y}
  \prob_x\left(\sigma_{V_z(u')} \le \xi_{\pi''} < \sigma_{V_z(u)}\right) 
  &\leq \prob \left(S_x[\sigma_{V_z(u')}, \sigma_{V_z(u)}] \cap S_z[0, \sigma_{V_z(u)}] 
     \neq \varnothing\right) \\\nonumber
  &\leq \sup_{y \in \partial V_z(u')} \prob \left(S_y[0, \sigma_{V_z(u)}] 
    \cap S_z[0, \sigma_{V_z(u)}] \neq \varnothing\right)\ .
\end{align}
The above, combined with Lemma \ref{lem:no_int_bdy_4d}, allows the choice of an appropriately small $\varepsilon$ to give a uniform lower bound of $c u^4 / \log u$ for the right-hand side of \eqref{e:intersect-diff}, completing the proof of the first moment of Prop. \ref{prop:moments-Y}(iii). 
\end{proof}

We turn to the proof of Lemma \ref{lem:no_int_bdy_4d}, 
which is an adaptation of the proof of
\cite[Theorem 3.3.2; $d = 4$ upper bound]{Lawler}. 
Let $u'' = (1+\varepsilon) u$. To avoid complications introduced by intersections near the boundary, consider the extended number of intersections
\[ J_x'
   := \sum_{k=0}^{\sigma_{V_x(u'')}} \sum_{\ell=0}^{\sigma_{V_z(u'')}}
      \mathbf{1}_{S_x(k) = S_z(\ell)}\ , \quad x \in V_z(u)\ . \]

\begin{lem}
\label{lem:exp_epsilon}
There is $C$ such that, uniformly in $\varepsilon < 1/2$ and $z$, and in $x \in V_z(u) \setminus V_z(u')$, 
we have
\[ \E J_x' \leq C \, \varepsilon\ . \]
\end{lem}
\begin{proof}
We have
\begin{align}
  \E J'_x 
  &\leq \sum_{y \in V_z(u'')} G(z,y) \, G_{V_z(u'')}(x,y) \nonumber\\
  &= \sum_{y \in V_x(u/10)} G(z,y) \, G_{V_z(u'')}(x,y) 
    + \sum_{y \in V_z(u'') \setminus V_x(u/10)} G(z,y) \, G_{V_z(u'')}(x,y)\ . \label{eq:onlynearx}
\end{align}
Let $y \in V_z(u'')$, and write $\|x-y\| = r$. 
A gambler's ruin estimate yields
  \begin{align}
    G_{V_z(u'')}(x,y) 
    &\leq \prob_x \left(\sigma_{V_x(r/2)} < \sigma_{V_z(u'')} \right) 
       \sup_{a \in V_x(r/2)} G(a,y) 
    \leq C \min\left\{\frac{\varepsilon u }{r},\, 1\right\} r^{-2}. \label{eq:gbdshell}
  \end{align}
Consider the first term of \eqref{eq:onlynearx}. 
Using $\| z - y \| \ge u/2$ and \eqref{eq:gbdshell}, this term is at most
\begin{align*}
  \frac{C}{u^2}\left[\sum_{r=1}^{\varepsilon u} \left(r^3 \cdot \frac{1}{r^2}\right) + \sum_{r = \varepsilon u}^{u/10} \frac{\varepsilon u}{r} \cdot r^3 \cdot \frac{1}{r^2} \right] \leq C \varepsilon\ ,
\end{align*}
The second term, using \eqref{eq:gbdshell} again, 
is bounded by $C \, (\eps u / u^3) \sum_{y \in V_z(u'')} G(y,z) \le C \eps$. 
\end{proof}

Recall the definition of $J_x$ from \eqref{eq:total_int_no}. 
We show that, conditional on $\{J_x > 0\}$, the expectation of $J_x'$ is at least 
$c \log u$. This gives the desired upper bound for $\prob(J_x > 0)$. 
\begin{lem}
  \label{lem:J_large}
We can find $r > 0$ such that, uniformly in $0 < \varepsilon < 1/2$, $u$ such that $\varepsilon u > u^{1/2}$, and $x \in \partial V_z(u')$ such that $x$ is at 
least distance $u/10$ from all but one face of $V_z(u)$, we have
\begin{equation}
\label{eq:K_close_intersect}
\E[J_x' \mid J_x > 0] \geq r \log u\ . \end{equation}
\end{lem}

\begin{proof}
We follow a similar argument to the proofs of \cite[Proposition 10.1.1]{Lawlim} and \cite[Theorem 3.3.2; $d = 4$ upper bound]{Lawler}. On $\{J_x > 0\}$, there is a lexicographically first intersection in $V_z(u)$. Specifically, we can define $\ell_1:= \inf\{j: S_x(j) \in S_z[0, \sigma_{V_z(u)}] \cap V_z(u)\}$ and $\ell_2:= \inf\{j: S_z(j) = S_x(\ell_1)\}$ and note that each $\ell_i$ is smaller than $\sigma_{V_z(u)}$. Using the strong Markov property of $S_x$ at time $\ell_1$, conditionally on $S_z[0, \sigma_{V_z(u)}]$ and $S_x[0, \ell_1]$  
the expected value of $J_x'$ is bounded below by
\begin{align*} 
\E\left[J_x' \mid S_z[0, \sigma_{V_z(u)}], S_x[0, \ell_1]\right] 
&\ge \sum_{i=\ell_2}^{\sigma^{S_z}_{V_z(u'')}} G_{V_z(u'')}(S_z(\ell_2), S_z(i) ) \\
&\geq c \sum_{i=0}^{\sigma_{V_z(\sqrt{u})}} G(z, S_z(i))
\stackrel{\mathrm{d}}{=} c \sum_{i=0}^{\sigma_{V(\sqrt{u})}} G(o, S_o(i)) \ .
 \end{align*}
Here we have used the fact that $G_{V_z(u'')}(a, S_z(\ell_2)) \geq c G(a, S_z(\ell_2))$ for $a \in V_{S_z(\ell_2)}(u^{1/2})$ along with translation invariance, and $\stackrel{\mathrm{d}}{=}$ denotes equality in distribution.

The conclusion of Lemma \ref{lem:J_large} follows immediately from the above, using the following proposition (along with an a priori power law lower bound for $\prob(J_x > 0)$):
\end{proof}

\begin{prop}[{\cite[Lemma 10.1.2]{Lawlim}}]
  For every $\alpha > 0$, there exist $c, r$ such that for all $n$ sufficiently large,
\[\prob\left(\sum_{j=0}^{\sigma_{n} - 1} G(o, S_o(j)) \leq r \log n \right) \leq c n^{-\alpha}\ . \]
\end{prop}


\begin{proof}[Proof of Lemma \ref{lem:no_int_bdy_4d}]
Comparing Lemma \ref{lem:J_large}, and Lemma \ref{lem:exp_epsilon}, the claim nearly follows, except that $y \in \partial V_z(u')$ in \eqref{e:sup-over-y} may be closer than distance $u/10$ to more than one face of $V_z(u)$. However, for such $y$ we can replace 
$V_z(u)$ by a larger box $V' \supset V_z(u)$, whose diameter is still of order $u$, in such a way that $y$ is at distance $\eps u$ from the boundary of $V'$, and $y$ is bounded away from the corners of $V'$. Since in $V'$ the intersection probability is larger than in $V_z(u)$, the claim follows.
\end{proof}


The above completes the proof of the lower bound in 
Proposition \ref{prop:moments-Y}(iii).

\medbreak

\subsubsection{Proof of the second moment bounds}

Here we prove the second moment bound in Proposition \ref{prop:moments-Y} (iii).  We first introduce some notation.
We will work with dyadic cubes $\prod_{i=1}^4 [ a_i 2^k, (a_i + 1) 2^k ) \cap \Z^4$, 
where $a_1, a_2, a_3, a_4 \in \Z$, and $k \ge 1$. We say that such 
a cube is of scale $k$. If $Q$ is dyadic cube, we denote by 
$Q'$, respectively, $Q''$, the cubes that are concentric with 
$Q$ and have twice, respectively, four times, the side-length.
Given $v \in \Z^4$, we denote by $Q(v;k)$ the unique dyadic cube 
of scale $k$ containing $v$.

The following inequality recasts the statement of Lemma \ref{lem:no_int_bdy_4d}.
Let $Q$ be a dyadic cube of scale $k$, $p \in Q$, $q \in Q''$, such that 
$\dist_\infty(q,\partial Q'') = \eps 2^k$. There is a constant
$C$ such that 
\eqn{e:cube-intersect}
{ \prob \left( S_q[0,\sigma_{Q''}] \cap S_p[0,\infty) \not= \es \right)
  \le \frac{C \, \eps}{k}. }
%
Fix $x, y \in B_z(u/2)$ and assume that $I(x) \cap J(x) \cap I(y) \cap J(y)$ occurs.
We distinguish the following two cases:
\begin{itemize}
\item[(I)] the paths from $x$ and $y$ to $\pi'$ do not meet;
\item[(II)] the paths from $x$ and $y$ to $\pi'$ meet at a vertex 
$q(x,y) \not\in \pi'$.
\end{itemize}
In Case (I), using that the events $J(x)$ and $J(y)$ occur, let $v, w \in \pi'(u')$, respectively, be the points where $S_{p(x)}$ and $S_{p(y)}$, respectively, first hit $\pi'(u')$. (Note that $v$ and $w$ may coincide.) It follows that 
\eqnsplst
{ \prob ( \mathrm{Case \ (I)} )
  &\le \prob \left( \exists v,w \in \pi'(u') : 
      \xi^{S_x}_v < \sigma^{S_x}_{V_z(u)},\,
      \xi^{S_y}_w < \sigma^{S_y}_{V_z(u)} \right). }
Due to Lemma \ref{lem:seg_indep}, there exists $C_1 = C_1(\eps)$, such that the distribution of $\pi'(u')$ is bounded above by $C_1$ times the distribution of
$\pi''(u')$. Therefore, we have
\eqnsplst
{ \prob ( \mathrm{Case \ (I)} )
  &\le C_1 \, \prob \left( \exists v,w \in \pi''(u') : 
      \xi^{S_x}_v < \sigma^{S_x}_{V_z(u)},\,
      \xi^{S_y}_w < \sigma^{S_y}_{V_z(u)} \right). }
In Case (II), we let $v = q(x,y)$, and noting $p(x) = p(y)$, let $w$ be the point
where $S_{p(x)}$ first hits $\pi'(u')$. Again bounding above by $\pi''(u')$, it follows that 
\eqnsplst
{ \prob ( \mathrm{Case \ (II)} )
  &\le C_1 \, \prob \left( \exists w \in \pi''(u'),\, v \in V_z(u) :  
      \xi^{S_x}_v < \xi^{S_x}_w < \sigma^{S_x}_{V_z(u)},\,
      \xi^{S_y}_v < \sigma^{S_y}_{V_z(u)} \right). } 
We bound the probabilities of Cases (I) and (II) separately. 
The idea of the bound is \emph{not to} sum over $v$ and $w$, but rather, sum over the choice of suitable dyadic cubes that $v$ and $w$ fall into, and use the bound \eqref{e:cube-intersect} for the probability of random walk intersections. Throughout, we write $K$ for the integer such that $2^{K-1} < u \le 2^K$. 


\medbreak

\textbf{Case (I).} 
We may assume without loss of generality that the walk $S_z$ generating $\pi''$ hits $v$ before $w$, as the other case follows by a symmetric argument. For convenience, we assume that $\| z - v \|$, $\| v - w \|$, $\| x - v \|$, $\| y - w \|$ are all at least $32$. At the end of the proof we comment on how to handle the remaining configurations of points.
We define the following dyadic scales and cubes:
\eqnsplst
{ k_v
  &:= \max \left\{ k \ge 1 : 2^{k+4} \le  
     \min \{ \| v - z \|_\infty,\, 
     \| v - w \|_\infty,\, \| v - x \|_\infty \} \right\} \\
  k_w
  &:= \max \left\{ k \ge 1 : 2^{k+4} \le 
     \min \{ \| w - v \|_\infty,\, 
     \| w - y \|_\infty \} \right\} \\
  k_{vw}
  &:= \max \left\{ k \ge 1 : 2^{k+4} \le 
     \| v - w \|_\infty \right\} \\
  Q(v) 
  &:= Q(v;k_v) \qquad\qquad
  Q(w) 
  = Q(w;k_w). } 
We also let 
\eqnspl{e:kx-ky-kz}
{ k_z
  &:= \max \left\{ k \ge 1 : 2^{k+4} \le
      \| z - v \|_\infty \right\} \qquad
  k_x
  := \max \left\{ k \ge 1 : 2^{k+4} \le 
      \| x - v \|_\infty \right\} \\
  k_y
  &:= \max \left\{ k \ge 1 : 2^{k+4} \le 
     \| y - w \|_\infty \right\}. }
A sketch of the argument is as follows: the walks $S_z$ and $S_x$ both have to 
hit $Q'(v)$, and then they intersect at a point of $Q(v)$. Following 
the intersection, the walk $S_z$ has to hit $Q'(w)$, and so does the 
walk $S_y$. These two walks then intersect at a point of $Q(w)$. 
Breaking up the paths into pieces, the various hitting and intersection 
events will give us the estimate:
\eqn{e:main-prob-bound}
{ C \, \frac{\left( 2^{k_v} \right)^2}{\left( 2^{k_z} \right)^2} \,
  \frac{\left( 2^{k_v} \right)^2}{\left( 2^{k_x} \right)^2} \,
  \frac{1}{\log 2^{k_v}} \,
  \frac{\left( 2^{k_w} \right)^2}{\| w - v \|_\infty^2} \,
  \frac{\left( 2^{k_w} \right)^2}{\left( 2^{k_y} \right)^2} \,
  \frac{1}{\log 2^{k_w}}. }
We need to sum this estimate over the choices of $x$ and $y$, and the 
choices of the boxes $Q(v)$ and $Q(w)$. In the summation we will need
to distinguish a number of sub-cases according to the relative 
sizes of the scales $k_v, k_w, k_z, k_x, k_y$. 

We first establish the bound in \eqref{e:main-prob-bound}. This is
provided by the following lemma.

\begin{lem}
\label{lem:main-prob-bound}
\textbf{(Probability bound for Case (I))}
Let $R_1$ and $R_2$ be dyadic boxes of scales $k_1$ and $k_2$,
and let $x, y \in B_z(u/2)$ be points such that:\\
(i) $R_1''$ and $R_2''$ are disjoint;\\
(ii) $\dist(z,R_1''), \dist(x,R_1'') \ge 2^{k_1}$;\\
(iii) $\dist(y,R_2'') \ge 2^{k_2}$.\\
Define $k_x', k_y', k_z'$ by the formulas \eqref{e:kx-ky-kz} where
$Q(v)$ and $Q(w)$ are replaced by $R_1$ and $R_2$, respectively.
Then 
\eqnspl{e:estimate-Ia}
{ &\prob \left[ \text{$\exists v \in R_1$,
    $\exists w \in R_2$ s.t.~Case (I)} \right] 
  \le C \, \frac{\left( 2^{k_1} \right)^2}{\left( 2^{k_z'} \right)^2} \,
  \frac{\left( 2^{k_1} \right)^2}{\left( 2^{k_x'} \right)^2} \,
  \frac{1}{k_1} \,
  \frac{\left( 2^{k_2} \right)^2}{\dist_\infty(R_1,R_2)^2} \,
  \frac{\left( 2^{k_2} \right)^2}{\left( 2^{k_y'} \right)^2} \,
  \frac{1}{k_2}. }
\end{lem}

\begin{proof}
We need to be careful about the event when the walk $S_z$ first 
hits $R_1'$, leaves $R_1''$ and returns, before intersecting the path of $S_x$.
The following definitions take care of this possibility by
introducing the variables $\ell_1$ and $\ell_2$ that count 
crossings from $\partial R_1''$ to $\partial R_1'$ and
from $\partial R_2''$ to $\partial R_2'$, respectively.
The definitions are somewhat tedious to write down; however,
estimating the resulting probabilities is then straightforward
using the strong Markov property. Given $\ell_1, \ell_2 \ge 0$, let 
\eqnsplst
{ T_{\ell_1}
  &= \inf \{ n \ge \xi^{S_z}_{R_1'} : \text{$S_z[\xi_{R_1'},n]$
    has made at least $\ell_1$ crossings from $\partial R_1''$ to $R_1'$} \} \\
  \sigma_{\ell_1,R_1''}
  &= \inf \{ n \ge T_{\ell_1} : S_z(n) \not\in R_1'' \} \\
  \xi_{\ell_1,R_2'}
  &= \inf \{ n \ge \sigma_{\ell_1,R_1''} : S_z \in R_2' \} \\
  T_{\ell_1,\ell_2}
  &= \inf \{ n \ge \xi_{\ell_1,R_2'} : \text{$S_z[\xi_{\ell_1,R_2'},n]$ 
    has made at least $\ell_2$ crossings from $\partial R_2''$ to $R_2'$} \} \\
  \sigma_{\ell_1,\ell_2,R_2''}
  &= \inf \{ n \ge T_{\ell_1,\ell_2} : S_z(n) \not\in R_2'' \}. }
On the event in the left hand side of \eqref{e:estimate-Ia}, 
the following events occur for some integers $\ell_1, \ell_2 \ge 0$:
\begin{align*}
&\text{(i)} & &\xi^{S_z}_{R_1'} < \infty & &&
&\text{(v)} & &\xi^{S^z}_{\ell_1,R_2'} < \infty \\
&\text{(ii)} & &\xi^{S_x}_{R_1} < \infty & &&
&\text{(vi)} & &\xi^{S_y}_{R_2} < \infty \\
&\text{(iii)} & &T_{\ell_1} < \infty & &&
&\text{(vii)} & &T_{\ell_1,\ell_2} < \infty \\
&\text{(iv)} & &S_z[T_{\ell_1},\sigma_{\ell_1,R_1''}] 
  \cap S_x[\xi_{R_1},\infty) \not= \es & &&
&\text{(viii)} & &S_z[T_{\ell_1,\ell_2},\sigma_{\ell_1,\ell_2,R_2''}] \cap 
  S_y[\xi_{R_2},\infty) \not= \es 
\end{align*}
We bound the probability that (i)--(viii) occur, 
with each estimate conditional on the previous ones.
The probability of (i)--(ii) is bounded by 
$C \, (2^{k_1}/2^{k_z'})^2 \, (2^{k_1}/2^{k_x'})^2$, 
since $d = 4$. Using the strong Markov property of $S_z$ at times 
$\xi^{R_1'}, T_1, \dots, T_{\ell_1-1}$, we have that (iii) occurs 
with conditional probability $\le c_1^{\ell_1}$ with some $0 < c_1 < 1$.
Since $S_z(T_{\ell_1}) \in \partial R_1'$ and $S_x(\xi_{R_1}) \in \partial R_1''$
are at distance of order $2^{k_1}$ from each other, the conditional 
probability of (iv) is bounded by $C/(\log 2^{k_1}) = C' / k_1$.
The probability of (v)-(vi) is bounded by 
$C \, (2^{k_2}/\dist_\infty(R_1,R_2))^2 \, (2^{k_2}/2^{k_y'})^2$.
The probability of (vii) is bounded by $c_1^{\ell_2}$.
Finally, the probability of (viii) is bounded by $C/k_2$, 
again due to \eqref{e:cube-intersect}. 
Multiplying the bounds and summing over $1 \le \ell_1, \ell_2 < \infty$ 
yields the lemma.
\end{proof}

We continue with the bound for Case (I). We break up Case (I) into the following sub-cases (that partially overlap, but together cover all possibilities):
\begin{align*}
&\text{(I-1)} & & \text{$k_v < k_z, k_x$ and $k_w < k_y$;} & &&
&\text{(I-4)} & & \text{$k_v = k_z \le k_x$ and $k_w = k_y$;} \\
&\text{(I-2)} & & \text{$k_v < k_z, k_x$ and $k_w = k_y$;} & &&
&\text{(I-5)} & & \text{$k_v = k_x \le k_z$ and $k_w < k_y$;} \\
&\text{(I-3)} & & \text{$k_v = k_z \le k_x$ and $k_w < k_y$;} & &&
&\text{(I-6)} & & \text{$k_v = k_x \le k_z$ and $k_w = k_y$.} 
\end{align*}
Fixing the scales $k_v, k_w, k_z, k_x, k_y$, we bound the number 
of choices of $x, y$ and the dyadic boxes containing $v$ and $w$
in each case separately, and apply Lemma \ref{lem:main-prob-bound}.
Then we sum over the scales allowed in each sub-case. A depiction of case (I-1) may be found in Figure \ref{fig:fig_case1}.

\medbreak

\emph{Sub-case (I-1).}
The number of choices for $Q(v)$ is of order $2^{4 k_z} / 2^{4 k_v}$. The 
number of choices for $x$ is of order $2^{4 k_x}$. Given $Q(v)$,
the number of choices for $Q(w)$ is $O(1)$ (note that 
$k_w = k_v$), and the number of
choices for $y$ is of order $2^{4 k_y}$. Mutiplying these
bound together, and applying Lemma \ref{lem:main-prob-bound}
to Sub-case (I-1), we get the estimate:
\eqnst
{ \frac{2^{4 k_z}}{2^{4 k_v}} \, 2^{4 k_x} \, 2^{4 k_y}
  \frac{2^{2 k_v}}{2^{2 k_z}} \, \frac{2^{2 k_v}}{2^{2 k_x}} \,
  \frac{1}{k_v} \, \frac{2^{2 k_v}}{2^{2 k_v}} \,
  \frac{2^{2 k_v}}{2^{2 k_y}} \, \frac{1}{k_v}
  = 2^{2 k_z} \, 2^{2 k_x} \, 2^{2 k_y} \, \frac{2^{2 k_v}}{k_v^2}. }
Summing this bound for fixed $k_v$ over $k_x, k_y, k_z$ such that 
$k_v < k_x, k_y, k_z \le K$, and then over $1 \le k_v \le K$, we get 
\eqnst
{ \sum_{k_v = 1}^K \left( 2^{2K} \right)^3 \frac{2^{2 k_v}}{k_v^2}
  \le C \frac{\left( 2^{K} \right)^8}{K^2}
  = C \frac{u^8}{(\log u)^2}. } 


\begin{figure}
\setlength{\unitlength}{1.2cm}
\begin{minipage}[c]{0.4\textwidth}
\begin{picture}(5,5)(-2,-2)
  \linethickness{0.3mm}
  
  \put(-2,-2){\line(1,0){5}}
  \put(3,-2){\line(0,1){5}}
  \put(3,3){\line(-1,0){5}}
  \put(-2,3){\line(0,-1){5}}

  \linethickness{0.1mm}
  \setlength{\fboxsep}{0pt}
  \put(0.5,0.5){\colorbox[gray]{0}{\makebox(0.1,0.1){}}}
  \put(0.4,0.2){\large{$z$}}
  \put(1.95,0.5){\circle*{0.1}}
  \put(2.0,0.55){\large{$v$}}
  \put(1.95,-0.25){\circle*{0.1}}
  \put(2.0,-0.55){\large{$w$}}
  \put(1.95,1.8){\colorbox[gray]{0}{\makebox(0.1,0.1){}}}
  \put(1.95,2.1){\large{$x$}}
  \put(0.6,-1.4){\colorbox[gray]{0}{\makebox(0.1,0.1){}}}
  \put(0.5,-1.7){\large{$y$}}
  \linethickness{0.2mm}
  \put(0.55,0.55){\line(140,-5){1.4}}
  \multiput(2.00,1.85)(-0.0025,-0.0675){20}{\line(-5,-135){0.00125}}
  \put(1.95,0.5){\line(0,-1){0.75}}
  \put(1.95,-0.25){\line(105,-20){1.05}}
  \multiput(0.65,-1.35)(0.05,0.0423){26}{\line(130,110){0.025}} 
  \linethickness{0.1mm} 
  \put(1.7,0.25){\framebox(0.5,0.5){}}
  \put(1.6,0.15){\dashbox{0.05}(0.7,0.7){}}
  \put(1.5,0.05){\framebox(0.9,0.9){}}
\end{picture}
\end{minipage}\hfill
\begin{minipage}[c]{0.6\textwidth}
\caption{Depiction of sub-case (I-1). The walk $S_z$ is depicted 
as a solid line; walks from $x$ and $y$ are dashed. $S_x$ and $S_z$ 
intersect at $v$, which is surrounded by boxes $Q(v)$, $Q'(v)$, and 
$Q''(v)$. Walks $S_z$ and $S_y$ intersect at $w$ (boxes not shown). 
Note that in this case, the factor limiting the size of $Q(v)$ is 
the proximity of the vertex $w$. After intersection, $S_z$ terminates 
upon intersecting the sink $s$ (i.e., the boundary of the large 
square).}\label{fig:fig_case1}
\end{minipage}
\end{figure}
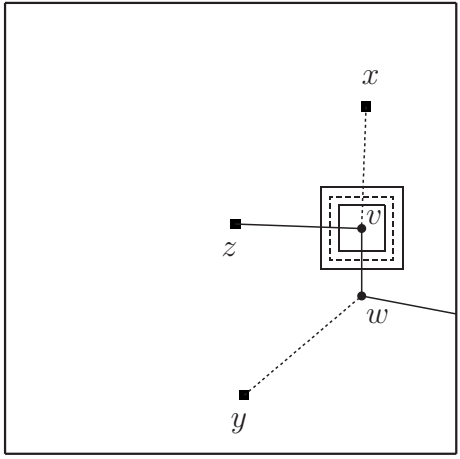

\medbreak

\emph{Sub-case (I-2).}
The number of choices for $Q(v)$ is of order $2^{4 k_z} / 2^{4 k_v}$. The 
number of choices for $x$ is of order $2^{4 k_x}$. Given $Q(v)$,
the number of choices for $Q(w)$ is of order $2^{4 k_v}/2^{4 k_w}$, 
and the number of choices for $y$ is of order $2^{4 k_w}$. 
Lemma \ref{lem:main-prob-bound} now gives:
\eqnst
{ \frac{2^{4 k_z}}{2^{4 k_v}} \, 2^{4 k_x} \, \frac{2^{4 k_v}}{2^{4 k_w}} \,
  {2^{4 k_w}} \, 
  \frac{2^{2 k_v}}{2^{2 k_z}} \, \frac{2^{2 k_v}}{2^{2 k_x}} \,
  \frac{1}{k_v} \, \frac{2^{2 k_w}}{2^{2 k_v}} \, \frac{1}{k_w}
  = 2^{2 k_z} \, 2^{2 k_x} \, \frac{2^{2 k_v}}{k_v} \frac{2^{2 k_w}}{k_w}. }
We sum over $k_v < k_x, k_z \le K$ for fixed $1 \le k_w \le k_v$, then over
$1 \le k_w \le k_v \le K$. This yields $C u^8 / (\log u)^2$
and completes the estimate in Sub-case (I-2).

\medbreak

The other four sub-cases are handled similarly, and we only state the 
number of choices, the probability estimate, and the range of summations
over the scales.

\medbreak

\noindent
\begin{tabular}{|c|c|c|c|}
\hline
 & \rule{0pt}{2.7ex}\rule[-1ex]{0pt}{0pt} Choices & Probability & Summed over \\ \hline
I-3    \tst \bst & $2^{4 k_x} \, 2^{4 k_y}$ & 
$\dfrac{2^{2 k_v}}{2^{2 k_x}} \, \dfrac{1}{k_v} \,
    \dfrac{2^{2 k_w}}{2^{2 k_y}} \, \dfrac{1}{k_w}$ &
\parbox{7.85cm}{$k_x$ such that $k_x \ge k_v$; $k_y$ such that $k_y > k_w$; 
and $k_v, k_w$ such that $1 \le k_v \le k_w \le K$} \\  \hline

I-4    \tst \bst  & $2^{4 k_x} \, \dfrac{2^{4 k_{vw}}}{2^{4 k_w}} \, 2^{4 k_w}$ &
$\dfrac{2^{2 k_v}}{2^{2 k_x}} \, \dfrac{1}{k_v} \,
    \dfrac{2^{2 k_w}}{2^{2 k_{vw}}} \, \dfrac{1}{k_w}$ &
\parbox{7.85cm}{$k_x$ such that $k_x \ge k_v$; $k_{vw}$ such that $k_{vw} \ge k_w$;
and $k_v, k_w$ such that $1 \le k_v, k_w \le K$} \\ \hline

I-5   \tst \bst   & $\dfrac{2^{4 k_z}}{2^{4 k_v}} \, 2^{4 k_v} \, 2^{4 k_y}$ &
$\dfrac{2^{2 k_v}}{2^{2 k_z}} \, \dfrac{1}{k_v} \,
    \, \dfrac{2^{2 k_w}}{2^{2 k_y}} \, \dfrac{1}{k_w}$ &
\parbox{7.85cm}{$k_z$ such that $k_z \ge k_v$; $k_y$ such that $k_y > k_w$; 
and $k_v,k_w$ such that $1 \le k_v \le k_w \le K$} \\ \hline

I-6     \tst \bst  & $\dfrac{2^{4 k_z}}{2^{4 k_v}} \, 2^{4 k_v} \, \dfrac{2^{4 k_{vw}}}{2^{4 k_w}} \, 2^{4 k_w}$ & 
$\dfrac{2^{2 k_v}}{2^{2 k_z}} \, \dfrac{1}{k_v} \,
  \dfrac{2^{2 k_w}}{2^{2 k_{vw}}} \, \dfrac{1}{k_w}$ &
\parbox{7.85cm}{$k_z$ such that $k_z \ge k_v$; $k_{vw}$ such that $k_{vw} \ge k_w$; 
and $k_v, k_w$ such that $1 \le k_v, k_w \le K$} \\ \hline
\end{tabular}

\medbreak

\medbreak

\textbf{Case (II).} 
We will use notation similar to Case (I), but with somewhat different meaning.
Let 
\eqnsplst
{ k_v
  &:= \max \left\{ k \ge 1 : 2^{k+4} \le
     \min \{ \| v - x \|_\infty,\, 
     \| v - y \|_\infty,\, \| v - w \|_\infty \} \right\} \\
  k_w
  &:= \max \left\{ k \ge 1 : 2^{k+4} \le 
     \min \{ \| w - v \|_\infty,\, 
     \| w - z \|_\infty \} \right\} \\
  k_{vw}
  &:= \max \left\{ k \ge 1 : 2^{k+4} \le 
     \| v - w \|_\infty \right\} \\
  Q(v)
  &:= Q(v;k_v) \qquad\qquad 
  Q(w)
  = Q(w;k_w). } 
We also let 
\eqnspl{e:kx-ky-kz-II}
{ k_z
  &:= \max \left\{ k \ge 1 : 2^{k+4} \le
      \| z - w \|_\infty \right\} \\
  k_x
  &:= \max \left\{ k \ge 1 : 2^{k+4} \le 
      \| x - v \|_\infty \right\} \\
  k_y
  &:= \max \left\{ k \ge 1 : 2^{k+4} \le 
     \| y - v \|_\infty \right\}. }
The following lemma provides the probability bound in Case (II), and is proved
similarly to Lemma \ref{lem:main-prob-bound}.

\begin{lem}
\label{lem:main-prob-bound-II}
\textbf{(Probability bound for Case (II))}
Let $R_1$ and $R_2$ be dyadic boxes of scales $k_1$ and $k_2$,
and let $x, y \in B_z(u/2)$ be points such that:\\
(i) $R_1''$ and $R_2''$ are disjoint;\\
(ii) $\dist(x,R_1''), \dist(y,R_1'') \ge 2^{k_1}$;\\
(iii) $\dist(z,R_2'') \ge 2^{k_2}$.\\
Define $k_x', k_y', k_z'$ by the formulas \eqref{e:kx-ky-kz-II} where
$Q(v)$ and $Q(w)$ are replaced by $R_1$ and $R_2$, respectively.
Then 
\eqnspl{e:estimate-II}
{ &\prob \left[ \text{$\exists v \in R_1$,
    $\exists w \in R_2$ s.t.~Case (II)} \right] 
  \le C \, \frac{\left( 2^{k_1} \right)^2}{\left( 2^{k_x'} \right)^2} \,
  \frac{\left( 2^{k_1} \right)^2}{\left( 2^{k_y'} \right)^2} \,
  \frac{1}{k_1} \,
  \frac{\left( 2^{k_2} \right)^2}{\dist_\infty(R_1,R_2)^2} \,
  \frac{\left( 2^{k_2} \right)^2}{\left( 2^{k_z'} \right)^2} \,
  \frac{1}{k_2}. }
\end{lem}

We now list the sub-cases to be considered. These are: 
\begin{align*}
&\text{(II-1)} & & \text{$k_v < k_x, k_y$ and $k_w < k_z$;} & &&
&\text{(II-3)} & & \text{$k_v = k_x \le k_y$ and $k_w < k_z$;} \\
&\text{(II-2)} & & \text{$k_v < k_x, k_y$ and $k_w = k_z$;} & &&
&\text{(II-4)} & & \text{$k_v = k_x \le k_y$ and $k_w = k_z$.} 
\end{align*}
Interchanging the roles of $x$ and $y$ in (II-3) and (II-4) yields 
the remaining configurations not covered.

\medbreak

\noindent
\begin{tabular}{|c|c|c|c|}
\hline
 &\rule{0pt}{2.7ex}\rule[-1ex]{0pt}{0pt} Choices & Probability & Summed over \\ \hline
 
II-1 & \tst \bst$\dfrac{2^{4 k_z}}{2^{4 k_v}} \, 2^{4 k_x} \, 2^{4 k_y}$ &
$\dfrac{2^{2 k_v}}{2^{2 k_x}} \, \dfrac{2^{2 k_v}}{2^{2 k_y}} \,
    \dfrac{1}{k_v} \, \dfrac{2^{2 k_v}}{2^{2 k_z}} \, \dfrac{1}{k_v}$ &
\parbox{7cm}{$k_x, k_y, k_z > k_v$ for fixed $k_v$; \\ 
and then over $1 \le k_v \le K$} \\ \hline

II-2 \tst \bst& $2^{4 k_x} \, 2^{4 k_y}$ &
$\dfrac{2^{2 k_v}}{2^{2 k_x}} \, \dfrac{2^{2 k_v}}{2^{2 k_y}} \,
    \dfrac{1}{k_v} \, \dfrac{2^{2 k_w}}{2^{2 k_v}} \, \dfrac{1}{k_w}$ &
\parbox{7cm}{$k_x, k_y > k_v$ for fixed $k_v, k_w$; \\ 
and then over $1 \le k_w \le k_v \le K$} \\ \hline

II-3\tst \bst & $\dfrac{2^{4 k_z}}{2^{4 k_w}} \, \dfrac{2^{4 k_w}}{2^{4 k_v}} \,
    2^{4 k_v} \, 2^{4 k_y}$ &
$\dfrac{2^{2 k_v}}{2^{2 k_y}} \, \dfrac{1}{k_v} \, 
    \dfrac{2^{2 k_w}}{2^{2 k_z}} \, \dfrac{1}{k_w}$ &
\parbox{7cm}{$k_y > k_v$ and $k_z > k_w$ for fixed $k_v, k_w$; \\
and then over $1 \le k_v \le k_w \le K$} \\ \hline

II-4 \tst\bst & $\dfrac{2^{4 k_{vw}}}{2^{4 k_v}} \, 2^{4 k_v} \, 2^{4 k_y}$ &
$\dfrac{2^{2 k_v}}{2^{2 k_y}} \, \dfrac{1}{k_v} \, 
    \dfrac{2^{2 k_w}}{2^{2 k_{vw}}} \, \dfrac{1}{k_w}$ &
\parbox{7cm}{$k_y > k_v$ and $k_{vw} > k_w$ for fixed $k_v, k_w$; \\
and then over $1 \le k_v,\, k_w \le K$} \\ \hline

\end{tabular}

\medbreak

This completes the analysis of Case (II).

\medbreak

It remains to comment on configurations where one of the $\ell_\infty$ distances 
is $< 32$. In these cases, we can replace the box $Q(v)$ and/or $Q(w)$ by the point $v$ and/or $w$ itself, and omit the random variables $T_{\ell_1}$, etc. 
The combinatorial bounds, as well as the probability 
bounds still hold with $k_v = 1$ and/or $k_w = 1$, and this completes the proof of the upper bound in Proposition \ref{prop:moments-Y} (iii).

\subsection{The size of the past in invariant forests}
\label{ssec:size-gen-lb}

\begin{proof}[Proof of Theorem \ref{thm:size-gen}]
Recall that $A_x(a,b) = D_x(b) \setminus D_x(a)$.
When $| \past_x \cap D_x(2r) | > t$, let $x$ send mass $1$ distributed equally among the vertices $z \in \past_x \cap A_x(r,(3/2)r)$. 
Then $\mathbb{E} [ \sentm(o) ] \le \mu \big( |\past_o| > t \big)$. On the other hand, using Jensen's inequality, we have
\eqnspl{e:ratio2}
{ \mathbb{E}[\getm(o)]
  &\ge \sum_{x \in A_o(r,(3/2)r)} 
    \mathbb{E}\Bigg[ 
    \frac{\mathbf{1}_{\{ o \in \past_x \}} 
    \mathbf{1}_{|\past_x \cap D_x(2r)| > t}}{|\past_x \cap A_x(r,(3/2)r)|} \Bigg]  
  \ge \sum_{x \in A_o(r,(3/2)r)} 
    \mathbb{E}\Bigg[ 
    \frac{\mathbf{1}_{\{ o \in \past_x \}} 
    \mathbf{1}_{|\tfrT_o(r/2)| > t}}{|\frT_o(4r)|} \Bigg] \\
  &\ge \sum_{x \in A_o(r,(3/2)r)} \frac{\mu \big( o \in \past_x,\, 
   |\tfrT_o(r/2)| > t \big)}
   {\mathbb{E} \Big[ |\frT_o(4r)| \,\Big|\, 
    o \in \past_x,\, |\tfrT_o(r/2)| > t \Big]} \\
   &\ge \sum_{x \in A_o(r,(3/2)r)} \frac{\mu \big( o \in \past_x,\, 
   |\tfrT_o(r/2)| > t \big)^2}
    {\mathbb{E} \Big[ |\frT_o(4r)| \, \mathbf{1}_{\{ o \in \past_x \}} \Big]}. }
This completes the proof.
\end{proof}

%

\subsection{Lower bounds: $d\geq 5$}
\label{ssec:size-lbd-highD}


We now apply Theorem \ref{thm:size-gen} to the measure $\wsf$.

\begin{proof}[Proof of Theorem \ref{thm:size}(iv); lower bound]
Let $t = \delta \, r^4$, where $\delta > 0$ will be chosen in course of the 
proof. Wilson's algorithm yields:
\eqnspl{e:cond-moment-bnd}
{ \mathbb{E} \left[ |\frT_o(4r)| \mathbf{1}_{\{ o \in \past_x \}} \right] 
  &\le \sum_{y \in B_o(4r)} \sum_{u \in \Z^d}
    \big[ \wsf ( o \in \past_u,\, u \in \past_x,\, y \in \past_u ) \\
  &\qquad + \wsf ( o \in \past_x,\, x \in \past_u,\, y \in \past_u ) \big] \\
  &\le \sum_{y \in B_o(4r)} \sum_{u \in \Z^d}
    \big[ G(o,u) \, G(u,x) \, G(y,u) 
    + G(o,x) \, G(x,u) \, G(y,u) \big] \\
  &\le C r^{6-d}. }
Write $\future_x = \{ y \in \Z^d : x \in \past_y \}$. 
Using Cauchy-Schwarz we have
\eqnspl{e:CauchySchwarz}
{ &\sum_{x \in A_o(r,(3/2)r)} \wsf( o \in \past_x,\, |\tfrT_o(r/2)| > t )^2 \\
  &\qquad \ge c r^{-d} \, \left( \sum_{x \in A_o(r,(3/2)r)} 
     \wsf( o \in \past_x,\, |\tfrT_o(r/2)| > t ) \right)^2 \\
  &\qquad = c \, r^{-d} \, \mathbb{E} \left[ | \future_o 
    \cap A_o(r,(3/2)r) | \, \mathbf{1}_{|\tfrT_0(r/2)| > t} \right]^2. }
We have $\mathbb{E} \big[ |\tfrT_o(r/2)| \big] \ge c \, r^4$, due to a result
of Pemantle \cite[Lemma 3.1]{pemantle}, and using Wilson's algorithm, we have
\eqnst
{ \mathbb{E} \big[ |\tfrT_o(r/2)|^2 \big] 
  \le \sum_{u, y, w \in B_o(r/2)} G(o,w) \, G(w,u) \, G(w,y)
  \le C \, r^8. } 
This yields $\wsf ( |\tfrT_o(r/2)| > t ) \ge c_1 > 0$ for some $c_1 > 0$ and 
sufficiently small $\delta$. Fix such $\delta$. We get a lower bound on 
$| \mathrm{future}_o \cap A_o(r,(3/2)r) |$ by considering the number of loop-free
points of the random walk $S_o$ generating the path from $o$ to $\infty$.
A result of Lawler says that with probability $1$, the fraction of loop-free points
is asymptotically a positive constant in $d \ge 5$; see \cite[Section 7.7]{Lawler}.
Therefore, we can find $\eps > 0$ small enough, such that 
$\wsf ( |\mathrm{future}_o \cap A_o(r,(3/2)r)| \ge \eps r^2 ) \ge 1 - c_1/2$.
With these choices of $\delta$ and $\eps$, the right hand side 
of \eqref{e:CauchySchwarz} is at least 
$r^{-d} \, ((c_1/2) \, \eps \, r^2)^2 = c_2 \, r^{4-d}$. 
Substituting this and \eqref{e:cond-moment-bnd} into the bound given by
Theorem \ref{thm:size-gen} we obtain
$\wsf \big( |\past_o| \ge t \big) \ge c \, r^{-2} = c' t^{-1/2}$.
As in Section \ref{sec:lowerbd}, we have
\begin{equation*}
  \nu(\vert \av \vert>t)
  \geq \frac{1}{2d} \, \wsf ( \vert \text{past}_o \vert > t),
\end{equation*}
and the claim of the theorem follows.
\end{proof}

\subsection{Upper bounds: $d\geq 5$}
\label{ssec:size-ubd-highD}

Unlike the case of the radius, an upper bound on the size of waves does not yield 
directly an upper bound on the size of the avalanche. However, we can still get
a power law upper bound, that we prove in this section.

\begin{proof}[Proof of theorem \ref{thm:size}(iv); upper bound]
Fix some $k\geq 1$. Recalling that $N$ denotes the number of waves, we have
\begin{align}
  \nu( S > t )
  = \nu( S > t,\, N > k ) + \nu( S > t,\, N \leq k ).
\end{align}
From Lemma \ref{lem:ptwise} we have $N \leq R$, and 
we can upper bound the first term in the expression above by $\nu( R > k).$ Next, 
writing $S^j$ for the size of the $j$-th wave, if $S > t$ and $N \le k$, then we
have $S^j > t/k$ for some $1 \le j \le N$. 
Hence, 
\begin{align}\notag \label{size}
 \nu( S > t)
 &\leq \nu( R > k) + \nu( S^j > t/k, \text{ for some } j \leq N) \\ \notag
 &\leq C \, \wsf_o( \diam(\T_o) > k ) + C \, \wsf_o( |\T_o| > t/k )\\ 
 &\leq k^{-2 + o(1)} + C \, \wsf_o( |\T_o| > t/k ),
\end{align}
where we used Theorem \ref{thm:radius}(iv) for the first term.
We now control the second term on the right in equation \eqref{size}.
\begin{lem}\label{volcompo}
$\wsf_o(\vert \T_o\vert > t) \leq t^{-1/2+o(1)}.$
\end{lem}

\begin{proof}
We note that 
an application of Theorem \ref{thm:radius}(iv) yields
\begin{align}\label{eq wsfvolo}\notag
 \wsf_o(\vert \frT_o \vert > t)
 &\leq \wsf_o(\text{diam}(\frT_o) > t^{1/4} )
   + \wsf_o(\vert B(t^{1/4}) \cap \frT_o \vert > t) \\
 &\leq (t^{1/4})^{-2 + o(1)} 
   + \wsf_o(\vert B(t^{1/4}) \cap \frT_o \vert > t).
\end{align}
Due to Wilson's algorithm, we have $\wsfo ( x \in \frT_o ) \le C \, |x|^{2-d}$, and hence
we have
\begin{align*}
  \wsf_o( \vert B(t^{1/4}) \cap \frT_o \vert > t)
  \leq \frac{\E_{\wsfo} \vert B(t^{1/4}) \cap \frT_o \vert}{t} 
  \leq \frac{C \, (t^{1/4})^2}{t}
  = C \, t^{-1/2}.
\end{align*}
\end{proof}

Using the above lemma and optimizing the number of waves $k$ in (\ref{size}), we get 
$\nu( S > t) \leq t^{-2/5+o(1)}$,
thereby completing the proof of the upper bound in Theorem \ref{thm:size}(iv).
\end{proof}

\vskip 0.3in 

\acknow \  The authors thank Russ Lyons for helpful discussions. 
J.H. and A.A.J. also thank the organizers of the 2014 Bath Summer School, 
where some of this work was initiated.

J.H. thanks Michael Damron for postdoctoral advising and encouragement.

S.B. thanks Russ Lyons for support, encouragement and advice.

\bigskip
\noindent
\end{document}